\documentclass[oneside,abstract=true]{scrartcl}

\usepackage{preamble}

\title{Functor calculus via non-cubes}
\author{Robin Stoll}
\date{\today}

\begin{document}

\maketitle

\begin{abstract}
  We study versions of Goodwillie's calculus of functors for indexing diagrams other than cubes.
  We in particular construct universal excisive approximations for a larger class of diagrams, which yields an extension of the Taylor tower.
  We prove that the limit of this extension agrees with the limit of the Taylor tower using criteria for the existence of maps between excisive approximations.
  Lastly we investigate in which cases our new notions of excision coincide with classical ones.
\end{abstract}

\tableofcontents

\renewcommand*{\thetheorem}{\Alph{theorem}}

\section{Introduction}

Classical functor calculus was developed by Goodwillie in the series of papers \cite{Goo90, Goo92, Goo03} as a tool to study functors from spaces to spaces or spectra.
Since then it has turned out to be a fruitful theory that has, together with a few variations, found many applications in homotopy theory and elsewhere, e.g.\ to algebraic $K$-theory, chromatic homotopy theory, or embedding spaces of manifolds.
A survey of some of these can be found in \cite{AC}.
Moreover there are generalizations to the setting of model categories (see e.g.\ \cite{Kuh} or \cite{Per}) and to the setting of quasi-categories (see \cite[Section 6]{LurHA}).
The latter framework is the one we use in this paper.

The fundamental notion of the theory is that of an \emph{$n$-excisive} functor: a functor that sends \emph{strongly cocartesian} $n$-cubical diagrams to \emph{cartesian} ones.
Here, one possible definition of strongly cocartesian is that the diagram is a left Kan extension of its restriction to the initial star of the cube; see the following picture for the case of the 3-cube.
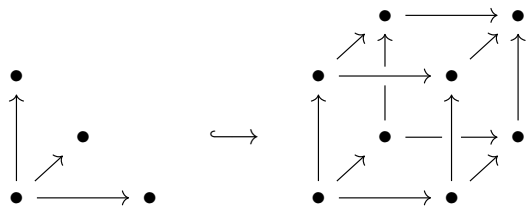
\[
\begin{tikzcd}[sep = 10]
& & &  & & \bullet \ar[leftarrow]{dd} \ar{rr} & & \bullet \\
\bullet & & &  & \bullet \ar{ur} \ar[crossing over]{rr} & & \bullet \ar{ur} & \\
& \bullet & & \longhookrightarrow & & \bullet \ar{rr} & & \bullet \ar{uu} \\
\bullet \ar{uu} \ar{ur} \ar{rr} & & \bullet &  & \bullet \ar{uu} \ar{ur} \ar{rr} & & \bullet \ar[crossing over]{uu} \ar{ur} &
\end{tikzcd}
\]

One may now wonder what is special about the cubes here: why not use other shapes of diagrams?
That is the question we investigate in this paper.
To this end we first note that the condition of being strongly cocartesian does not make any reference to the structure of the cube; it only needs the datum of the inclusion above.
In particular this allows us to generalize to arbitrary maps of posets $\sigma \colon \pos P \to \pos Q$ by defining a $\pos Q$-indexed diagram to be \emph{$\sigma$-cocartesian} if it is a left Kan extension along $\sigma$ of its restriction along $\sigma$.
If $\pos Q$ additionally has an initial object, then this yields a notion of excision:
\begin{definition*}
  A functor is \emph{$\sigma$-excisive} if it sends $\sigma$-cocartesian diagrams to cartesian diagrams.
\end{definition*}

Most of this paper is concerned with studying this notion.
The natural first question is whether there exists an analogue of the universal $n$-excisive approximation $\P[n]$ of classical Goodwillie calculus.
Indeed, our first main result states that this is the case for a class of well-behaved maps of posets, called \emph{shapes} (see \Cref{def:preshape,def:shape}), which generalize the inclusions of the classical setting:
\begin{theorem} \label{thm:intro_approx}
  Let $\sigma$ be a full shape and $F$ a functor.
  Then there is an explicit construction of its universal $\sigma$-excisive approximation $\P F$.
\end{theorem}
\noindent (See \Cref{thm:full_approximation} for the precise statement.)
Here a shape is \emph{full} if it is full as a functor.

In Goodwillie calculus there is, for any functor $F$, a sequence of maps under $F$
\[ F \longto \dots \longto \P[2] (F) \longto \P[1] (F) \longto \P[0] (F) \]
called the \emph{Taylor tower} of $F$.
In good cases this tower \emph{converges}, i.e.\ recovers $F$ in the sense that the canonical map $F \to \holim{n} \P[n] (F)$ is an equivalence (potentially after restricting to some subcategory).
Clearly one would like to have an analogue of the Taylor tower in our more general framework.
Again, this exists (though it is a bit more technical to construct) and takes the form of the \emph{Taylor graph}: a diagram that contains, for a fixed functor $F$, all of its universal excisive approximations $\P (F)$ and the maps between them induced by their universal properties.
Large parts of the paper are devoted to studying this diagram.
The second main result we prove is the following.
\begin{theorem} \label{thm:intro_Taylor}
  The limits of the Taylor tower and the Taylor graph agree when the latter is restricted to non-inane shapes between finite posets.
\end{theorem}
\noindent (Actually we even show the stronger statement that the evident functor between their indexing categories is homotopy initial; see \Cref{thm:tower_is_initial}.)
Notably this tells us that any convergence criteria for the Taylor tower can also be used for the Taylor graph.
Here, a shape $\sigma$ is \emph{inane} if it fulfills a certain combinatorial condition that implies that any functor is $\sigma$-excisive (see \cref{sec:inane}).
In particular we do not lose any information by discarding them.

To prove \cref{thm:intro_Taylor} we first pursue the naturally arising question of when there exists a map $\P (F) \to \P[\tau] (F)$ in the Taylor graph or, equivalently, when $\tau$-excisive implies $\sigma$-excisive.
It is not true that any (naively defined) morphism of shapes $f \colon \tau \to \sigma$ induces such a map; however we are able to give an explicit combinatorial condition on $f$ for this to be the case (see \cref{sec:indirect}).
This generalizes the classical fact that $n$-excisive implies $(n + 1)$-excisive.
Moreover, maybe surprisingly, we prove that there is another condition on $f$ that guarantees the existence of a map $\P[\tau] (F) \to \P (F)$, i.e.\ in the other direction (see \cref{sec:direct}).
Together, these two conditions turn out to be very useful in studying shapes and their relations (which we do in \cref{section:shapes}).
In particular, by considering shapes freely generated by a finite poset, this leads to a proof of \cref{thm:intro_Taylor}.
However the tools employed, including the two aforementioned conditions on a map of shapes, might be of independent interest as well.

Going back to our original motivation of ``Why cubes?'', \Cref{thm:intro_Taylor} suggests the following question: is the inclusion of the Taylor tower into the Taylor graph an equivalence?
Or, equivalently: given a (non-inane) shape $\sigma$, does there always exist an $n$ such that $\sigma$-excisive is equivalent to $n$-excisive?
If this were true, it would provide compelling evidence that cubes are the ``correct'' indexing diagrams to use, as they would cover all notions of excision arising from shapes.

While we are unable to completely solve this question, we do provide a partial answer.
This takes the form of the following theorem.
\begin{theorem} \label{thm:intro_cubes}
  Let $\sigma$ be a shape with codomain a cube.
  Then $\sigma$-excisive is equivalent to $n_\sigma$-excisive for a certain natural number $n_\sigma$.
\end{theorem}
\noindent (See \Cref{thm:cubes} for the precise statement.)
We remark that, when $\sigma$ is full, being a shape with codomain a cube is equivalent to being the inclusion of a (non-empty) downward closed subposet of the cube (see \cref{lemma:cube_shape}).

All evidence known to the author, including \cref{thm:intro_Taylor,thm:intro_cubes}, points towards the answer to the question asked above being affirmative.
Hence we propose the following.

\begin{conjecture*}
  Let $\sigma$ be a non-inane shape between finite posets.
  Then $\sigma$-excisive is equivalent to $n_\sigma$-excisive for some natural number $n_\sigma$.
\end{conjecture*}

Let us conclude this introduction by remarking that it would also be very interesting if this conjecture were false.
In that case the Taylor graph would be a finer resolution of the Taylor tower and could potentially contain additional useful information.

\paragraph{Structure of this paper.}

In \Cref{section:notation} we collect the conventions and notations we will use throughout the rest of this paper.
In \Cref{section:excisive} we define $\sigma$-excisive and shapes, give the construction of the universal excisive approximation, and prove \Cref{thm:intro_approx}.
In \Cref{section:maps} we give conditions for morphisms of shapes to induce maps between their universal excisive approximations.
In \Cref{section:shapes} we employ these conditions to study shapes and their relations, introducing the notions of free and inane shapes along the way.
In \Cref{sec:cubes} we study cubical shapes and prove \Cref{thm:intro_cubes}.
In \Cref{section:taylor} we construct the Taylor graph and use the results of the preceding sections to prove \Cref{thm:intro_Taylor}.
In \Cref{section:mates} we recall the calculus of mates of natural transformations.
In \Cref{section:facts} we recall a number of $\infty$-categorical facts that will be used throughout the paper (often without reference).
In \Cref{section:basics} we give references or proofs for these and other needed basic facts about (co)limits and Kan extensions.
In \Cref{section:generalitites} we prove various general facts that are needed but would hinder the flow of the main exposition.

The suggested reading order is to start with \Cref{section:mates} in the case of unfamiliarity with the calculus of mates, then read \Cref{section:notation} and afterwards quickly remind oneself of the statements in \Cref{section:facts}.
Then the reading of the main exposition in \Crefrange{section:excisive}{section:taylor} can begin, with some thumbing forwards to the statements in \Cref{section:basics} and the statements and proofs in \Cref{section:generalitites} when they are referenced.
The proofs in \Cref{section:basics} are included for completeness and only when the author was not able to find a reference; the statements being quite basic, their proofs are not a main part of this paper and reading them is not necessary for understanding the central exposition.

\paragraph{Acknowledgments.}

First and foremost I would like to thank Emanuele Dotto, who was the advisor for my masters' thesis \cite{Sto}, which forms the basis for the majority of this paper.
He suggested the topic of the thesis and guided me through the process of writing it, always being available when I had questions.

Secondly I would like to thank Tashi Walde and Greg Arone for many useful discussions, and Markus Land for providing me with a proof of a technical lemma.

Moreover I would like to thank Kevin Carlson, Denis Nardin, and Dylan Wilson for answering questions of mine on MathOverflow that occurred while working on the thesis.
I am particularly grateful to Kevin Carlson for making me aware of the calculus of mates and the work of Riehl and Verity as well as the usefulness of their interplay.

I would also like to thank Emanuele Dotto, Tashi Walde, Greg Arone, Alexander Berglund, Thomas Blom, and the anonymous referee for useful comments on earlier versions of this paper.

Last but not least I would like to thank Emily Riehl for answering a question of mine in detail, going so far as to add the relevant statement to the draft that would become \cite{RV}.

\renewcommand*{\thetheorem}{\arabic{theorem}}
\numberwithin{theorem}{section}

\section{Notation and conventions} \label{section:notation}

When working with $(\infty, 1)$-categories (which will be the case most of the time) we will use the framework of quasi-categories developed by Joyal, Lurie, and others.
In particular, when we say $\infty$-category we mean quasi-category.
However, we will often work purely in the homotopy 2-category of $\infty$-categories, thereby employing the theory developed by Riehl and Verity in a series of articles starting with \cite{RV15} and concluded with their book \cite{RV}.
Due to this, most of our arguments should not, in any fundamental way, depend on the precise model chosen for $\infty$-categories.

Moreover, in quite a few places we will employ the calculus of mates of natural transformations.
The needed facts are recalled in \Cref{section:mates}, for those unfamiliar with the theory.

In the following we state the conventions and notations we will use throughout this paper.
A few basic facts concerning these notions that we will often use without explicit mention are collected in \Cref{section:facts}.
It is recommended to quickly remind oneself of the statements after finishing this first section.

Generally, if there is a pair of dual definitions or statements, we will only give one of them and leave the other implicit.

\begin{convention}
  To avoid set-theoretic problems, we will throughout assume that there is a sufficient supply of Grothendieck universes, so that any constructions we may employ will make sense in a potentially higher universe (this is the same approach as taken by Lurie; see \cite[Section 1.2.15]{LurHTT}).
  Objects belonging to the first such universe will be called \emph{small}.
  
  We will not assume $\infty$-categories to be small, so that the examples we are interested in (such as the $\infty$-category of all (small) spaces) are actually examples.
  Consequently we will not assume simplicial sets nor categories to be small (nor locally small), so that an $\infty$-category is a simplicial set and its homotopy category a category.
  Posets, however, will be assumed to be small.
\end{convention}

\begin{convention}
  We will implicitly treat posets as categories and categories as $\infty$-categories whenever it is convenient, without a change of notation.
  In particular, we will often just write ``functor'' for an order preserving map between posets.
\end{convention}

\begin{notation}
  We write $\Pos$ for the category of posets, $\Posini$ for the subcategory of posets that have an initial object together with initial object preserving functors, and $\Poscop$ for the subcategory of posets that admit all (small) coproducts together with functors that preserve (small) coproducts.
\end{notation}

\begin{remark}
  Note that requiring a functor $f \colon \pos P \to \pos Q$ of posets to preserve $I$-indexed coproducts is equivalent to requiring equalities $f\left(\coprod_{i \in I} p_i\right) = \coprod_{i \in I} f(p_i)$.
  (For one direction we note that if the canonical map is an isomorphism, then it is already the identity; for the other that if we have the above equality, then the canonical map must be the identity since there are no other endomorphisms.)
\end{remark}

\begin{notation}
  We say a functor between categories is a \emph{homotopy equivalence} if the geometric realization of its nerve is a (weak) homotopy equivalence.
  Similarly, we say a category is \emph{contractible} if the geometric realization of its nerve is (weakly) contractible.
\end{notation}

\begin{notation}
For $\cat I$ a category we denote by $\gini{\cat I} \subseteq \cat I$ the full subcategory spanned by the non-initial objects.
\end{notation}

\begin{notation}
  We denote by $\termCat$ the terminal category and, for an $\infty$-category $\infcat C$ and $c \in \infcat C$, by $\const[c] \colon \termCat \to \infcat C$ the functor representing $c$ (sometimes we will also just write $c$ for $\const[c]$).
  More generally, for a simplicial set $K$, we denote by $\const[c] \colon K \to \infcat C$ the unique map that factors over $\const[c] \colon \termCat \to \infcat C$, and omit the index if the target category is $\termCat$.
\end{notation}

\begin{notation}
  When $\cat I$, $\cat I'$, and $\cat J$ are categories and $f \colon \cat I \to \cat J$ and $f' \colon \cat I' \to \cat J$ are functors, we write $f \slice f'$ for the comma category and denote its objects by tuples $(i, i', f(i) \to f'(i'))$ where $i \in \cat I$ and $i' \in \cat I'$.
  Furthermore, we denote by $\pr[\cat I]$ respectively $\pr[\cat I']$ (or just by $\pr$ (or $\pr[f \slice f']$) if it is clear which one is meant) the forgetful functor from $f \slice f'$ to $\cat I$ respectively $\cat I'$.
  Sometimes we will replace one or both of $f$ and $f'$ either with $\cat J$, in which case we mean the functor $\id[\cat J]$, or with an object $j \in \cat J$, in which case we mean the functor $\const[j] \colon \termCat \to \cat J$.
  In the latter case, we will, when writing an object of the comma category, omit the unique object of $\termCat$ from the tuple.
\end{notation}

\begin{remark}
  When $\cat J$ is a poset, this comma category $f \slice f'$ is canonically isomorphic to the full subcategory of $\cat I \times \cat I'$ spanned by those $(i, i')$ such that $f(i) \le f'(i')$ (since any diagram in a poset automatically commutes).
  In particular we can omit mention of the structure maps $f(i) \to f'(i')$ in this case.
\end{remark}

\begin{notation}
  For $\infcat C$ an $\infty$-category, we denote by $\hcat{\infcat C}$ its homotopy category (as a (1-)category, cf.\ \cite[Section 1.2.3]{LurHTT}) and by $\hcatmap[\infcat C] \colon \infcat C \to \hcat{\infcat C}$ the canonical functor.
\end{notation}

\begin{notation}
  For $K$ a simplicial set, we denote by $\cone{K}$ the cone over $K$, by $\conept$ the cone point, and by $\inc[K] \colon K \to \cone{K}$ the inclusion (we will sometimes drop the index if there is no risk of confusion).
\end{notation}

\begin{notation}
  For $K$ a simplicial set and $\infcat C$ an $\infty$-category, we will denote by $\Fun{K}{\infcat C}$ the $\infty$-category of functors from $K$ to $\infcat C$, i.e.\ the internal hom of simplicial sets.
  
  We will often implicitly identify $\Fun{\termCat}{\infcat C}$ with $\infcat C$ itself.
\end{notation}

\begin{notation}
  Let $f \colon I \to J$ be a map of simplicial sets and $\infcat C$ an $\infty$-category.
  We denote by $\Res{f} \colon \Fun{J}{\infcat C} \to \Fun{I}{\infcat C}$ the restriction along $f$.
\end{notation}

\begin{notation}
  By an \emph{adjunction} of functors between $\infty$-categories we will mean an adjunction in the \emph{homotopy 2-category of $\infty$-categories}, i.e.\ the strict 2-category with objects the $\infty$-categories, morphisms the functors between $\infty$-categories, and 2-morphisms the homotopy classes of natural transformations between those functors (cf.\ \cite[Definition 1.4.1]{RV}).
\end{notation}

\begin{remark}
  This is the definition of an adjunction used by Riehl and Verity (see \cite[Definition 2.1.1]{RV}).
  We chose it since it is very pleasant to work with, in particular in relation to functors of Kan extension.
  That it agrees with the more hands-on definition of Lurie given in \cite[Definition 5.2.2.1]{LurHTT} is shown in \cite[Appendix F.5]{RV}.
\end{remark}

\begin{definition}
  Let $f \colon I \to J$ be a map of simplicial sets.
  An $\infty$-category $\infcat C$ is \emph{weakly left $f$-extensible} if the restriction $\Res{f} \colon \Fun{J}{\infcat C} \to \Fun{I}{\infcat C}$ has a left adjoint.
  In this case we fix such an adjunction $\Lan{f} \dashv \Res{f}$.
  In particular, we fix a unit-counit pair of this adjunction, which will be what we mean when we write ``the'' unit (or counit) of the adjunction.
\end{definition}

\begin{remark}
  If $f = \id[I] \colon I \to I$, then $\Res{\id} = \id$.
  In particular we can choose $\Lan{\id}$, as well as the unit and counit of the adjunction, to be identities as well.
  This is the adjunction we fix in this case.
\end{remark}

\begin{notation}
  Let $I$ be a simplicial set and $\infcat C$ an $\infty$-category.
  Then we write $\Diag \colon \infcat C \to \Fun{I}{\infcat C}$ for the diagonal, i.e.\ the restriction along $\const \colon I \to \termCat$, and say that $\infcat C$ \emph{admits all colimits indexed by $I$} if it is weakly left $\const$-extensible, i.e. if the functor $\Diag$ admits a left adjoint.
  In this case we write ${\colim{I}} \defeq {\Lan{\const}} \colon \Fun{I}{\infcat C} \to \infcat C$.
  Note that in particular we have a fixed adjunction $\colim{I} \dashv \Res{\const}$.
\end{notation}

\begin{definition}
  Let $f \colon \cat I \to \cat J$ be a functor between categories.
  We say that an $\infty$-category $\infcat C$ is \emph{left $f$-extensible} if it admits colimits indexed by $f \slice j$ for all $j \in \cat J$.
\end{definition}

\begin{remark} \label{rem:functorial_ptw_ext}
  By \cite[Corollary 12.3.10]{RV}, an $\infty$-category $\infcat C$ that is left $f$-extensible is weakly left $f$-extensible.
\end{remark}

\begin{remark}
  Note that an $\infty$-category $\infcat C$ that admits colimits indexed by a category $\cat I$ is left $\inc$-extensible, where ${\inc} \colon \cat I \to \cocone{\cat I}$ is the inclusion.
  This follows by considering, for an $a \in \cocone{\cat I}$, the slice category ${\inc} \slice a$.
  If $a = \coconept$, it is isomorphic to $\cat I$.
  If otherwise $a \in \cat I$, it has a terminal object, in which case \Cref{lemma:htpy_terminal_admitting} implies that $\infcat C$ admits colimits indexed by ${\inc} \slice a$.
\end{remark}

\begin{notation}
  Let $I$ be a simplicial set, $\infcat C$ an $\infty$-category, and $p \colon I \to \infcat C$ a diagram.
  We say that a diagram $\cocone{I} \to \infcat C$ is a \emph{colimit diagram extending $p$} if it is an initial object of $\infcat{C}_{p/}$ (cf.\ \cite[Remark 1.2.13.5]{LurHTT}).
\end{notation}

\begin{notation} \label{def:structure_map}
  Let $\cat I$ be a category, $i \in \cat I$, and $\infcat C$ an $\infty$-category that admits colimits indexed by $\cat I$.
  Denote by $t_i \colon \Simplex{1} \to \cocone{\cat I}$ the functor representing the unique morphism $i \to \coconept$.
  Then the functor $\Res{t_i} \colon \Fun{\cocone{\cat I}}{\infcat C} \to \Fun{\Simplex{1}}{\infcat C}$ curries to a natural transformation $\alpha \colon \Res{i} \to \Res{\coconept}$ of functors $\Fun{\cocone{\cat I}}{\infcat C} \to \infcat C$.
  Now we can form the composition (where the first equivalence is provided by \Cref{lemma:fully_faithful_kan_unit} and the last one by \Cref{lemma:colim_is_kan_to_cocone})
  \[ \Res{i} \xlongto{\eq} \Res{i} \Res{\inc} \Lan{\inc} = \Res{i} \Lan{\inc} \xlongto{\alpha} \Res{\coconept} \Lan{\inc} \xlongfrom{\eq} \colim{\cat I} \]
  which we consider to be the \emph{structure map} from the value at $i$ to the colimit.
\end{notation}

\begin{remark} \label{rem:structure_maps}
  We could have defined the structure map as the restriction of the unit $\id \to \Diag \circ \colim{\cat I}$ to $i$, but the definition we gave is easier to compare to the notions of \cite{LurHTT}.
  In particular note that, by \Cref{lemma:cocone_comparison}, the structure map $D(i) \to \colim{\cat I} D$ is an equivalence if and only if all colimit diagrams $\cocone{\cat I} \to \infcat C$ extending $D$ send the unique morphism $i \to \coconept$ to an equivalence (or, equivalently, if there exists one that does so).
\end{remark}

\begin{notation} \label{def:can_map}
  Let $\cat I$ be a category and $\infcat C$ be an $\infty$-category that admits colimits indexed by $\cat I$.
  Then we have the following diagram on the left and its image under $\Fun{\blank}{\infcat C}$ on the right:
  \[
  \begin{tikzcd}[sep = 30]
    \cat I \rar{\inc} \dar & \cocone{\cat I} \dar{\id} \dlar[Rightarrow, shorten < = 15, shorten > = 15, swap]{\xi} & & \Fun{\cocone{\cat I}}{\infcat C} \rar{\Res{\id}} \dar[swap]{\Res{\coconept}} & \Fun{\cocone{\cat I}}{\infcat C} \dar{\Res{\inc}} \dlar[Rightarrow, shorten < = 25, shorten > = 25, swap]{\xi} \\
    \termCat \rar{\coconept} & \cocone{\cat I} & & \infcat C \rar{\Diag} & \Fun{\cat I}{\infcat C}
  \end{tikzcd}
  \]
  where $\xi$ is, at an object $i \in \cat I$, the unique map to the cone point.
  We obtain a mate $\mate \xi \colon \colim{\cat I} \Res{\inc} \to \Res{\coconept} \Lan{\id} = \Res{\coconept}$.
  Evaluated at a diagram $D \colon \cocone{\cat I} \to \infcat C$ this takes the form of a map
  \[ \colim{\cat I} \restrict D {\cat I} \longto D(\coconept) \]
  that is natural in $D$.
  This is what we will mean by the \emph{canonical map} from the colimit.
\end{notation}

\begin{remark} \label{rem:one_point_colimit}
  When $\cat I = \termCat$, then $\colim{\cat I} = \id$ and the natural transformation $\Res{\inc} \to \Res{\coconept}$ is just given by $\xi$, i.e.\ evaluation at the unique morphism $\term \to \coconept$ in $\cocone{\termCat}$.
\end{remark}

\begin{notation}
  Let $\cat I$ be a category with an initial object $\ini$, $\infcat C$ an $\infty$-category that admits colimits indexed by $\gini{\cat I}$, and $D \colon \cat I \to \infcat C$ a diagram.
  Then, noting that the full subcategory of $\cat I$ spanned by $\ini$ and $\gini{\cat I}$ is canonically isomorphic to $\cone{(\gini{\cat I})}$, we obtain, by the dual of what we did in \Cref{def:can_map}, a transformation $\Res{\ini} \to \lim{\gini{\cat I}} \Res{\gini{\cat I}}$, i.e.\ a natural map
  \[ D(\ini) \longto \lim{\gini{\cat I}} \restrict{D}{\gini{\cat I}} \]
  which we see as the \emph{canonical map} in this situation.
\end{notation}

\begin{notation}
  Let $f \colon I \to J$ be a map of simplicial sets and $\infcat C$ an $\infty$-category that admits colimits indexed both by $I$ and by $J$.
  Then we have the following diagram on the left and its image under $\Fun{\blank}{\infcat C}$ on the right:
  \[
  \begin{tikzcd}[sep = 30]
    I \rar{f} \dar & J \dar \dlar[Rightarrow, shorten < = 14, shorten > = 14, swap]{\id} & & \infcat C \rar{\Diag[J]} \dar[swap]{\id} & \Fun{J}{\infcat C} \dar{\Res{f}} \dlar[Rightarrow, shorten < = 17, shorten > = 17, swap]{\id} \\
    \termCat \rar & \termCat & & \infcat C \rar{\Diag[I]} & \Fun{I}{\infcat C}
  \end{tikzcd}
  \]
  whose mate $\mate{\id}$ is a natural transformation $\colim{I} \Res{f} \to \colim{J}$ which we will denote by $f_*$ and call the \emph{induced map} on the colimit.
\end{notation}

\begin{notation}
  We will say a map $f \colon K \to L$ of simplicial sets is \emph{homotopy terminal} if, for each $\infty$-category $\infcat C$ and colimit diagram $p \colon \cocone L \to \infcat C$, the induced map $p \circ \cocone f \colon \cocone K \to \infcat C$ is again a colimit diagram.
  The dual concept will be called \emph{homotopy initial}.
\end{notation}

\begin{remark}
  By \cite[Proposition 4.1.1.8]{LurHTT}, this definition of homotopy terminal is an equivalent characterization of what Lurie calls cofinal (cf.\ \cite[Definition 4.1.1.1]{LurHTT}).
\end{remark}

\begin{notation} \label{def:preservation}
  Let $f \colon I \to J$ be a map of simplicial sets and $F \colon \infcat C \to \infcat D$ a functor between weakly left $f$-extensible $\infty$-categories.
  We say that $F$ \emph{preserves left Kan extension along $f$} if the mate $\chi \colon {\Lan{f}} \circ (F \;\circ) \to (F \;\circ) \circ \Lan{f}$ of the natural transformation
  \[
  \begin{tikzcd}[sep = 30]
    \Fun{J}{\infcat C} \rar{\Res{f}} \dar[swap]{F \circ} & \Fun{I}{\infcat C} \dar{F \circ} \dlar[Rightarrow, shorten < = 20, shorten > = 20, swap]{\id} \\
    \Fun{J}{\infcat D} \rar{\Res{f}} & \Fun{I}{\infcat D}
  \end{tikzcd}
  \]
  is an equivalence.
  
  We say that $F$ \emph{preserves colimits} indexed by a simplicial set $I$ if it preserves left Kan extension along $\const \colon I \to \termCat$.
\end{notation}

\section{Excisive functors} \label{section:excisive}

Classical Goodwilie calculus (as developed originally in the series of papers \cite{Goo90, Goo92, Goo03} and generalized to the $\infty$-categorical context in \cite[Section 6]{LurHA}) studies functors which have certain behaviors with respect to diagrams indexed by cubes, i.e.\ posets of the following form:

\begin{notation}
  Let $S$ be a set.
  We write $\Cube S$ for the poset of subsets of $S$ ordered by inclusion.
  Moreover, for $n \in \NN$, we set $\Cube n \defeq \Cube{\finset n}$ and call it the \emph{$n$-cube} (here $\finset n \defeq \set{0, \dots, n - 1}$).
\end{notation}

Namely, one defines a functor $F \colon \infcat C \to \infcat D$ between sufficiently nice $\infty$-categories to be $n$-excisive if it sends strongly cocartesian $(n+1)$-cubes to cartesian cubes.
Here, cartesian means that the cube is a limit diagram, which makes sense in much greater generality:

\begin{definition} \label{def:cartesian}
  Let $\cat I$ be a category that has an initial object and $\infcat C$ an $\infty$-category that admits limits indexed by $\ginipos I$.
  A diagram $D \colon \cat I \to \infcat C$ is \emph{cartesian} if there is an initial object $\ini \in \cat I$ such that the canonical map
  \[ D(\ini) \longto (\lim[s]{\gini{\cat I}} \Res{\gini{\cat I}}) (D) \]
  is an equivalence (by naturality of the map to the limit this is equivalent to requiring it to be an equivalence for each initial object).
\end{definition}

\begin{remark} \label{rem:cartesian}
  By \Cref{lemma:can_is_counit_of_kan}, this is equivalent to requiring the restriction $\restrict{D}{\cat J}$ of $D$ to the full subcategory $\cat J$ of $\cat I$ spanned by $\ini$ and $\gini{\cat I}$ (which is canonically isomorphic to $\cone{(\gini{\cat I})}$) to fulfill the condition that the unit map $\restrict{D}{\cat J} \to (\Ran{\inc} \Res{\inc}) (\restrict{D}{\cat J})$ is an equivalence, where $\inc \colon \gini{\cat I} \to \cat J$ is the inclusion.
  By \Cref{lemma:cocone_comparison,lemma:extensions_are_cocartesian}, this is in turn equivalent to $\restrict{D}{\cat J}$ being a limit diagram.
  The latter description makes sense even if not all limits indexed by $\gini{\cat I}$ exist, which makes it useful in some circumstances.
\end{remark}

Strongly cocartesian is a slightly more complicated condition: it means that any $2$-face of the cube is cocartesian, i.e.\ a pushout (cf.\ \cite[Definition 2.1]{Goo92}).
However, this can be rephrased in a more abstract way: it is equivalent to requiring the cube to be a left Kan extension of its restriction to the initial star, i.e.\ the following subposet:

\begin{notation}
  Let $S$ be a set.
  We write $\Cubeini S \subseteq \Cube S$ for the full subposet consisting of all subsets of $S$ with cardinality at most $1$.
  Moreover, for $n \in \NN$, we denote the inclusion $\Cubeini n \subseteq \Cube n$ by $\cubeinc n$.
\end{notation}

This rephrased condition does not make reference to the structure of the cube anymore, only to the inclusion $\cubeinc{n}$.
In particular, we can formulate it for arbitrary functors:

\begin{definition}
  Let $f \colon \gen{\cat I} \to \cat I$ be a functor between categories and $\infcat C$ a weakly left $f$-extensible $\infty$-category.
  A diagram $D \colon \cat I \to \infcat C$ is \emph{$f$-cocartesian} if the counit map $(\Ext[f] \Res{f})(D) \to D$ is an equivalence.
\end{definition}

\begin{remark} \label{rem:defs_of_str_cocart}
  That for $f = \cubeinc n$ this actually specializes to the condition of being a strongly cocartesian $n$-cube is shown in \cite[Proposition 6.1.1.15]{LurHA} and in slightly different language in \cite[Proposition 2.2]{Goo92}.
\end{remark}

\begin{remark}
  If $f$ is fully faithful, any diagram in the essential image of $\Lan{f}$ will be $f$-cocartesian, by \Cref{lemma:extensions_are_cocartesian}.
\end{remark}

Now that we have general notions of being ``strongly cocartesian'' with respect to some functor, we obtain a corresponding notion of excision for each of them:

\begin{definition} \label{def:excisive}
  Let $f \colon \gen{\cat I} \to \cat I$ be a functor between categories such that $\cat I$ has an initial object, $\infcat C$ a left $f$-extensible $\infty$-category, and $\infcat D$ an $\infty$-category that admits limits indexed by $\gini{\cat I}$.
  A functor $F \colon \infcat C \to \infcat D$ is \emph{$f$-excisive} if it takes $f$-cocartesian diagrams to cartesian diagrams.
  We write $\Exc{f}{\infcat C}{\infcat D} \subseteq \Fun{\infcat C}{\infcat D}$ for the full subcategory of $f$-excisive functors.
  
  Moreover, for $n \in \ZZge{-1}$, we say that a functor is $n$-excisive if it is $\cubeinc{n+1}$-excisive.
\end{definition}

\begin{remark}
  By \Cref{rem:defs_of_str_cocart}, our definition of an $n$-excisive functor agrees with Goodwillie's original one in \cite[Definition 3.1]{Goo92}.
  In particular our notion of an $f$-excisive functor generalizes the classical one.
\end{remark}

\begin{example} \label{ex:fixed_points}
  Let $G$ be a (discrete) group and set $\cat I_G \defeq \cone{(\B G)}$, i.e.\ the category with an initial object $\conept$ and a single other object $\bullet$, whose automorphisms are given by $G$.
  Furthermore, let $\iota_G \colon \set \conept \to \cat I_G$ denote the inclusion.
  
  Note that, for any $\infty$-category $\infcat C$, a diagram $\cat I_G \to \infcat C$ is $\iota_G$-cocartesian if and only if it is equivalent to a constant diagram.
  Moreover, in the $\infty$-category of spaces, the limit of a constant diagram $\B G \to \Spaces$, i.e.\ the homotopy fixed points of a space $X$ equipped with the trivial $G$-action, is given by $\Map{\B G}{X}$.
  In particular a diagram $\cat I_G \to \Spaces$ that is constant with value $X$ is cartesian if and only if the map $X \to \Map{\B G}{X}$, given by the inclusion of the constant maps, is an equivalence.
  This is the case if $X$ is discrete (when $G = \ZZ$ this is also a necessary condition).
  
  Hence, any functor $\infcat C \to \Spaces$ that takes values in discrete spaces is $\iota_G$-excisive (for all~$G$).
  For example, this is the case for the truncation functor $\pi_0 \colon \Spaces \to \Spaces$.
  However, it is easy to see that this functor is not $n$-excisive for any $n \ge -1$.
  In particular, for all groups $G$ and all $n \ge -1$, being $\iota_G$-excisive does not imply being $n$-excisive.
  
  On the other hand, if $G$ is non-trivial, then the constant functor $\const[\B G] \colon \infcat C \to \Spaces$ is not $\iota_G$-excisive.
  In particular, for any $n \ge 0$, being $n$-excisive does not imply being $\iota_G$-excisive.
\end{example}

\subsection{Preshapes}

Our overarching goal in this section is to show that under some hypotheses on a functor $f \colon \gen{\cat I} \to \cat I$ and the $\infty$-categories $\infcat C$ and $\infcat D$ there is, as in classical Goodwillie calculus, for any functor $F \colon \infcat C \to \infcat D$, a universal $f$-excisive functor approximating $F$ which can be explicitly constructed.

For this, and the rest of this paper, we will focus on functors of the following form, for reasons that will become apparent later.

\begin{definition} \label{def:preshape}
  A \emph{preshape} is a functor $\sigma \colon \gpos S \to \pos S$ between posets such that $\pos S$ has an initial object $\ini$ and $\inv \sigma (\ini)$ is non-empty.
  
  Let $\sigma \colon \gpos S \to \pos S$ and $\tau \colon \gpos T \to \pos T$ be preshapes.
  A \emph{map of preshapes} $\sigma \to \tau$ is a tuple $(f, \gen f)$ consisting of functors $f \colon \pos S \to \pos T$ and $\gen f \colon \gpos S \to \gpos T$ such that $f \circ \sigma = \tau \circ \gen f$ and $\inv f (\ini_{\pos T}) = \set{\ini_{\pos S}}$.
\end{definition}

\begin{remark}
  The collection of preshapes together with maps of preshapes forms a category with composition given by componentwise composition of functors.
\end{remark}

\begin{definition}
  Let $\sigma \colon \gpos S \to \pos S$ be a preshape.
  \begin{enumerate}[label=\alph*)]
    \item It is \emph{finite} if both $\gpos S$ and $\pos S$ are finite.
    \item It is \emph{full} if $\sigma$ is a full functor.
    \item It is \emph{reduced} if $\gpos S$ has an initial object.
  \end{enumerate}
\end{definition}

\begin{remark}
  Since functors between posets are automatically faithful, a full preshape is already fully faithful.
  Furthermore, by \Cref{lemma:pos_full_implies_injective}, it is also injective (on objects).
\end{remark}

\begin{remark}
  Note that a reduced preshape preserves the initial object (as $\inv\sigma(\ini_{\pos S})$ is downward closed).
  Furthermore, a full preshape is automatically reduced since fully faithful functors reflect initial objects.
\end{remark}

\begin{notation}
  When $\sigma \colon \gpos S \to \pos S$ is a reduced preshape, we denote the initial object of $\gpos S$ by $\genini$.
\end{notation}

\subsection{Construction of the excisive approximation}

Our construction of the universal excisive approximation is a generalization of Goodwillie's original construction for topological spaces \cite[Section 1]{Goo03} (which is also used in generalized form by Lurie in \cite[Section 6.1.1]{LurHA}).
There, an important part is played by the cubical diagrams given by mapping a subset $U \subseteq \finset n$ to the join of $X$ and $U$, where $X$ is some space and $U$ is considered as a discrete space (one concrete construction of this join is to take, for each element of $U$, a copy of the cone of $X$ and glue them all together at their bases).
We will now describe a more abstract way of constructing these diagrams (which is basically the same way Lurie does it).

\begin{notation} \label{def:Pad}
  Let $\sigma \colon \gpos S \to \pos S$ be a reduced preshape and $\infcat C$ an $\infty$-category with a terminal object.
  We write $\Pad \defeq \Ran{\inc} \colon \infcat C \longto \Fun{\gpos S}{\infcat C}$, where $\inc \colon \set \genini \to \gpos S$ is the inclusion.
  Note that the categories $\c s \slice \inc$ are, for all $\c s \in \gpos S$, isomorphic to either the empty or the terminal category, and hence our assumption on $\infcat C$ implies that $\Pad$ exists.
\end{notation}

\begin{remark}
  The diagram $\Pad (X)$ has $X$ at $\genini$ and the terminal object of $\infcat C$ at all other points of $\gpos S$ (i.e.\ it pads the diagram with terminal objects).
\end{remark}

\begin{notation}
  Let $\sigma \colon \gpos S \to \pos S$ be a reduced preshape and $\infcat C$ a weakly left $\sigma$-extensible $\infty$-category with a terminal object.
  Then the composition
  \[ \infcat C \xlongto{\Pad} \Fun{\gpos S}{\infcat C} \xlongto{\Ext} \Fun{\pos S}{\infcat C} \]
  curries to a functor $\infcat C \times \pos S \to \infcat C$ which we denote by $\join$.
\end{notation}

\begin{remark}
  By \Cref{lemma:kan_local}, we have, for $X \in \infcat C$ and $s \in \pos S$, the more explicit formula
  \[ X \join s = (\Ext \Pad)(X)(s) \eq \colim{\sigma \slice s} (\Pad (X) \circ \pr[\sigma \slice s]) \]
  (as long as $\infcat C$ is left $\sigma$-extensible).
\end{remark}

\begin{remark} \label{rem:Pad}
  By definition of $\Pad(X)$, it admits a canonical map from any diagram $\gpos S \to \infcat C$ with $X$ at $0$.
  This is its main useful property and will allow us to factor maps into it in a useful way.
\end{remark}

\begin{remark}
  In the case where $\sigma = \cubeinc{n}$ and $\infcat C$ is the $\infty$-category of spaces, the functor $\join$ specializes to the join with a discrete set (we can imagine the terminal objects occurring in $\Pad[\Cubeini{n}] (X)$ to be the cone over X which are then glued together by taking a left Kan extension).
  Hence $\join$ generalizes the cubical diagrams mentioned above.
\end{remark}

The following basic property will be needed later.

\begin{lemma} \label{lemma:join_preserves}
  Let $\sigma \colon \gpos S \to \pos S$ be a reduced preshape, $\cat I$ a contractible category, and $\infcat C$ a left $\sigma$-extensible $\infty$-category that admits colimits indexed by $\cat I$ and has a terminal object.
  Then, for any $s \in \pos S$, the functor $(\blank \join s) = \Res{s} \Ext \Pad$ preserves terminal objects and colimits indexed by $\cat I$.
\end{lemma}

\begin{proof}
  Note that, by \Cref{lemma:limits_preserve_limits}, the functor $\Lan{\sigma}$ preserves colimits and the functor $\Pad$ preserves limits.
  Hence, since $\Res{s}$ also preserves colimits, it is enough to show that $\Pad$ preserves colimits indexed by $\cat I$ and that $\Res{s} \Ext$ preserves terminal objects.
  
  For the second statement note that an object of $\Fun{\gpos S}{\infcat C}$ is terminal if and only if it is pointwise terminal.
  In particular, any such terminal object is equivalent to the restriction ${\const[\term]} \circ c$, where $\const[\term] \colon \termCat \to \infcat C$ represents a terminal object of $\infcat C$ and $c \colon \gpos S \to \termCat$ is the constant map.
  Now, by \Cref{lemma:kan_local}, the object $(\Res{s} \Lan{\sigma}) ({\const[\term]} \circ c)$ can be computed by $\colim{\sigma \slice s} ({\const[\term]} \circ c \circ \pr[\sigma \slice s])$.
  But this is the terminal object of $\infcat C$ since $c \circ \pr[\sigma \slice s]$ is homotopy terminal as $\sigma \slice s$ has an initial object and is thus contractible.
  
  For the first statement it is, by \Cref{lemma:pointwise_preservation}, enough to show that, for any $\c s \in \gpos S$, the functor $\Res{\c s} \Pad$ preserves colimits indexed by $\cat I$.
  Note that, again by \Cref{lemma:kan_local}, the functor $\Res{\c s} \Pad$ is equivalent to $\lim{\c s \slice \inc} \circ \Res{\pr[{\c s \slice \inc}]}$, where $\inc$ denotes the inclusion of $\set \genini$ into $\gpos S$.
  But $\c s \slice \inc$ is either the terminal category, in which case $\lim{\c s \slice \inc} = \id$ clearly preserves colimits, or empty.
  In the latter case $\lim{\c s \slice \inc}$ is the functor $\const[\term] \colon \termCat \to \infcat C$ for some terminal object $\term \in \infcat C$.
  This preserves colimits indexed by the contractible category $\cat I$ since the constantly terminal diagram $\const[\term] \colon \cocone{\cat I} \to \infcat C$ is a colimit diagram extending $\const[\term] \colon \cat I \to \infcat C$ by \cite[Proposition 4.3.1.12]{LurHTT} (together with \cite[Proposition 2.4.1.5]{LurHTT}).
\end{proof}

Now we are ready to give the construction of the universal excisive approximation.
Note, however, that it will only have the desired properties after assuming more conditions on $\sigma$ and the $\infty$-categories $\infcat C$ and $\infcat D$.

\begin{construction} \label{def:T}
  Let $\sigma \colon \gpos S \to \pos S$ be a reduced preshape, $\infcat C$ a left $\sigma$-extensible $\infty$-category with a terminal object, and $\infcat D$ an $\infty$-category admitting sequential colimits and limits indexed by $\ginipos S$.
  We write
  \[ \T \colon \Fun{\infcat C}{\infcat D} \longto \Fun{\infcat C}{\infcat D} \]
  for the functor given by
  \[ F \longmapsto \lim[s]{\ginipos S} \circ {\Res{\ginipos S}} \circ (F \;\circ) \circ \Ext \circ \Pad \;. \]
  There is a natural transformation of functors $\infcat C \to \infcat C$
  \[ \tau_{\sigma} \colon \id \xlongfrom{\eq} \Res{\genini} \Pad \longto \Res{\genini} \Res{\sigma} \Ext \Pad = \Res{\ini} \Ext \Pad \]
  coming from the counit of the adjunction $\Res{\genini} \dashv \Pad$ and the unit of $\Res{\sigma} \dashv \Ext$, the first of which is an equivalence since the inclusion $\set \genini \subseteq \gpos S$ is fully faithful (later we will need to assume that $\sigma$ is full precisely because we need $\tau_\sigma$ and thus the mentioned unit to be equivalences).
  We obtain a natural transformation $\t \colon \id \to \T$ of functors $\Fun{\infcat C}{\infcat D} \to \Fun{\infcat C}{\infcat D}$ defined at $F \in \Fun{\infcat C}{\infcat D}$ by the composition
  \[
  \begin{tikzcd}
    F \rar{F \circ \tau_{\sigma}} & F \circ \Res{\ini} \circ \Ext \circ \Pad \dar[equal] &[-10] \T(F) \dar[equal] \\[-12]
     & {\Res{\ini}} \circ (F \;\circ) \circ \Ext \circ \Pad \rar & \lim[s]{\ginipos S} \circ {\Res{\ginipos S}} \circ (F \;\circ) \circ \Ext \circ \Pad
  \end{tikzcd}
  \]
  where the second morphism is the canonical map to the limit.
  Now, by \Cref{lemma:sequential_diagrams}, the sequence of morphisms
  \begin{equation} \label[diagram]{diag:P}
  \id \xlongto{\t} \T \xlongto{\t \circ \T} \T \T \xlongto{\t \circ \T \circ \T} \dots
  \end{equation}
  defines a sequential diagram.
  By our conditions on $\infcat D$, its colimit exists.
  We will denote it by $\P$ and by $\p \colon \id \to \P$ the structure map to the colimit.
\end{construction}

\begin{remark}
  A more explicit formula for computing $\T$ is
  \[ \T(F)(X) \eq \lim[s]{s \in \ginipos S} F(X \join s) \]
  which (in a less general form) was Goodwillie's original definition (cf.\ \cite[Section 1]{Goo03}).
  In this form $\t$ is the composition of $F(X) \to F(X \join \ini)$ and the canonical map into the limit (the first of which is an equivalence if $\sigma$ is full).
\end{remark}

\begin{remark} \label{rem:P_not_functorial}
  The construction of $\T$ (and thus the one of $\P$) is \emph{not} functorial in $\sigma$.
  The problem is that, although a map $(f, \c f) \colon (\sigma \colon \gpos S \to \pos S) \to (\tau \colon \gpos T \to \pos T)$ of reduced preshapes induces maps $X \join_\sigma s \to X \join_\tau f(s)$ and a map $\lim{t \in \ginipos T}(X \join_\tau t) \to \lim{s \in \ginipos S} (X \join_\tau f(s))$, they do not combine into a map between $\T[\sigma]$ and $\T[\tau]$.
  However, there are functorial properties when restricted to certain subcategories.
  This is discussed in \Cref{section:maps}.
\end{remark}

For this construction to work well, we will need to assume further conditions on the target $\infty$-category $\infcat D$, namely that the occurring sequential colimits are compatible with certain limits.

\begin{definition}
  Let $\cat I$ be a category.
  An $\infty$-category $\infcat D$ is \emph{$\cat I$-differentiable} if it admits sequential colimits as well as limits indexed by $\gini{\cat I}$, and taking sequential colimits preserves limits indexed by $\gini{\cat I}$.
  
  It is \emph{differentiable} if it admits sequential colimits as well as all finite limits, and taking sequential colimits preserves finite limits.
\end{definition}

\begin{remark}
  The notion of a differentiable $\infty$-category was introduced by Lurie; see \cite[Definition 6.1.1.6]{LurHA}.
\end{remark}

\begin{remark}
  Note that differentiable implies $\cat I$-differentiable for any finite category $\cat I$ (i.e.\ a category with finitely many objects and morphisms).
\end{remark}

\begin{remark}
  By \Cref{lemma:colim_commutes_lim}, the condition that sequential colimits preserve limits indexed by $\ginipos S$ is equivalent to requiring the functor $\lim{\ginipos S}$ to preserve sequential colimits.
\end{remark}

\begin{example}
  The following are examples of differentiable $\infty$-categories:
  \begin{itemize}
    \item any $\infty$-topos (see \cite[Example 6.1.1.8]{LurHA}), in particular the $\infty$-category of spaces (see \cite[Proposition 6.3.4.1]{LurHTT}).
    
    \item any stable $\infty$-category (cf.\ \cite[Definition 1.1.1.9]{LurHA}) that admits countable coproducts (see \cite[Example 6.1.1.7]{LurHA}).
    
    \item the $\infty$-category $\infcat C_*$ of pointed objects (cf.\ \cite[Definition 7.2.2.1]{LurHTT}) in a differentiable $\infty$-category $\infcat C$.
    In particular this tells us that the $\infty$-category of pointed spaces is differentiable.
    
    To see this, note that $\infcat C_*$ is defined as the full subcategory of $\Fun{\Simplex{1}}{\infcat C}$ spanned by the maps $f \colon \Simplex{1} \to \infcat C$ with $f(0)$ a terminal object of $\infcat C$.
    But this subcategory is closed under the formation, in $\Fun{\Simplex{1}}{\infcat C}$, of finite limits and sequential colimits (the latter fact uses \Cref{lemma:contractible_colimit_over_equivalences}).
    Since fully faithful functors reflect (co)limits by \cite[Proposition 2.4.7]{RV}, this implies that $\infcat C_*$ is differentiable.
  \end{itemize}
\end{example}

We can already prove some elementary properties of $\T$ and $\P$ that we will need later.

\begin{lemma} \label{lemma:T_and_P_elementary}
  Let $\sigma \colon \gpos S \to \pos S$ be a reduced preshape, $\infcat C$ a left $\sigma$-extensible $\infty$-category with a terminal object, and $\infcat D$ an $\pos S$-differentiable $\infty$-category.
  \begin{enumerate}[label=\alph*)]
    \item If $\sigma$ is full, then, for any $\sigma$-excisive functor $F \colon \infcat C \to \infcat D$, both $\t(F)$ and $\p(F)$ are equivalences.
    \item Both $\T$ and $\P$ preserve limits indexed by $\ginipos S$.
    \item Both $\T$ and $\P$ preserve sequential colimits.
    \item Let $\infcat C'$ be another left $\sigma$-extensible $\infty$-category with a terminal object and $F \colon \infcat C' \to \infcat D$ a functor.
    Then the functors ${\P} \circ (F \;\circ)$ and $(\P(F) \;\circ)$ from $\Fun{\infcat C}{\infcat C'}$ to $\Fun{\infcat C}{\infcat D}$ are equivalent when restricted to the full subcategory consisting of those functors $G \colon \infcat C \to \infcat C'$ that preserve terminal objects and left Kan extension along $\sigma$.
  \end{enumerate}
\end{lemma}

\begin{proof} \leavevmode
\begin{enumerate}[label=\alph*)]
  \item The map $\t(F)$ is the composition of $F \circ \tau_{\sigma}$, which is an equivalence when $\sigma$ is full, and the canonical map to the limit.
  The latter map is an equivalence by construction since, for any $X \in \infcat C$, the diagram $(\Ext \Pad)(X)$ is $\sigma$-cocartesian by \Cref{lemma:extensions_are_cocartesian} and hence $F \circ (\Ext \Pad)(X)$ cartesian.
  Thus $\p(F)$ is also an equivalence by \Cref{lemma:contractible_colimit_over_equivalences} as each map in \cref{diag:P} is an equivalence (using that $(\T)^n(F)$ is, by induction, equivalent to $F$ and hence $\sigma$-excisive).
  
  \item \label{item:T_P_preserve limits} It follows directly from \Cref{lemma:fun_preserving_limits,lemma:limits_preserve_limits} that $\T$ preserves limits indexed by $\ginipos S$.
  This also implies that the functor $\NN \to \Fun{\Fun{\infcat C}{\infcat D}}{\Fun{\infcat C}{\infcat D}}$ described by \cref{diag:P} sends each $n \in \NN$ to a functor that preserves limits indexed by $\ginipos S$.
  Hence, by \Cref{lemma:pointwise_preservation}, the associated functor $\Fun{\infcat C}{\infcat D} \to \Fun{\NN}{\Fun{\infcat C}{\infcat D}}$ preserves limits indexed by $\ginipos S$.
  Now, since $\infcat D$ is $\pos S$\=/differentiable, the functor ${\colim{}} \colon \Fun{\NN}{\Fun{\infcat C}{\infcat D}} \to \Fun{\infcat C}{\infcat D}$ preserves limits indexed by $\ginipos S$ (the category $\Fun{\infcat C}{\infcat D}$ is again differentiable by \Cref{lemma:kan_and_currying,lemma:fun_preserving_limits}).
  As the composition of these two functors is precisely $\P$, this implies that $\P$ preserves these limits as well.
  
  \item This follows similarly to \ref{item:T_P_preserve limits} by noting that ${\lim{}} \colon \Fun{\ginipos S}{\infcat D} \to \infcat D$ preserves sequential colimits when $\infcat D$ is $\pos S$-differentiable.
  
  \item Using (the dual of) \Cref{lemma:preserving_Kan_extensions}, we see that any such $G$ preserves right Kan extensions along the inclusion $\inc \colon \set \genini \to \gpos S$ (using that, for all $\c s \in \gpos S$, the comma categories $\c s \slice \inc$ are either empty or the terminal category).
  This implies the corresponding statement for $\T$, i.e.\ that there is an equivalence $\alpha \colon (\T(F) \;\circ) \to {\T} \circ (F \;\circ)$ (when restricted to the subcategory).
  Furthermore, by two applications of \Cref{lemma:preservation_comp_unit}, we obtain that the diagram
  \begin{equation} \label[diagram]{diag:t_preserving}
  \begin{tikzcd}
    (F \;\circ) \rar{\t(F)} \dar[swap]{\id} & (\T(F) \;\circ) \dar{\alpha} \\
    (F \;\circ) \rar{\t} & {\T} \circ (F \;\circ)
  \end{tikzcd}
  \end{equation}
  commutes up to homotopy.
  
  Now we can inductively define equivalences $\alpha_n \colon ((\T)^n(F) \;\circ) \to (\T)^n \circ (F \;\circ)$ by setting $\alpha_0 = \id$ and
  \[ \alpha_{n+1} \colon ((\T)^{n+1}(F) \;\circ) \xlongto{\alpha} {\T} \circ ((\T)^n(F) \;\circ) \xlongto{\alpha_n} {\T} \circ (\T)^n \circ (F \;\circ) \;. \]
  Now note that, since the diagram
  \[
  \begin{tikzcd}
    ((\T)^n(F) \;\circ) \rar{\t} \dar[swap]{\id} & (\T ((\T)^n(F)) \;\circ) \dar{\alpha} \\
    ((\T)^n(F) \;\circ) \rar{\t} \dar[swap]{\alpha_n} & {\T} \circ ((\T)^n(F) \;\circ) \dar{\alpha_n} \\
    (\T)^n \circ (F \;\circ) \rar{\t} & {\T} \circ (\T)^n \circ (F \;\circ)
  \end{tikzcd}
  \]
  commutes up to homotopy (using that the upper square is a special case of \cref{diag:t_preserving}), the $\alpha_n$ assemble into an equivalence from the sequential diagram defining $(\P(F) \;\circ)$ to the sequential diagram defining ${\P} \circ (F \;\circ)$ (using \Cref{lemma:sequential_diagrams}).
  \qedhere
\end{enumerate}
\end{proof}

\begin{remark} \label{rem:weaker_excisive}
  Note that in the first part of the previous lemma we did not use the full strength of $F$ being $\sigma$-excisive, only that it sends diagrams in the essential image of $\Lan{\sigma} \Pad$ to cartesian diagrams.
  Since we will show that (under certain conditions) the functor $\P(F)$ is $\sigma$-excisive (see \Cref{lemma:P_is_excisive}), this implies that (under these conditions) the a priori weaker property above is actually equivalent to being $\sigma$-excisive.
\end{remark}

\subsection{Shapes}

Unfortunately, it is not true that, for any preshape $\sigma$, the functor $\P (F)$ is a universal $\sigma$-excisive approximation to $F \colon \infcat C \to \infcat D$ (even if the $\infty$-categories $\infcat C$ and $\infcat D$ are as nice as we want).
This is shown by the following example:

\begin{example} \label{ex:P_not_always_exc}
  Consider the full subposet $\gpos D \defeq \set{\emptyset, \set{1,2}} \subseteq \Cube{2}$ and denote the inclusion by $\delta$.
  It is clear that $\delta$ is a full preshape.
  Now let $\infcat C$ and $\infcat D$ be two $\infty$-categories.
  By \Cref{lemma:extensions_are_cocartesian}, a diagram $\Cube{2} \to \infcat C$ is $\delta$-cocartesian if and only if it is equivalent to one of the form $\Lan{\delta} (E)$ for some diagram $E \colon \gpos D \to \infcat C$ (the diagram $\Lan{\delta} (E)$ looks like
  \[
  \begin{tikzcd}
    X \rar{\id} \dar[swap]{\id} & X \dar{f} \\
    X \rar{f} & Y
  \end{tikzcd}
  \]
  where $f$ is the morphism of $\infcat C$ represented by $E \colon \Simplex{1} \iso \gpos D \to \infcat C$).
  In particular, for a functor $F \colon \infcat C \to \infcat D$ to be $\delta$-excisive it has to send any such $\Lan{\delta} (E)$ to a cartesian diagram.
  For $Y$ the terminal object of $\infcat C$ this would imply, if $F$ preserves terminal objects, that
  \[ F(X) \xlongfrom{\id} F(X) \xlongto{\id} F(X) \]
  exhibits $F(X)$ as a product of $F(X)$ with itself.
  In the case where $\infcat D$ is the $\infty$-category $\ptSpaces$ of pointed spaces, this can only be the case if $F(X)$ is weakly contractible (by considering its homotopy groups).
  This shows that a $\delta$-excisive functor $\infcat C \to \ptSpaces$ that preserves terminal objects has its image contained in the terminal objects.
  
  Now let $F \colon \infcat C \to \ptSpaces$ be any functor that preserves terminal objects.
  Since $\Pad[\gpos D]$ is given by sending $X \in \infcat C$ to a map $X \to \term$, we obtain that $\T[\delta](F)$ is the functor $X \mapsto F(X) \times F(X)$ and that $\t[\delta]$ is given, at $X$, by the diagonal $F(X) \to F(X) \times F(X)$.
  In particular, we have that $\P[\delta](F)(X)$ is given by the colimit of the sequence
  \[ F(X) \xlongto{\Diag} F(X) \times F(X) \xlongto{\Diag} (F(X) \times F(X)) \times (F(X) \times F(X)) \longto \cdots \]
  and hence is not weakly contractible when $F(X)$ is not weakly contractible since homotopy groups commute with products and sequential (homotopy) colimits (for 1-categorical sequential colimits over inclusions this can be found in \cite[Chapter 9.4]{May}; for sequential homotopy colimits it follows from the 1-categorical case by taking a cofibrant replacement).
  Thus, if the image of $F$ is not contained in the terminal objects, then $\P[\delta](F)$ cannot be $\delta$-excisive.
\end{example}

\begin{remark}
  If, in the definition of a preshape, we relax the condition of being a functor of posets to being a functor of categories, then the functor $\iota_\ZZ$ of \cref{ex:fixed_points} is another example such that $\P[\iota_\ZZ]$ is not necessarily $\iota_\ZZ$-excisive.
\end{remark}

However, we can put conditions on $\sigma$ so that $\P$ is a functor of universal $\sigma$-excisive approximation.
This is done in the next definition.
We will see that $\cubeinc{n}$ fulfills these conditions, so that our statements actually generalize the classical ones.

\begin{definition} \label{def:shape}
  A preshape $\sigma \colon \gpos S \to \pos S$ is a \emph{shape} if $\pos S$ has all (small) coproducts, and, for all $s, t \in \pos S$ and $\c k \in \gpos S$ such that $\sigma(\c k) \le t \cop s$, the full subposet
  \[ \gpos{S}_{s,t,\c k} := \set*{\c s \in \gpos S \mid \sigma(\c s) \le s \text{ and } \sigma(\c k) \le t \cop \sigma(\c s)} \subseteq \gpos S \]
  is contractible.
  A \emph{map of shapes} is a map between the underlying preshapes.
\end{definition}

\begin{remark}
  For everything we do with shapes in this section it would be enough to only require $\pos S$ to admit finite coproducts.
  However, we use the stronger version since it makes the constructions in \Cref{section:free_shapes} easier to work with (though one should be able to work around this, so that in the end (almost) all statements we will make should also hold with the weaker requirements).
\end{remark}

The following is an easy-to-check sufficient criterion for a reduced preshape to be a shape.

\begin{lemma} \label{lemma:easy_shape_condition}
  Let $\sigma \colon \gpos S \to \pos S$ be a full preshape.
  Assume that, for all $a, b \in \pos S$ and $\c c \in \gpos S$ such that $\sigma(\c c) \le a \cop b$, we have $\sigma(\c c) \le a$ or $\sigma(\c c) \le b$.
  Then $\sigma$ is a shape.
\end{lemma}

\begin{proof}
  We want to show that, for all $s, t \in \pos S$ and $\c k \in \gpos S$ such that $\sigma(\c k) \le t \cop s$, the poset $\gpos{S}_{s,t,\c k}$ is contractible.
  For this first assume that $\sigma(\c k) \le t$.
  In this case ${\gpos S}_{s,t,\c k}$ has $\genini$ as an initial object and is thus contractible.
  Otherwise $\sigma(\c k) \not\le t$ and our assumption implies $\sigma(\c k) \le s$ and hence $\c k \in {\gpos S}_{s,t,\c k}$.
  Thus we have that $\sigma(\c k) \le \sigma(\c s)$ for any $\c s \in {\gpos S}_{s,t,\c k}$.
  As $\sigma$ is full this implies that $\c k$ is an initial object of ${\gpos S}_{s,t,\c k}$, which finishes the proof.
\end{proof}

\begin{example} \label{ex:shapes}
  We give some (non-)examples for \Cref{def:shape}:
  \begin{itemize}
    \item The preshape $\cubeinc{n}$ is a shape.
    This follows directly from \Cref{lemma:easy_shape_condition}.
    (A more general version of this statement will be proven in \Cref{lemma:univ_is_full_shape}.)
    
    \item More generally, let $S$ be a set and $\gpos S$ any non-empty full subposet of $\Cube S$.
    Then the inclusion $\gpos S \to \Cube S$ is a shape if and only if $\gpos S$ is downward closed in $\Cube S$.
    This will be proven in \Cref{lemma:cube_shape}.
    
    \item The preshape $\delta$ from \Cref{ex:P_not_always_exc} is not a shape.
    For example the poset ${\gpos D}_{s,t,\c k}$ for $s = \set 1$, $t = \set 2$, and $\c k = \set{1, 2}$ is empty and thus not contractible.
  \end{itemize}
\end{example}

The main motivation for the definition of a shape is that it is precisely what we need for the next lemma.
However, before we can state it, we need a way to assume, depending on the shape, enough colimits to exist.
This will be achieved by the following definition.
It is somewhat stronger than what we will actually need, but a lot more convenient to work with (it would be possible to track the precise requirements; however we chose to not do so in favor of increased readability).

\begin{definition} \label{def:nice}
  Let $f \colon \pos I \to \pos J$ be a functor between posets and $\infcat C$ an $\infty$-category.
  If both $\pos I$ and $\pos J$ are finite, we say that $\infcat C$ is \emph{$f$-nice} if it admits all finite colimits.
  Otherwise we say that $\infcat C$ is \emph{$f$-nice} if it admits colimits of size up to the maximum of the cardinalities of $\pos I$ and $\pos J$.
\end{definition}

\begin{remark}
  Note that an $f$-nice $\infty$-category is automatically left $g$-extensible for any functor $g$ with source $\pos I$ or $\pos J$ and target a poset (since the corresponding slice categories have a cardinality bounded by the cardinality of $\pos I$ respectively $\pos J$).
\end{remark}

\begin{lemma} \label{lemma:shifted_is_cocartesian}
  Let $\sigma \colon \gpos S \to \pos S$ be a shape, $\infcat C$ a $\sigma$-nice $\infty$-category, and $D \colon \gpos S \to \infcat C$ a diagram.
  Then, for any $t \in \pos S$, the diagram $\Lan{\sigma} (D) \circ (t \cop \blank)$ is $\sigma$-cocartesian.
  In particular $\Lan{\sigma} (D)$ is $\sigma$-cocartesian, even though $\sigma$ is not necessarily full (in which case \Cref{lemma:extensions_are_cocartesian} would imply the statement).
  
  This also implies that, for any $\sigma$-cocartesian diagram $D' \colon \pos S \to \infcat C$ and $t \in \pos S$, the diagram $D' \circ (t \cop \blank)$ is again $\sigma$-cocartesian.
\end{lemma}

\begin{proof}
  This is a special case of the (technical) next lemma.
  More precisely, we apply it to the situation
  \[
  \begin{tikzcd}
    \gpos S \dar{\sigma} & \gpos S \rar{D} \dar{\sigma} &[10] \infcat C \\
    \pos S \rar{t \cop \blank} & \pos S &
  \end{tikzcd}
  \]
  for which we need that, for all $s \in \pos S$ and $\c k \in \gpos S$ such that $\sigma(\c k) \le (t \cop \blank)(s) = t \cop s$, the poset
  \[ \set*{\c s \in \gpos S \mid \sigma(\c s) \le s \text{ and } \sigma(\c k) \le t \cop \sigma(\c s)} = \gpos{S}_{s,t,\c k} \]
  is contractible.
  But this is precisely the assumption on $\sigma$ for it to be a shape.
\end{proof}

\begin{lemma} \label{lemma:gc_equiv}
  Let $\pos I$, $\pos J$, $\pos K$, and $\pos L$ be posets, and $f$, $g$, and $h$ functors as the diagram
  \[
  \begin{tikzcd}
    \pos J \dar{g} & \pos I \dar{f} \\
    \pos K \rar{h} & \pos L
  \end{tikzcd}
  \]
  specifies and $\infcat C$ an $f$-nice and $g$-nice $\infty$-category.
  Furthermore, assume that, for all $k \in \pos K$ and $i \in \pos I$ such that $f(i) \le h(k)$, the full subposet
  \[ \set{j \in \pos J \mid f(i) \le h(g(j)) \text{ and } g(j) \le k} \subseteq \pos J \]
  is contractible.
  Then, for any diagram $D \colon \pos I \to \infcat C$, the diagram $(\Res{h} \Lan{f})(D)$ is $g$-cocartesian.
\end{lemma}

\begin{remark}
  The counit of the adjunction $\Lan{g} \dashv \Res{g}$ has, precomposed with $\Res{h} \Lan{f}$ and evaluated at $k \in \pos K$, the form
  \[ \colim{j \in g \slice k} \colim{f \slice h(g(j))} \longto \colim{f \slice h(k)}. \]
  The conditions of \Cref{lemma:gc_equiv} precisely guarantee that the collection $(f \slice h(g(j)))_{j \in g \slice k}$ is a cover of $f \slice h(k)$ that is nice enough to force the above map to be an equivalence (cf.\ \cite[Corollary 4.2.3.10 and Remark 4.2.3.9]{LurHTT}).
\end{remark}

\begin{proof}[Proof of \Cref{lemma:gc_equiv}]
  Consider the diagram
  \[
  \begin{tikzcd}[sep = 35]
  f \slice (h \circ g) \rar{p} \dar[swap]{q} & \pos I \dar{f} \dlar[Rightarrow, shorten < = 20, shorten > = 20][swap]{\alpha} \\
  \pos J \rar{h \circ g} \dar[swap]{g} & \pos L \dar{\id} \dlar[Rightarrow, shorten < = 20, shorten > = 20][swap]{\id} \\
  \pos K \rar{h} & \pos L
  \end{tikzcd}
  \]
  where $p = \pr[\cat I]$ and $q = \pr[\cat J]$ are the two projections and $\alpha$ comes from $f(p(i, j)) = f(i) \le h(g(j)) = h(g(q(i, j)))$.
  Applying $\Fun{\blank}{\infcat C}$ yields the following diagram on the left and subsequently taking mates the one on the right:
  \[
  \begin{tikzcd}[sep = 35]
  \Fun{f \slice (h \circ g)}{\infcat C} & \Fun{\pos I}{\infcat C} \lar[swap]{\Res{p}} \dlar[Rightarrow, shorten < = 25, shorten > = 25][swap]{\alpha} &[-15] &[-15] \Fun{f \slice (h \circ g)}{\infcat C} \dar[swap]{\Lan{q}} \drar[Rightarrow, shorten < = 25, shorten > = 25]{\mate \alpha} & \Fun{\pos I}{\infcat C} \lar[swap]{\Res{p}} \dar{\Lan{f}} \\
  \Fun{\pos J}{\infcat C} \uar{\Res{q}} & \Fun{\pos L}{\infcat C} \lar[swap]{\Res{h \circ g}} \uar[swap]{\Res{f}} \dlar[Rightarrow, shorten < = 25, shorten > = 25][swap]{\id} & & \Fun{\pos J}{\infcat C} \dar[swap]{\Lan{g}} \drar[Rightarrow, shorten < = 25, shorten > = 25]{\mate \id} & \Fun{\pos L}{\infcat C} \lar[swap]{\Res{h \circ g}} \dar{\Lan{\id}} \\
  \Fun{\pos K}{\infcat C} \uar{\Res{g}} & \Fun{\pos L}{\infcat C} \lar[swap]{\Res{h}} \uar[swap]{\Res{\id}} & & \Fun{\pos K}{\infcat C} & \Fun{\pos L}{\infcat C} \lar[swap]{\Res{h}}
  \end{tikzcd}
  \]
  (where all occurring Kan extensions exist by our assumptions on $\infcat C$).
  
  Now note that since, by definition, the mate $\mate \id$ is given by the composition
  \[ \Lan{g} \Res{h \circ g} = \Lan{g} \Res{h \circ g} \Res{\id} \Lan{\id} = \Lan{g} \Res{g} \Res{h} \Lan{\id} \xlongto{\epsilon} \Res{h} \Lan{\id} = \Res{h} \]
  applying $\mate \id$ to a diagram $D \colon \pos L \to \infcat C$ results precisely in the map from the definition of $g$-cocartesianness for $\Res{h}(D)$.
  In particular it is enough to show that $\mate \id \circ \Lan{f}$ is an equivalence, which is one of the maps that occur in the paste $\mate \id \paste \mate \alpha = \mate{({\id} \paste \alpha)}$ (where the equality (in the homotopy 2-category of $\infty$-categories) comes from the pasting law for mates).
  Since $\mate \alpha$ is an equivalence by \Cref{lemma:comma_square_mate}, it is thus enough to show that $\mate{({\id} \paste \alpha)} \colon \Lan{g \circ q} \Res{p} \longto \Res{h} \Lan{{\id} \circ f}$ is one as well.
  
  For this it is, by \Cref{lemma:kan_mate}, enough to show that for any $k \in \pos K$, the map
  \[ \colim[s]{(g \circ q) \slice k} \Res{\pr[(g \circ q) \slice k]} \Res{p} \xlongto{r_*} \colim[s]{f \slice h(k)} \Res{\pr[f \slice h(k)]} \]
  induced by the functor $r \colon (g \circ q) \slice k \to f \slice h(k)$ given by $p$ is an equivalence.
  We claim that $r$ is homotopy terminal, for which we need that, for all $i \in f \slice h(k)$, the poset
  \[ i \slice r = \set{(i', j') \in \pos I \times \pos J \mid f(i') \le h(g(j')) \text{ and } g(j') \le k \text{ and } i \le i'} \]
  is contractible.
  For this note that the map
  \[ \pos Q \defeq \set{j' \in \pos J \mid f(i) \le h(g(j')) \text{ and } g(j') \le k} \longto i \slice r \]
  given by $j' \mapsto (i, j')$ is left adjoint to the projection $i \slice r \to \pos Q$ given by $(i', j') \mapsto j'$ and hence a homotopy equivalence.
  Thus, it is enough to show that the poset on the left is contractible, which is true by assumption.
\end{proof}

\subsection{Proof of the excisive approximation}

We can now formulate the main result of this section.

\begin{theorem} \label{thm:full_approximation}
  Let $\sigma \colon \gpos S \to \pos S$ be a full shape, $\infcat C$ a $\sigma$-nice $\infty$-category with a terminal object, and $\infcat D$ an $\pos S$-differentiable $\infty$-category.
  Then there is an adjunction with left adjoint $\P \colon \Fun{\infcat C}{\infcat D} \to \Exc{\sigma}{\infcat C}{\infcat D}$, right adjoint the inclusion $\inc \colon \Exc{\sigma}{\infcat C}{\infcat D} \to \Fun{\infcat C}{\infcat D}$, and unit $\p \colon \id \to \inc \circ \P$.
\end{theorem}

\begin{remark}
  Later we will also obtain \Cref{thm:finite_approximation}, a version of this theorem for shapes which are finite but not necessarily full.
\end{remark}

The main input in proving this theorem is the following lemma, whose proof is adapted from Rezk's streamlined proof (see \cite{Rez}) of the corresponding statement for ordinary Goodwillie calculus.

\begin{lemma} \label{lemma:cartesian_factorization}
  Let $\sigma \colon \gpos S \to \pos S$ be a reduced shape, $\infcat C$ a $\sigma$-nice $\infty$-category with a terminal object, and $\infcat D$ an $\infty$-category admitting sequential colimits and limits indexed by $\ginipos S$.
  Furthermore, let $F \colon \infcat C \to \infcat D$ be a functor and $D \colon \pos S \to \infcat C$ a $\sigma$-cocartesian diagram.
  Then there is a homotopy commutative diagram
  \[
  \begin{tikzcd}[row sep = 12, column sep = 20]
    F \circ D \ar{rr}{\t (F)} \drar & & \T (F) \circ D \\
     & E \urar &
  \end{tikzcd}
  \]
  such that $E \colon \pos S \to \infcat D$ is cartesian.
\end{lemma}

\begin{proof}
  The general strategy is to define, dependent on $D$, the following data:
  \begin{itemize}
  	\item a functor $\pos S \times \pos S \to \infcat C$ written $(s, t) \mapsto D_s(t)$ and $D' \colon \pos S \to \Fun{\pos S}{\infcat C},\; t \mapsto D_\blank(t)$ after currying, 
  	\item a natural transformation $\alpha \colon D' \to {\Ext} \circ {\Pad} \circ D$ (which has the pointwise form $\alpha \colon D_s(t) \to D(t) \join s$),
  	\item and a natural transformation $\beta \colon D \to {\Res{\ini}} \circ D' = D_\ini(\blank)$
  \end{itemize}
  such that
  \begin{enumerate}
  	\item the composition $({\Res{\ini}} \circ \alpha) \beta \colon D \to {\Res{\ini}} \circ {\Lan{\sigma}} \circ {\Pad} \circ D$ is homotopic to $\tau_{\sigma} \circ D$ (where $\tau_{\sigma}$ is as in \Cref{def:T}),
  	\item and for each $s \in \ginipos S$ the diagram $F \circ {\Res{s}} \circ D' = {\Res{s}} \circ (F \;\circ) \circ D' \colon \pos S \to \infcat D$, which is more explicitly given by $t \mapsto F(D_s(t))$, is cartesian.
  \end{enumerate}
  Assuming this exists, we obtain the homotopy commutative diagram
  \[
  \begin{tikzcd}[column sep = 20]
     & F \circ D \dlar[bend right = 12][swap]{\beta} \dar{\tau_{\sigma}} \ar[bend left = 40]{dddr}{\t (F)} &[-12] \\
    F \circ {\Res{\ini}} \circ D' \rar{\alpha} \dar[equal] & F \circ {\Res{\ini}} \circ {\Ext} \circ {\Pad} \circ D \dar[equal] & \\
    {\Res{\ini}} \circ (F \;\circ) \circ D' \rar{\alpha} \dar & {\Res{\ini}} \circ (F \;\circ) \circ {\Ext} \circ {\Pad} \circ D \dar & \\
    \lim{\ginipos S} \circ {\Res{\ginipos S}} \circ (F \;\circ) \circ D' \rar{\alpha} & \lim{\ginipos S} \circ {\Res{\ginipos S}} \circ (F \;\circ) \circ {\Ext} \circ {\Pad} \circ D \rar[equal] & \T(F) \circ D
  \end{tikzcd}
  \]
  which proves the claim by setting $E \defeq \lim{\ginipos S} \circ {\Res{\ginipos S}} \circ (F \;\circ) \circ D'$ and noting that $E$ is cartesian by \Cref{lemma:pointwise_cartesian} and the assumption that ${\Res{s}} \circ (F \;\circ) \circ D'$ is cartesian for each $s \in \ginipos S$.
  
  Now for the construction.
  We will use the following idea: let $s \in \ginipos S$, and assume that, for all $t \in \pos S$, applying $D_s$ to the morphism $t \to t \cop s$ yields an equivalence.
  Then $F \circ D_s$ is cartesian by \Cref{cor:coprod_cartesian}, which can be found after this proof (this is where we essentially need that $\pos S$ is a poset).
  This suggests to try to define $D_s(t) = D(t \cop s)$; since the functor $(\blank \cop s)$ is idempotent (as $\ginipos S$ is a poset), this implies the needed condition.
  However, it is not clear how one should then define the natural transformation $D(t \cop s) \to D(t) \join s$.
  So, our strategy is to construct $D_s(t)$ in a way that comes with this map but, in good cases, still computes $D(t \cop s)$.
  
  The rough idea is to set \[ D_s(t) \defeq (\Ext \Res{\sigma}) (D \circ (t \cop \blank)) (s) \]
  which will come equipped with a map $D_\blank(t) \to D(t) \join \blank$ since the latter is defined as $(\Ext \Pad)(D(t))$ which is terminal in the correct sense (see \Cref{rem:Pad}).
  Furthermore, for $D_\blank(t)$ to compute $D(t \cop \blank)$, we precisely need that $D \circ (t \cop \blank)$ is $\sigma$-cocartesian, which was the statement of \Cref{lemma:shifted_is_cocartesian}.
  
  More formally, to also obtain functoriality in $t$, we consider the diagram
  \[
  \begin{tikzcd}
    \pos S \rar{p'} \ar[bend right = 15]{drrr}{D} & \Fun{\pos S}{\infcat C} \rar{\Res{\sigma}} & \Fun{\gpos S}{\infcat C} \ar{rr}{\id}[swap, name = U]{} \drar[swap]{\Res{\genini}} & & \Fun{\gpos S}{\infcat C} \rar{\Ext} & \Fun{\pos S}{\infcat C} \\
     & & & \infcat C \urar[swap]{\Pad} \ar[from = U, Rightarrow, shorten < = 5, shorten > = 5]{}{\eta} & &
  \end{tikzcd}
  \]
  where $p'$ is given by $t \mapsto D \circ (t \cop \blank)$ and $\eta$ is the unit of the adjunction $\Res{\genini} \dashv \Pad$.
  The upper composition gives the desired functor $D' \colon \pos S \to \Fun{\pos S}{\infcat C}$ and $\eta$ the natural transformation $\alpha$, using that the left triangle commutes as $t \cop \sigma(\genini) = t$.
  
  Furthermore, let $\beta$ be given by the natural transformation in the diagram
  \[
  \begin{tikzcd}
    \pos S \dar[swap]{p'} \ar[bend left = 20]{ddrrr}{D} & & &[15] \\
    \Fun{\pos S}{\infcat C} \dar[swap]{\Res{\sigma}} & & & \\
    \Fun{\gpos S}{\infcat C} \drar[swap]{\Ext} \ar{rr}{\id}[swap, name = T]{} & & \Fun{\gpos S}{\infcat C} \rar{\Res{\genini}} & \infcat C \\
     & \Fun{\pos S}{\infcat C} \urar[swap]{\Res{\sigma}} \ar[bend right = 20]{urr}[swap]{\Res{\ini}} \ar[from = T, Rightarrow, shorten < = 5, shorten > = 5]{}{\epsilon} & &
  \end{tikzcd}
  \]
  where $\epsilon$ is the unit of the adjunction $\Ext \dashv \Res{\sigma}$.
  
  To show that $({\Res{\ini}} \circ \alpha) \beta$ is homotopic to $\tau_{\sigma} \circ D$, write $p = {\Res{\sigma}} \circ p'$ and consider the diagram
  \[
  \begin{tikzcd}
    \Res{\genini} \Res{\sigma} \Ext p \dar[equal] &[-17] \Res{\genini} p \lar[swap]{\beta} \rar{\id} \dar &[-40] \Res{\genini} p \\
    \Res{\ini} \Ext p \dar[swap]{\alpha} & \Res{\genini} (\Pad \Res{\genini}) p \rar[equal] \dar & (\Res{\genini} \Pad) \Res{\genini} p \uar[swap]{\eq} & \\
    \Res{\ini} \Ext \Pad \Res{\genini} p \rar[equal] & \Res{\genini} \Res{\sigma} \Ext \Pad \Res{\genini} p &
  \end{tikzcd}
  \]
  where all maps (apart from the four identities) are given by the (co)units of the adjunctions $\Res{\genini} \dashv \Pad$ and $\Ext \dashv \Res{\sigma}$, and the right vertical morphism is an equivalence since the inclusion $\set \genini \subseteq \gpos S$ is fully faithful.
  The right square commutes up to homotopy by one of the triangle identities and the left square commutes up to homotopy since the two sides are just the two possible horizontal compositions of a pair of natural transformations.
  Now note that the composition along the left of the diagram is just $({\Res{\ini}} \circ \alpha) \beta$ and the one along the right is $\tau_{\sigma} \circ D$.
  
  The only thing left to show is that $D_s(\blank) = {\Res{s}} \circ D'$ is equivalent to $D(\blank \cop s) = {\Res{s}} \circ p'$.
  This then implies that the morphism $D_s(t) \to D_s(t \cop s)$ is an equivalence for all $t \in \pos S$, as $t \cop s \to (t \cop s) \cop s$ is the identity.
  We will actually even show that the map
  \[ D' = \Lan{\sigma} \Res{\sigma} p' \xlongto{\epsilon} p' \]
  is an equivalence, where $\epsilon$ is the counit of the adjunction $\Ext \dashv \Res{\sigma}$.
  For this it is enough to show that $\epsilon$ evaluated at $p'(t) = D \circ (t \cop \blank)$ is an equivalence for all $t \in \pos S$.
  But this is equivalent to $D \circ (t \cop \blank)$ being $\sigma$-cocartesian, which was the statement of \Cref{lemma:shifted_is_cocartesian}.
\end{proof}

\begin{lemma} \label{cor:coprod_cartesian}
  Let $\pos I$ be a poset that admits finite coproducts, $\ini$ its initial object, $i$ an element of $\ginipos I$, $\infcat C$ an $\infty$-category that admits limits indexed by $\ginipos I$, and $D \colon \pos I \to \infcat C$ a diagram.
  Furthermore, assume that $D$ applied to the morphism $j \to j \cop i$ is an equivalence for all $j \in \pos I$.
  Then $D$ is cartesian.
\end{lemma}

\begin{proof}
  Let $p_i \colon \pos I \to i \slice \ginipos I$ be the functor given by $j \mapsto j \cop i$.
  Note that the restriction $\restrict {p_i} {\ginipos I}$ is left adjoint to the projection ${\pr} \colon i \slice \ginipos I \to \ginipos I$ by the universal property of the coproduct.
  Additionally, let $\eta \colon \id \to {\inc} \circ {\pr} \circ p_i$ be the natural transformation of functors $\pos I \to \pos I$ coming from $j \le j \cop i$ (this restricts to the unit of the above adjunction on $\ginipos I$).
  
  Now consider the diagram $D' \colon \cone{(i \slice \ginipos I)} \to \infcat C$ given by the composition of $D \circ \inc \circ \pr$ and the functor $\cone{(i \slice \ginipos I)} \to i \slice \ginipos I$ given by $i$ at the cone point and the identity otherwise.
  From this, we obtain the diagram
  \[
  \begin{tikzcd}
  D(\ini) \rar{\eq}[swap]{\eta} \dar & (D \circ {\inc} \circ {\pr} \circ p_i)(\ini) \dar &[25] (D \circ \inc \circ \pr)(i) \lar[equal] \dar{\eq} \\
  \lim{\ginipos I} (D \circ \inc) \rar{\eq}[swap]{\eta} & \lim{\ginipos I} (D \circ {\inc} \circ {\pr} \circ p_i \circ \inc) & \lim{i \slice \ginipos I} (D \circ \inc \circ \pr) \lar{\left( p_i \circ \inc \right)^*}[swap]{\eq}
  \end{tikzcd}
  \]
  which commutes up to homotopy by (the dual of) \Cref{lemma:functors_and_map_from_colim} (by \Cref{lemma:htpy_terminal_admitting}, limits indexed by the poset $i \slice \ginipos I$ exist since it has an initial object).
  The left horizontal morphisms are equivalences as $D \circ \eta$ is an equivalence by assumption.
  Furthermore, the bottom right horizontal map is an equivalence since $p_i \circ \inc$ is left adjoint and hence homotopy initial, and the vertical morphism on the right is an equivalence by \Cref{lemma:terminal_object_colimit} since $i$ is an initial object of $i \slice \ginipos I$.
  Hence the vertical map on the left is an equivalence, which was the claim.
\end{proof}

The rest of the proof of \Cref{thm:full_approximation} is analogous to the one given by Lurie in \cite[Theorem~6.1.1.10]{LurHA}.
However, for completeness' sake, we still recount it here (with a bit more details).
We need a few more lemmas:

\begin{lemma} \label{lemma:P_is_excisive}
  Let $\sigma \colon \gpos S \to \pos S$ be a reduced shape, $\infcat C$ a $\sigma$-nice $\infty$-category with a terminal object, and $\infcat D$ an $\pos S$-differentiable $\infty$-category.
  Then, for any functor $F \colon \infcat C \to \infcat D$, the functor $\P(F) \colon \infcat C \to \infcat D$ is $\sigma$-excisive.
\end{lemma}

\begin{proof}
  Let $D \colon \pos S \to \infcat C$ be a $\sigma$-cocartesian diagram.
  By definition, the diagram $\P(F) \circ D$ is given by the colimit of
  \begin{equation} \label[diagram]{diag:P(F)D}
    F \circ D \longto \T(F) \circ D \longto (\T)^2(F) \circ D \longto \cdots
  \end{equation}
  where the maps are given by $\t$.
  By \Cref{lemma:cartesian_factorization}, each of these maps
  \[ \t((\T)^n(F)) \circ D \colon (\T)^n(F) \circ D \longto (\T)^{n+1}(F) \circ D \]
  factors, up to homotopy, through a cartesian diagram $E_n \colon \pos S \to \infcat D$.
  The resulting sequence of morphisms
  \[ F \circ D \longto E_0 \longto \T(F) \circ D \longto E_1 \longto (\T)^2(F) \circ D \longto \cdots \]
  defines a sequential diagram (by \Cref{lemma:sequential_diagrams}).
  Restricting along the inclusion $2 \NN \subset \NN$, we obtain a diagram where each morphism is homotopic to the corresponding one in \cref{diag:P(F)D} and which is thus, by again \Cref{lemma:sequential_diagrams}, equivalent to it.
  On the other hand, restricting along the inclusion $2 \NN + 1 \subset \NN$ yields a diagram $E$ of the form
  \[ E_0 \longto E_1 \longto E_2 \longto \cdots \]
  whose colimit is cartesian by \Cref{lemma:pointwise_cartesian}, as each $E_n$ is cartesian and $\infcat D$ is $\pos S$\=/differentiable.
  Since both inclusions are homotopy terminal, the colimit of the $E_n$ is equivalent to the one of \cref{diag:P(F)D}, which finishes the proof.
\end{proof}

\begin{lemma} \label{lemma:P(t)_is_equiv}
  Let $\sigma \colon \gpos S \to \pos S$ be a full shape, $\infcat C$ a $\sigma$-nice $\infty$-category with a terminal object, and $\infcat D$ an $\pos S$-differentiable $\infty$-category.
  Then, for any functor $F \colon \infcat C \to \infcat D$, applying $\P$ to $\t(F) \colon F \to \T (F)$ yields an equivalence.
\end{lemma}

\begin{proof}
  Recall that $\t(F)$ is defined as the composition of the transformation $F \circ \tau_{\sigma}$, which is an equivalence as $\sigma$ is full, and the upper horizontal map in the homotopy commutative diagram 
  \[
  \begin{tikzcd}
    {\Res{\ini}} \circ (F \;\circ) \circ \Ext \circ \Pad \rar \dar[equal] & \lim[s]{\ginipos S} \circ {\Res{\ginipos S}} \circ (F \;\circ) \circ \Ext \circ \Pad \rar[equal] \dar{\eq} &[-17] \T (F) \\
    F \circ (\blank \join \ini) \rar & \lim[s]{s \in \ginipos S} (F \circ (\blank \join s)) &
  \end{tikzcd}
  \]
  which we obtain from (the dual of) \Cref{lemma:limits_and_currying}.
  Hence it is enough to consider the lower horizontal map in the above diagram.
  Since, by \Cref{lemma:T_and_P_elementary}, the functor $\P$ preserves limits indexed by $\ginipos S$, applying it to this map yields, by \Cref{lemma:preservation_and_can}, the upper triangle in the following homotopy commutative diagram in $\Fun{\infcat C}{\infcat D}$
  \begin{equation} \label[diagram]{diag:P(t)_is_equiv}
  \begin{tikzcd}
     & \P \left( \lim[s]{s \in \ginipos S} (F \circ (\blank \join s)) \right) \dar{\eq} \\
    \P(F \circ (\blank \join \ini)) \rar \urar[end anchor = south west] \dar[swap]{\eq} & \lim[s]{s \in \ginipos S} \P(F \circ (\blank \join s)) \dar{\eq} \\
    \P(F) \circ (\blank \join \ini) \rar & \lim[s]{s \in \ginipos S} (\P(F) \circ (\blank \join s))
  \end{tikzcd}
  \end{equation}
  where the lower two vertical equivalences are provided by \Cref{lemma:T_and_P_elementary}.
  This uses that $(\blank \join s)$ preserves left Kan extension along $\sigma$ by \Cref{lemma:preserving_Kan_extensions}, for which we in turn use \Cref{lemma:join_preserves} to see that $(\blank \join s)$ preserves terminal objects and, for $s' \in \pos S$, colimits indexed by $\sigma \slice s'$ (since $\sigma$ is full and hence $\sigma \slice s'$ has an initial object and is contractible).
  
  But the lower horizontal map in \cref{diag:P(t)_is_equiv} is an equivalence since $\P(F)$ is $\sigma$-excisive by \Cref{lemma:P_is_excisive} and, for any $X \in \infcat C$, the diagram
  \[ (X \join \blank) = (\Ext \Pad)(X) \colon \pos S \longto \infcat C \]
  is $\sigma$-cocartesian by \Cref{lemma:extensions_are_cocartesian} (here we implicitly use that $\Res{X}$ preserves limits and again \Cref{lemma:preservation_and_can}).
  This finishes the proof.
\end{proof}

\begin{lemma} \label{lemma:P(p)_is_equiv}
  Let $\sigma \colon \gpos S \to \pos S$ be a full shape, $\infcat C$ a $\sigma$-nice $\infty$-category with a terminal object, and $\infcat D$ an $\pos S$-differentiable $\infty$-category.
  Then, for any functor $F \colon \infcat C \to \infcat D$, applying $\P$ to $\p(F) \colon F \to \P (F)$ yields an equivalence.
\end{lemma}

\begin{proof}
  By \Cref{lemma:T_and_P_elementary}, the functor $\P$ preserves sequential colimits.
  Hence we have, by \Cref{lemma:preserving_structure_map}, that the map $\P(\p(F))$ is an equivalence if and only if the structure map from $\P(F)$ to the colimit of the diagram
  \[ \P(F) \xlongto{\P(\t)} (\P \T)(F) \xlongto{\P(\t \circ \T)} (\P \T \T)(F) \longto \dots \]
  is.
  But the latter follows from \Cref{lemma:contractible_colimit_over_equivalences} since each of the maps $\P({\t} \circ (\T)^n)$ is an equivalence by \Cref{lemma:P(t)_is_equiv}.
\end{proof}

\begin{lemma} \label{lemma:right_adjoint_right_inverse}
  Let $C$ and $D$ be objects of a 2-category, $l \colon C \to D$ and $r \colon D \to C$ morphisms, and $\eta \colon \id[C] \to r \circ l$ a 2-morphism.
  Assume that there is a 2-isomorphism $l \circ r \iso \id[D]$, and that both $l \circ \eta$ and $\eta \circ r$ are 2-isomorphisms.
  Then there is an adjunction $l \dashv r$ with unit $\eta$.
\end{lemma}

\begin{proof}
  This is \cite[Lemma B.4.2 and Remark B.4.3]{RV}.
\end{proof}

We are now ready to complete the proof of the main theorem:

\begin{proof}[Proof of \Cref{thm:full_approximation}]
  We want to apply \Cref{lemma:right_adjoint_right_inverse} (in the homotopy 2-category of $\infty$-categories) to $\P \colon \Fun{\infcat C}{\infcat D} \to \Exc{\sigma}{\infcat C}{\infcat D}$ and the inclusion $\inc \colon \Exc{\sigma}{\infcat C}{\infcat D} \to \Fun{\infcat C}{\infcat D}$.
  First note that $\P$ actually lands in $\Exc{\sigma}{\infcat C}{\infcat D}$ by \Cref{lemma:P_is_excisive}.
  Now, by \Cref{lemma:T_and_P_elementary}, the transformation $\p \colon \id \to \inc \circ \P$ precomposed with $\inc$ is an equivalence.
  This also implies that $\P \circ \inc \eq \id$ since $\Exc{\sigma}{\infcat C}{\infcat D}$ is a full subcategory.
  Furthermore the transformation $\P \circ \p$ is an equivalence by \Cref{lemma:P(p)_is_equiv}.
\end{proof}

\section{Maps between approximations} \label{section:maps}

As was mentioned in \Cref{rem:P_not_functorial}, the construction of the universal $\sigma$-excisive approximation is not functorial in $\sigma$.
However, by its universal property, we will get a map $\P(F) \to \P[\tau](F)$ if $\P[\tau](F)$ is $\sigma$-excisive.
So, in this section, we will study what maps between (pre)shapes tell us about the relationship of the corresponding notions of excision.

One special case is the map $\cubeinc{n} \to \cubeinc{n+1}$ induced by the inclusion of $\finset{n-1}$ into $\finset{n}$.
This will correspond to the classical fact that $(n-1)$-excisive implies $n$-excisive (cf.\ \cite[Proposition 3.2]{Goo92} or \cite[Corollary 6.1.1.14]{LurHA}).
The first subsection will focus on a generalization of this to more general (pre)shapes.
However, under some conditions, a map $\sigma \to \tau$ of preshapes can also tell us that $\tau$-excisive implies $\sigma$-excisive.
This does not have an analogue in classical Goodwillie calculus (at least in the form we prove) and will be explored in the second subsection.

The names chosen for these concepts, indirect respectively direct maps, come from their effect on the universal excisive approximations.
The mnemonic is that a direct map $\sigma \to \tau$ induces a map $\P \to \P[\tau]$ and the other way around for an indirect map.

\begin{remark}
  In \Cref{rem:P_not_functorial}, we saw that a map of (reduced) preshapes $(f, \c f) \colon (\sigma \colon \gpos S \to \pos S) \to (\tau \colon \gpos T \to \pos T)$ induces maps
  \[ \T(X) = \lim{s \in \ginipos S} (X \join_\sigma s) \longto \lim{s \in \ginipos S} (X \join_\tau f(s)) \longfrom \lim{t \in \ginipos T}(X \join_\tau t) = \T[\tau](X) \]
  which do not combine into a map between $\T$ and $\T[\tau]$.
  However, if one of these two maps were an equivalence for all $X$, then we would get a map in one direction.
  This can be seen as a motivation for there being two conditions, direct and indirect, one for each possible direction of the resulting map (and the conditions we give are basically chosen such that they guarantee one of the above maps to be an equivalence).
\end{remark}

\subsection{Indirect maps} \label{sec:indirect}

\begin{definition}
  Let $\sigma \colon \gpos S \to \pos S$ and $\tau \colon \gpos T \to \pos T$ be preshapes.
  A map $(f, \gen f) \colon \sigma \to \tau$ of preshapes is \emph{indirect} if, for all $s \in \pos S$, the induced functor
  \[ \c f_s \colon \sigma \slice s \longto \tau \slice f(s), \quad \left( \c s, \sigma(\c s) \xto{l} s \right) \longmapsto \left( \gen f(\c s), \tau(\gen f(\c s)) = f(\sigma(\c s)) \xto{f(l)} f(s) \right) \]
  is homotopy terminal.
\end{definition}

\begin{lemma} \label{lemma:inclusion_of_cubes_indirect}
  Let $n \le m$ be elements of $\NN$.
  Then the following is an indirect map of preshapes:
  \[
  \begin{tikzcd}
    \Cubeini{n} \rar{\c f} \dar[swap]{\cubeinc{n}} & \Cubeini{m} \dar{\cubeinc{m}} \\
    \Cube{n} \rar{f} & \Cube{m}
  \end{tikzcd}
  \]
  where $f$ and $\c f$ are induced by the inclusion of $\finset{n - 1}$ into $\finset{m - 1}$.
\end{lemma}

\begin{proof}
  That it is a map of preshapes is clear.
  For indirectness we note that, for $S \in \Cube{n}$, the induced map $\cubeinc n \slice S \to \cubeinc m \slice f(S)$ is even an isomorphism.
\end{proof}

\begin{lemma} \label{lemma:indirect_preserves_cocartesian}
  Let $\sigma \colon \gpos S \to \pos S$ and $\tau \colon \gpos T \to \pos T$ be preshapes, $(f, \gen f) \colon \sigma \to \tau$ an indirect map, and $\infcat C$ a $\sigma$-nice and $\tau$-nice $\infty$-category.
  Then, for any $\tau$-cocartesian diagram $D \colon \pos T \to \infcat C$, the diagram $\Res{f}(D)$ is $\sigma$-cocartesian.
\end{lemma}

\begin{proof}
  Consider the map $(f, \gen f)$ as the identity transformation in the diagram
  \[
  \begin{tikzcd}[sep = 30]
    \gpos S \rar{\gen f} \dar[swap]{\sigma} & \gpos T \dar{\tau} \dlar[Rightarrow, shorten < = 13, shorten > = 13][swap]{\id} \\
    \pos S \rar{f} & \pos T
  \end{tikzcd}
  \]
  Applying $\Fun{\blank}{\infcat C}$ and taking the mate gives us a transformation $\mate \id \colon \Lan{\sigma} \Res{\gen f} \to \Res{f} \Lan{\tau}$, which is an equivalence by \Cref{lemma:kan_mate} and assumption.
  By \Cref{lemma:mate_and_units} we obtain a homotopy commutative diagram
  \[
  \begin{tikzcd}
    \Lan{\sigma} \Res{\gen f} \Res{\tau} \rar{\eq}[swap]{\mate \id} \dar[equal] & \Res{f} \Lan{\tau} \Res{\tau} \dar{\epsilon_\tau} \\
    \Lan{\sigma} \Res{\sigma} \Res{f} \rar{\epsilon_\sigma} & \Res{f}
  \end{tikzcd}
  \]
  where $\epsilon_\sigma$ and $\epsilon_\tau$ are the counits of the respective adjunctions.
  Now note that $\epsilon_\tau$ evaluated at $D$ is an equivalence since $D$ is $\tau$-cocartesian.
  Hence $\epsilon_\sigma$ evaluated at $\Res{f}(D)$ is also an equivalence, as we wanted to show.
\end{proof}

\begin{proposition} \label{prop:indirect}
  Let $\sigma \colon \gpos S \to \pos S$ be a preshape, $\tau \colon \gpos T \to \pos T$ a shape, $\infcat C$ a $\sigma$-nice and $\tau$-nice $\infty$-category, and $\infcat D$ an $\infty$-category that admits limits indexed both by $\ginipos S$ and by $\ginipos T$.
  Furthermore, assume that there is an indirect map $(f, \gen f) \colon \sigma \to \tau$.
  Then each $\sigma$-excisive functor $F \colon \infcat C \to \infcat D$ is also $\tau$-excisive.
\end{proposition}

\begin{proof}
  Let $D \colon \pos T \to \infcat C$ a $\tau$-cocartesian diagram.
  We need to prove that $F \circ D$ is cartesian.
  For this we use the functor $p \colon \pos T \times \pos S \to \pos T$ given by $p(t, s) = t \cop f(s)$ to transport diagrams indexed by $\pos T$ to diagrams indexed by $\pos S$.
  
  By \Cref{lemma:shifted_is_cocartesian} and $\tau$ being a shape, the diagram $D \circ (t \cop \blank)$ is $\tau$-cocartesian for any $t \in \pos T$.
  Then, by \Cref{lemma:indirect_preserves_cocartesian}, the diagram $D \circ (t \cop \blank) \circ f$ is is $\sigma$-cocartesian.
  In particular this implies that $F \circ D \circ p \colon \pos T \times \pos S \to \infcat D$ is cartesian when restricted to any $\set t \times \pos S$.
  Then, by \Cref{lemma:cartesian_products}, we have that $F \circ D \circ p$ is already a limit diagram itself.
  Now, using that, by \Cref{lemma:gini_cop_htpy_initial}, the restriction $p|_{\gini{(\pos T \times \pos S)}} \colon \gini{(\pos T \times \pos S)} \to \gini{\pos T}$ is homotopy initial, we obtain, by \Cref{lemma:preserves_cartesian}, that $F \circ D$ is a limit diagram, as we wanted to show.
\end{proof}

Together with \Cref{lemma:inclusion_of_cubes_indirect} this implies the classical statement that $n$-excisive implies $m$-excisive for $n \le m$:

\begin{corollary} \label{cor:n-excisive_implies_m-excisive}
  Let $n \le m$ be elements of $\NN$, $\infcat C$ an $\infty$-category that admits all finite colimits, and $\infcat D$ an $\infty$-category that admits limits indexed both by $\gini{\Cube{n}}$ and by $\gini{\Cube{m}}$.
  Then each $(n-1)$-excisive functor $F \colon \infcat C \to \infcat D$ is also $(m-1)$-excisive.
\end{corollary}

The following lemma will be needed later.

\begin{lemma} \label{lemma:full_sub_is_shape}
  Let $\gpos S \in \Pos$ and $\pos S, \pos T \in \Poscop$.
  Furthermore, let there be a commutative diagram of functors between posets
  \[
  \begin{tikzcd}
     & \gpos S \dlar[swap]{\sigma} \drar{\tau} & \\
    \pos S \ar{rr}{f} & & \pos T
  \end{tikzcd}
  \]
  such that $\tau$ is a shape and such that $f$ is full and preserves finite coproducts.
  Then $\sigma$ is also a shape that is finite if $\tau$ is, and $(f, \id[\gpos S]) \colon \sigma \to \tau$ is an indirect map of shapes.
\end{lemma}

\begin{proof}
  First note that $f$ is injective by \Cref{lemma:pos_full_implies_injective}, hence $\pos S$ is finite if $\pos T$ is.
  Furthermore, this gives us that $\inv\sigma(\ini_{\pos S}) = \inv\sigma(\inv f(\ini_{\pos T})) = \inv\tau(\ini_{\pos T})$ is non-empty (note that $\inv f(\ini_{\pos T}) = \set{\ini_{\pos S}}$ as $f$ is injective and preserves initial objects).
  In particular $\sigma$ is a preshape.
  
  We now need to show that, for all $s, t \in \pos S$ and $\c k \in \gpos S$ such that $\sigma(\c k) \le t \cop s$, the full subposet
  \[ \gpos{S}_{s,t,\c k}^{\sigma} = \{\c s \in \gpos S \mid \sigma(\c s) \le s \text{ and } \sigma(\c k) \le t \cop \sigma(\c s)\} \subseteq \gpos S \]
  is contractible.
  For this note that our assumptions imply
  \begin{align*}
  \gpos{S}_{s,t,\c k}^{\sigma}
  &= \{\c s \in \gpos S \mid f(\sigma(\c s)) \le f(s) \text{ and } f(\sigma(\c k)) \le f(t \cop \sigma(\c s)) \} \\
  &= \{\c s \in \gpos S \mid \tau(\c s) \le f(s) \text{ and } \tau(\c k) \le f(t) \cop \tau(\c s) \} \\
  &= \gpos{S}_{f(s),f(t),\c k}^{\tau}
  \end{align*}
  which is contractible as $\tau$ is a shape and $\sigma(\c k) \le t \cop s$ implies $\tau(\c k) \le f(t) \cop f(s)$.
  
  The tuple $(f, \id[\gpos S])$ is a map of shapes since we have, as noted above, that $\inv f(\ini_{\pos T}) = \set{\ini_{\pos S}}$.
  For indirectness we need that, for all $s \in \pos S$, the functor $\sigma \slice s \to \tau \slice f(s)$ induced by $\id[\gpos S]$ is homotopy terminal.
  We claim that it is even an isomorphism.
  For this, it is enough to show surjectivity since the functor in question is just the inclusion of one full subposet of $\gpos S$ into another.
  This surjectivity is equivalent to $\tau(\c s) \le f(s)$ implying $\sigma(\c s) \le s$ for all $\c s \in \gpos S$, which follows from the equality $f(\sigma(\c s)) = \tau(\c s)$ and $f$ being full.
\end{proof}

\subsection{Direct maps} \label{sec:direct}

\begin{definition}
  Let $\sigma \colon \gpos S \to \pos S$ and $\tau \colon \gpos T \to \pos T$ be preshapes.
  A map $(f, \gen f) \colon \sigma \to \tau$ of preshapes is \emph{direct} if $f$ is full and $f|_{\ginipos S} \colon \ginipos S \to \ginipos T$ is homotopy initial.
\end{definition}

\begin{remark}
  Clearly any map of preshapes $(f, \c f) \colon \sigma \to \tau$ such that $f$ is an isomorphism is direct.
  Let us now furthermore assume that $\tau$ is full.
  In this case one can see very clearly why $\tau$-excisive should imply $\sigma$-excisive: any $\sigma$-cocartesian diagram is also $\tau$-cocartesian.
  This is the case since, if we have a diagram $D \colon \pos S \to \infcat C$ such that $D \eq (\Lan{\sigma} \Res{\sigma}) (D)$, then also $D \eq (\Lan{\tau} \Lan{\c f} \Res{\sigma}) (D)$, and thus $D$ is $\tau$-cocartesian by \Cref{lemma:extensions_are_cocartesian}.
  This can be seen as further motivation for the following proposition.
\end{remark}

\begin{proposition} \label{prop:direct}
  Let $\sigma \colon \gpos S \to \pos S$ and $\tau \colon \gpos T \to \pos T$ be preshapes, $\infcat C$ a $\sigma$-nice and $\tau$-nice $\infty$-category, and $\infcat D$ an $\infty$-category that admits limits indexed both by $\ginipos S$ and by $\ginipos T$.
  Assume that $\tau$ is a shape or full and that there is a direct map $(f, \gen f) \colon \sigma \to \tau$.
  Then each $\tau$-excisive functor $F \colon \infcat C \to \infcat D$ is also $\sigma$-excisive.
\end{proposition}

\begin{proof}
  Let $D \colon \pos S \to \infcat C$ be a $\sigma$-cocartesian diagram.
  We need to show that $F \circ D$ is cartesian.
  For this we use $\Lan{f}$ to transport $D$ to a diagram indexed by $\pos T$.
  
  We have
  \[ \Lan{f}(D) \eq (\Lan{f} \Lan{\sigma} \Res{\sigma})(D) \eq (\Lan{\tau} \Lan{\c f} \Res{\sigma})(D) \]
  which is $\tau$-cocartesian.
  This follows from \Cref{lemma:extensions_are_cocartesian} if $\tau$ is full and from \Cref{lemma:shifted_is_cocartesian} if $\tau$ is a shape.
  In particular, we obtain that $F \circ \Ext[f] (D)$ is cartesian.
  Hence, by \Cref{lemma:preserves_cartesian}, the restriction $F \circ \Ext[f] (D) \circ f$ is cartesian (here we use that $f|_{\ginipos S}$ is homotopy initial).
  Now, since $f$ is full and hence fully faithful, we have that
  \[ F \circ \Ext[f] (D) \circ f = F \circ (\Res{f} \Ext[f]) (D) \eq F \circ D \]
  and thus that $F \circ D$ is cartesian, as we wanted to show.
\end{proof}

\begin{remark}
  The condition that $f$ is full is a bit stronger than actually required.
  We do not need that $(\Res{f} \Lan{f}) (D) \eq D$ for all $D \colon \pos S \to \infcat C$, but only for those that are $\sigma$-cocartesian.
  A weaker, combinatorial condition guaranteeing this can be formulated.
\end{remark}
\section{The structure of shapes} \label{section:shapes}

In this section we study various classes of shapes and their properties.
The results obtained here will be useful for the following sections.

\subsection{Full shapes}

In this subsection we consider shapes which are full as a functor.
The main result is that for every finite shape there is a full shape with the same excision properties.

\begin{notation}
  Let $f \colon \pos P \to \pos Q$ be a map of posets.
  We denote by $\im f \subseteq \pos Q$ the essential image of $f$, i.e.\ the poset with elements $f(\pos P)$ and partial order $\ile$ generated by the relation $q \ile q'$ if there are $p \in \inv f(q)$ and $p' \in \inv f(q')$ such that $p \le p'$.
\end{notation}

\begin{remark}
  Note that $\im f$ is a subposet of $\pos Q$ but does not need to be full.
  To distinguish the different partial orders in this situation we will use, as in the definition, the symbol $\ile$ for the one of $\im f$.
\end{remark}

\begin{notation} \label{def:im_preshape}
  Let $\sigma \colon \gpos S \to \pos S$ be a preshape.
  We write $\i \sigma \colon \im \sigma \to \pos S$ for the preshape occurring in the factorization
  \[
  \begin{tikzcd}
     & \im \sigma \drar{\i \sigma} & \\
    \gpos S \ar{rr}{\sigma} \urar & & \pos S \;.
  \end{tikzcd}
  \]
\end{notation}

\begin{remark}
  That $\i \sigma$ is again a preshape is immediate as $\ini_{\pos S} \in \im \sigma$.
\end{remark}

\begin{lemma} \label{lemma:im_shape_full}
  If $\sigma \colon \gpos S \to \pos S$ is a finite shape, then $\i \sigma \colon \im \sigma \to \pos S$ is full.
\end{lemma}

\begin{proof}
  We need to show that for any two elements $i$ and $j$ of $\im \sigma$ such that $i \le j$, we also have $i \ile j$.
  We do this by induction on
  \[ d(i, j) := \sup \set{ n \in \NN \mid \exists s_0, \dots, s_n \in \pos S \text{ such that } i = s_0 \lneq \dots \lneq s_n = j} \]
  which is finite since $\pos S$ is finite by assumption.
  
  If $d(i, j) = 0$, then $i = j$ and we are done. Otherwise, let $\c k \in \inv\sigma(i)$. Then, by the definition of a shape, the full subposet
  \[ \gpos{S}_{j,\ini,\c k} = \set*{\ci \in \gpos S \mid i \le \sigma(\ci) \le j} \subseteq \gpos S \]
  is contractible, in particular connected, as $\sigma(\c k) = i \le j = j \cop \ini$.
  Since $i \neq j$, the preimage $\inv\sigma(i)$ is a proper non-empty subset of $\gpos{S}_{j,\ini,\c k}$.
  Thus, there must be a morphism in $\gpos{S}_{j,\ini,\c k}$ with exactly one of target or source lying in $\inv\sigma(i)$.
  But if it were the target, there would be an element $\ci \in \gpos{S}_{j,\ini,\c k}$ such that $\sigma(\ci) \lneq i$, a contradiction.
  So there must be $\ci \in \inv\sigma(i)$ and $\c l \in \gpos{S}_{j,\ini,\c k} \setminus \inv\sigma(i)$ such that $\ci \le \c l$.
  In particular $i = \sigma(\ci) \ile \sigma(\c l) \le j$, hence it is enough to show that $\sigma(\c l) \ile j$.
  But as $i \neq \sigma(\c l)$ we have $d(\sigma(\c l), j) \lneq d(i, j)$, so the statement follows by induction.
\end{proof}

\begin{lemma}
  A finite shape $\sigma \colon \gpos S \to \pos S$ is full if and only if it is injective.
\end{lemma}

\begin{proof}
  The ``only if'' part follows from \Cref{lemma:pos_full_implies_injective}.
  For the other direction note that, if $\sigma$ is injective, then $\sigma \colon \gpos S \to \im \sigma$ is an isomorphism since the generating relation in \Cref{def:im_preshape} is, in this case, already transitive.
  Hence, by the factorization
  \[
  \begin{tikzcd}
  & \im \sigma \drar{\i \sigma} & \\
  \gpos S \ar{rr}{\sigma} \urar{\iso} & & \pos S \;
  \end{tikzcd}
  \]
  and $\i \sigma$ being full by \Cref{lemma:im_shape_full}, we obtain that $\sigma$ is full as well.
\end{proof}

The following result tells us that for each finite shape there is a full finite shape with the same excision properties.

\begin{proposition} \label{prop:finite_shape_eq_to_full}
  Let $\sigma \colon \gpos S \to \pos S$ be a finite shape.
  Then $\i \sigma \colon \im \sigma \to \pos S$ is again a finite shape (and full, by \Cref{lemma:im_shape_full}) and the map of shapes
  \begin{equation} \label[diagram]{diag:map_im_shape}
  \begin{tikzcd}
    \gpos S \rar{\gen \sigma} \dar[swap]{\sigma} & \im \sigma \dar{\i \sigma} \\
    \pos S \rar{\id} & \pos S
  \end{tikzcd}
  \end{equation}
  is both direct and indirect.
  In particular, a functor $F \colon \infcat C \to \infcat D$ between $\infty$-categories is $\sigma$-excisive if and only if it is $\i\sigma$-excisive (as long as $\infcat C$ is $\sigma$-nice (which implies $\i \sigma$-nice), and $\infcat D$ admits limits indexed by $\ginipos S$).
\end{proposition}

\begin{proof}
  For the first part we need to show that, for all $s, t \in \pos S$ and $\i k \in \im \sigma$ such that $\i \sigma(\i k) \le t \cop s$, the full subposet
  \[ \i{\pos{I}}_{s,t,\i k} \defeq \set{\ii \in \im \sigma \mid \i\sigma(\ii) \le s \text{ and } \i\sigma(\i k) \le t \cop \i\sigma(\ii)} \subseteq \im \sigma \]
  is contractible.
  Let $\c l \in \inv{\c\sigma}(\i k)$.
  Then, as $\sigma$ is a shape and $\sigma(\c l) \le t \cop s$, we know that the full subposet
  \[ \gpos{S}_{s,t,\c l} \defeq \set{\c s \in \gpos S \mid \sigma(\c s) \le s \text{ and } \i\sigma(\i k) = \sigma(\c l) \le t \cop \sigma(\c s)} \subseteq \gpos S \]
  is contractible.
  By commutativity of \cref{diag:map_im_shape}, we have that $\c \sigma$ induces a functor $g \colon \gpos{S}_{s,t,\c l} \to \i{\pos{I}}_{s,t,\i k}$.
  It is now enough to check that $g$ is a homotopy equivalence, in particular that it is homotopy initial, i.e.\ that for each $\ii \in \i{\pos{I}}_{s,t,\i k}$ the full subposet
  \[ \pos G_\ii \defeq g \slice \ii = \set{\c s \in \gpos{S}_{s,t,\c l} \mid g(\c s) \ile \ii} \subseteq \gpos{S}_{s,t,\c l} \]
  is contractible.
  By \Cref{lemma:im_shape_full}, the condition $g(\c s) \ile \ii$ is equivalent to $\sigma(\c s) = \i\sigma(g(\c s)) \le \i\sigma(\ii)$.
  Hence, using that $\i\sigma(\ii) \le s$, we obtain $\pos G_\ii = \gpos S_{\i\sigma(\ii),t,\c l}$, which is contractible since $\sigma$ is a shape and $\sigma(\c l) = \i\sigma(\i k) \le t \cop \i\sigma(\ii)$ by definition of $\i{\pos{I}}_{s,t,\i k}$.
  
  That $(\id, \i\sigma)$ is direct is clear.
  For indirectness we have to show that for any $s \in \pos S$ the functor $\c \sigma_s \colon \sigma \slice s \to \i\sigma \slice s$ is homotopy terminal, i.e.\ that for any $\i s \in \i\sigma \slice s$ the category $\i s \slice \c\sigma_s$ is contractible.
  This category $\i s \slice \c\sigma_s$ can be identified with the full subposet
  \[ \pos H_{s, \i s} = \set{\c t \in \gpos S \mid \sigma(\c t) \le s \text{ and } \i s \ile \c\sigma(\c t)} \subseteq \gpos S \;. \]
  Let $\c k \in \inv{\c \sigma}(\i s)$.
  Then we have, again using \Cref{lemma:im_shape_full}, that $\pos H_{s, \i s}$ is just $\gpos S_{s,\ini,\c k}$ and hence contractible since $\sigma$ is a shape and $\sigma(\c k) = \i\sigma(\i s) \le s = \ini \cop s$ by definition of $\i s$.
\end{proof}

As a corollary we obtain the following version of \Cref{thm:full_approximation} for shapes which are finite but not necessarily full.

\begin{corollary} \label{thm:finite_approximation}
  Let $\sigma \colon \gpos S \to \pos S$ be a finite shape, $\infcat C$ a $\sigma$-nice $\infty$-category with a terminal object, and $\infcat D$ an $\pos S$-differentiable $\infty$-category.
  Then there is an adjunction with left adjoint $\P[\i \sigma] \colon \Fun{\infcat C}{\infcat D} \to \Exc{\i \sigma}{\infcat C}{\infcat D} = \Exc{\sigma}{\infcat C}{\infcat D}$, right adjoint the inclusion $\inc \colon \Exc{\sigma}{\infcat C}{\infcat D} \to \Fun{\infcat C}{\infcat D}$, and unit $\p[\i \sigma] \colon \id \to \inc \circ \P[\i \sigma]$.
\end{corollary}

\subsection{Free shapes} \label{section:free_shapes}

In this subsection we study shapes $\sigma \colon \gpos S \to \pos S$ such that $\pos S$ is freely generated by $\gpos S$ under taking coproducts.
These turn out to be useful since it is easy to map out of them.
We show that, if $\gpos S$ is finite, being excisive with respect to such a ``free'' shape $\sigma$ is, for some $n$, equivalent to being $n$-excisive.

\begin{notation}
  Let $\pos P$ be a poset and $M \subseteq \pos P$ a subset.
  We write
  \begin{enumerate}[label=\alph*)]
    \item $\down M \defeq \set{x \in \pos P \mid \exists y \in M \colon x \le y}$ for the \emph{down-set} of $M$.
    \item $\copcomp(\pos P) \defeq \set{N \subseteq \pos P \mid N = \down N}$ for the \emph{down-set lattice}, i.e.\ the poset of downward closed subsets ordered by inclusion.
    \item $\copcompmap[\pos P] \colon \pos P \to \copcomp(\pos P)$ for the canonical functor given by $x \mapsto \down {\set x}$.
  \end{enumerate}
\end{notation}

It is well-known (cf.\ \cite[Examples 2.6 (3)]{DP}) that $\copcomp(\pos P)$ is a complete lattice, in particular that it has all coproducts, and that these coproducts are given by taking the union of subsets.
Actually, the map $\copcompmap[\pos P]$ is even the universal map from $\pos P$ to a poset with all coproducts.

The following definition introduces a similar construction which adds coproducts to posets that already have initial objects.
Just taking $\copcomp$ would add a new initial object (the empty set) which we do not want.

\begin{notation}
  Let $\pos P \in \Posini$.
  We write $\copcompi(\pos P) \subsetneq \copcomp(\pos P)$ for the full subposet of non-empty subsets.
  Noting that $\copcompmap[\pos P]$ uniquely factors as $\pos P \to \copcompi(\pos P) \subsetneq \copcomp(P)$, we will also, by abuse of notation, write $\copcompmap[\pos P] \colon \pos P \to \copcompi(\pos P)$ for the first map in this factorization.
\end{notation}

\begin{remark}
  Note that $\copcompi(\pos P)$ still has all coproducts: non-empty ones are again given by taking unions of subsets and the initial object $\set{\ini_{\pos P}}$ constitutes an empty one.
\end{remark}

\begin{lemma} \label{rem:cubes_are_free}
  For $\pos P = \Cubeini n$, there is a canonical isomorphism of posets under $\Cubeini n$ between $\copcompmap[\Cubeini n] \colon \Cubeini n \to \copcompi(\Cubeini n)$ and $\cubeinc{n} \colon \Cubeini n \to \Cube n$.
\end{lemma}

\begin{proof}
  This follows directly from the definitions.
\end{proof}

The following lemma now makes precise what we claimed before about the universality of this construction.

\begin{lemma}
  The construction $\copcompi$ extends to a functor $\Posini \to \Poscop$ by defining, for $f \in \Hom[\Posini] {\pos P} {\pos P'}$, the induced map $\copcompi(f)$ to be given by $M \mapsto \down{f(M)}$.
  
  Furthermore $\copcompi$ is left adjoint to the forgetful functor $U \colon \Poscop \to \Posini$.
  More explicitly, for all posets $\pos P \in \Posini$, $\pos Q \in \Poscop$, and any $f \in \Hom[\Posini] {\pos P} {U(\pos Q)}$, the functor $u_f \colon \copcompi(\pos P) \to \pos Q$ given by $M \mapsto \coprod_{m \in M} f(m)$ is the unique element of $\Hom[\Poscop]{\copcompi(\pos P)} {\pos Q}$ such that
  \[
  \begin{tikzcd}
     & \pos P \dlar[swap]{\copcompmap[\pos P]} \drar{f} & \\
    \copcompi(\pos P) \ar{rr} & & Q
  \end{tikzcd}
  \]
  commutes, and the assignment $f \mapsto u_f$ is natural in both $\pos P$ and $\pos Q$.
  
  Moreover $\copcompmap$ is a natural transformation $\id \to U \circ \copcompi$ of functors $\Posini \to \Posini$.
\end{lemma}

\begin{proof}
  First note that it is clear that $\copcompi(f)$ is a morphism in $\Poscop$.
  For functoriality of $\copcompi$ we furthermore need that, for $M \in \copcompi(\pos P)$ and $f \colon \pos P \to \pos Q$ and $g \colon \pos Q \to \pos R$ maps of posets, we have $M = \down \id(M)$ and $\down (g \circ f) (M) = \down g(\down f(M))$.
  The former statement, as well as the inclusion $\subseteq$ of the latter, are clear.
  For the other inclusion let $x \in \pos R$ and note that $x \le g(y)$ for some $y \in \down f(M)$ implies that there is $z \in M$ such that $y \le f(z)$.
  Hence $x \le g(f(z))$ and thus $x \in \down g(f(M))$.
  
  For naturality of $\copcompmap$ we need that, for $x \in \pos P$ and $f \in \Hom[\Posini]{\pos P}{\pos Q}$, we have $\down {f(\down {\set x})} = \down {\set{f(x)}}$, which follows from the same argument as the functoriality.
  
  For the adjunction, it is clear that $u_f$ is a functor and that $u_f \circ \copcompmap[\pos P] = f$ as we have, for $x \in \pos P$, that $\coprod_{y \le x} f(y) = f(x)$ (this uses that in a poset the coproduct of a set of objects with a maximum is that maximum).
  To see that $u_f$ is a morphism in $\Poscop$, note that, for any non-empty subset $L \subseteq \copcompi(\pos P)$, we have
  \[ u_f\left(\coprod_{M \in L} M\right) = u_f\left(\bigcup_{M \in L} M\right) = \coprod_{m \in \bigcup_{M \in L} M} m = \coprod_{M \in L}\; \coprod_{m \in M} m = \coprod_{M \in L} u_f(M) \;, \]
  where the third equation follows from the fact that, in a poset, a coproduct over multiple copies of the same object is that object again.
  For the empty coproduct we have $u_f(\set{\ini_{\pos P}}) = f(\ini_{\pos P}) = \ini_{\pos Q}$ as $f$ preserves initial objects by assumption.
  For uniqueness, note that any element of $\copcompi(\pos P)$ is the coproduct of elements in the image of $\copcompmap[\pos P]$.
  Hence any element of $\Hom[\Poscop] {\copcompi(\pos P)} {\pos Q}$ is already uniquely determined by its restriction along $\copcompmap[\pos P]$.
  
  For naturality of the map $f \mapsto u_f$ in the variable $\pos P$ we need that $u_{f \circ p} = u_f \circ \copcompi(p)$ for all morphisms $p$ in $\Posini$, which follows from $\coprod_{m \in M} f(p(m)) = \coprod_{m' \in \down p(M)} f(m')$.
  For naturality in the variable $\pos Q$ we need that $u_{q \circ f} = q \circ u_f$ for all morphisms $q$ in $\Poscop$, which follows from $q$ preserving coproducts.
\end{proof}

\begin{lemma} \label{lemma:univ_is_full_shape}
  For any $\gpos S \in \Posini$ the map $\copcompmap[\gpos S] \colon \gpos S \to \copcompi(\gpos S)$ is a full shape.
  Moreover, if $\gpos S$ is finite, then $\copcompmap[\gpos S]$ is also finite.
\end{lemma}

\begin{proof}
  First note that it is clear that $\copcompmap[\gpos S]$ is a preshape and that $\copcompi(\gpos S)$ is finite if $\gpos S$ is.
  Fullness follows from the fact that, for $\c a$ and $\c b$ in $\gpos S$, the inclusion $\down{\set{\c a}} \subseteq \down{\set{\c b}}$ implies $\c a \le \c b$.
  
  Now we want to employ \Cref{lemma:easy_shape_condition} to show that $\copcompmap[\gpos S]$ is a shape.
  For this we need that, for all $A, B \in \copcompi(\gpos S)$ and $\c c \in \gpos S$ such that $\copcompmap[\gpos S](\c c) \le A \cop B$, we have $\copcompmap[\gpos S](\c c) \le A$ or $\copcompmap[\gpos S](\c c) \le B$.
  But $\down{\set{\c c}} \subseteq A \cup B$ implies $\c c \in A$ or $\c c \in B$.
  In both cases the set $\down{\set{\c c}}$ must already lie in one of $A$ and $B$ as they are downward closed.
  This implies the statement.
\end{proof}

\begin{remark}
  Combining \Cref{lemma:univ_is_full_shape} with \Cref{rem:cubes_are_free}, we obtain as a special case the fact that $\cubeinc{n} \colon \Cubeini{n} \to \Cube{n}$ is a full shape (which we had already seen in \Cref{ex:shapes}).
\end{remark}

The following result tells us in particular that, for any finite poset $\gpos S$ with more than one object, there exists $n \in \NN$ such that $\copcompmap[\gpos S]$-excisive is equivalent to $n$-excisive.

\begin{proposition} \label{prop:free_shape_eq_to_smaller}
  Let $\gpos S \in \Posini$ and $\pos M \subseteq \gpos S$ a non-empty full downward closed subposet such that, for all $\c s \in \gini{\gpos S}$, there exists an element $m \in \ginipos M$ with $m \le \c s$ (if $\gpos S$ is finite this is equivalent to requiring $\pos M$ to contain all minimal elements of $\gini{\gpos S}$).
  Then the inclusion $i \colon \pos M \to \gpos S$ induces a map of shapes
  \[
  \begin{tikzcd}
    \pos M \rar{i} \dar[swap]{\copcompmap[\pos M]} & \gpos S \dar{\copcompmap[\gpos S]} \\
    \copcompi(\pos M) \rar{\copcompi(i)} & \copcompi(\gpos S)
  \end{tikzcd}
  \]
  which is both direct and indirect.
  
  In particular, a functor $F \colon \infcat C \to \infcat D$ between $\infty$-categories is $\copcompmap[\gpos S]$-excisive if and only if it is $\copcompmap[\pos M]$-excisive (as long as $\infcat C$ is $\copcompmap[\gpos S]$-nice (which implies $\copcompmap[\pos M]$-nice), and $\infcat D$ admits limits indexed by both $\gini{\copcompi(\pos M)}$ and $\gini{\copcompi(\pos S)}$).
\end{proposition}

\begin{proof}
  First note that the square commutes by naturality of $\copcompmap$.
  
  Furthermore $\copcompi(i)$ is full since we have, for any $M \in \copcompi(\pos M)$, that $\down i(M) = i(M)$ (as $\pos M$ is downward closed) and hence that $\copcompi(i)(M) \subseteq \copcompi(i)(N)$ implies $M \subseteq N$.
  Thus, for directness, it only remains to show that $\gini{\copcompi(i)} \colon \gini{\copcompi(\pos M)} \to \gini{\copcompi(\gpos S)}$ is homotopy initial.
  For this let $L \in \gini{\copcompi(\gpos S)}$.
  Then $\gini{\copcompi(i)} \slice L$ has the terminal object $\set{m \in \pos M \mid i(m) \in L}$ (which is actually contained in $\gini{\copcompi(\pos M)}$ since our assumptions imply that there exists an element $m \in \ginipos M$ with $i(m) \in L$) and is thus contractible.
  
  For indirectness we have to check that, for all $M \in \copcompi(\pos M)$, the functor $i_M \colon \copcompmap[\pos M] \slice M \to \copcompmap[\gpos S] \slice \copcompi(i)(M)$ induced by $i$ is homotopy terminal.
  But, by $\pos M$ being downward closed in $\gpos S$, we have that $\down{\set{\c s}} \subseteq \down{i(M)}$ implies $\c s \in M$.
  Hence $i_M$ is even an isomorphism.
\end{proof}

\begin{corollary} \label{cor:free_shape_eq_to_n-exc}
  Let $\gpos S \in \Posini$ be finite.
  Then a functor $F \colon \infcat C \to \infcat D$ between $\infty$-categories is $\copcompmap[\gpos S]$-excisive if and only if it is $(n - 1)$-excisive, where $n$ is the cardinality of the discrete poset $\pos M$ of minimal elements of $\gini{\gpos S}$ (as long as $\infcat C$ admits all finite colimits, and $\infcat D$ admits limits indexed by both $\gini{\Cube n}$ and $\gini{\copcompi(\gpos S)}$).
\end{corollary}

\begin{proof}
  This follows from \Cref{prop:free_shape_eq_to_smaller} applied to $\pos M \cup \set{\ini}$ by noting that $\pos M \cup \set{\ini} \iso \Cubeini{n}$ and using \Cref{rem:cubes_are_free}.
\end{proof}

\subsection{Inane shapes} \label{sec:inane}

In this section we study a certain class of shapes which we call ``inane''.
The terminology is motivated by the fact that every functor is excisive with respect to every inane shape.

\begin{notation} \label{def:gen_sub}
  Let $\pos P \in \Posini$, $\pos Q \in \Poscop$, and $f \in \Hom[\Posini]{\pos P}{\pos Q}$.
  We write $v_f \colon \pos P \to \im u_f$ and $w_f \colon \im u_f \to \pos Q$ for the maps occurring in the diagram
  \[
  \begin{tikzcd}
     & \pos P \dlar[swap]{\copcompmap[\pos P]} \drar{f} & \\
    \copcompi(\pos P) \ar{rr}[near start]{u_f} \drar & & Q \\
     & \im u_f \ar[from = uu, crossing over]{}[near start]{v_f} \urar[swap]{w_f} &
  \end{tikzcd}
  \]
  of factorizations of $f$.
\end{notation}

This cannot be applied to $f = \sigma \colon \gpos S \to \pos S$ an arbitrary (pre)shape as $\gpos S$ does not necessarily have an initial object, so we need to restrict ourselves to reduced (pre)shapes.
Luckily, by \Cref{prop:finite_shape_eq_to_full}, we do not lose much by doing so (at least in the finite case).

\begin{lemma} \label{lemma:univ_im}
  Let $\pos P \in \Posini$, $\pos Q \in \Poscop$, and $f \in \Hom[\Posini]{\pos P}{\pos Q}$.
  Then $\im u_f$ has all coproducts and $w_f \colon \im u_f \to \pos Q$ is full and preserves coproducts.
\end{lemma}

\begin{proof}
  We first show fullness.
  To this end, let $M$ and $M'$ be elements of $\copcompi(\pos P)$ such that $u_f(M) \le u_f(M')$.
  But then, since $u_f$ preserves coproducts, we have $u_f(M) \ile u_f(M \cup M') = u_f(M) \cop_{\pos Q} u_f(M') = u_f(M')$.
  
  Now the fact that fully faithful functors reflect coproducts implies that $\im u_f$ and $w_f$ are an object respectively a morphism in $\Poscop$ since $\copcompi(\pos P)$ has all coproducts and they are preserved by $u_f$.
\end{proof}

\begin{definition}
  Let $\gpos S \in \Posini$, $\pos S \in \Poscop$, and $f \in \Hom[\Posini]{\gpos S}{\pos S}$.
  In this situation, we set
  \[ \pos I^f_s \defeq w_f \slice s = \set{i \in \im u_f \mid w_f(i) \le s} \subseteq \im u_f \]
  which is a full subposet.
  
  A reduced shape $\sigma \colon \gpos S \to \pos S$ is \emph{inane} if there exists $s \in \ginipos S$ such that $\pos I^\sigma_s = \set \ini$, i.e.\ such that for all $\c s \in \gpos S$ with $\sigma(\c s) \le s$ one has $\sigma(\c s) = \ini$.
\end{definition}

\begin{example} \label{ex:cube_not_inane}
  For any $n \in \NN$, the shape $\cubeinc n$ is reduced but not inane, since any $S \in \gini{\Cube n}$ has a subset of cardinality one.
\end{example}

\begin{example} \label{ex:inane_shape}
  The inclusion of posets $\set 0 \to \set{0 \le 1}$ is an inane reduced shape.
\end{example}

This definition is motivated by the statement of \Cref{prop:inane}, which tells us that finite inane shapes do not have any interesting excision properties.
The intuition is that, if there is an element $s \in \pos S$ with $\pos I^\sigma_s = \set{\ini}$, left Kan extension along $\sigma$ copies the value at $\genini \in \gpos S$ to $s$ (since $s$ does not ``see'' anything else of $\gpos S$), so that afterwards applying a functor $F$ and taking the limit over $\ginipos S$ yields $F$ applied to that value again.
However, before we can prove the formal statement, we need the following lemma.

\begin{lemma} \label{lemma:Is_has_terminal}
  Let $\gpos S \in \Posini$, $\pos S \in \Poscop$, and $f \in \Hom[\Posini]{\gpos S}{\pos S}$.
  Then the poset $\pos I^f_s$ has a terminal object.
\end{lemma}

\begin{proof}
  We claim that $i_* := \coprod_{i \in \pos I^f_s} i \in \im u_f$ is terminal in $\pos I^f_s$, where the coproduct is taken in $\im u_f$ and exists by \Cref{lemma:univ_im}.
  It is clear that we have $i \le i_*$ for any $i \in \pos I^f_s$.
  But $i_*$ also lies in $\pos I^f_s$ as $w_f(i_*) = \coprod_{i \in \pos I^f_s} w_f(i) \le s$ by \Cref{lemma:univ_im} and the universal property of the coproduct.
\end{proof}

Now we can prove that finite inane shapes are actually not interesting from an excision standpoint.

\begin{proposition} \label{prop:inane}
  Let $\sigma \colon \gpos S \to \pos S$ be a finite inane reduced shape.
  Then any functor $F \colon \infcat C \to \infcat D$ between $\infty$-categories is $\sigma$-excisive (as long as $\infcat C$ is $\sigma$-nice, and $\infcat D$ admits limits indexed by both $\ginipos S$ and $\gini{(\im u_\sigma \times \pos P)}$, where $\pos P$ is as in the proof below).
\end{proposition}

\begin{proof}
  Let $D \colon \pos S \to \infcat C$ be a $\sigma$-cocartesian diagram and $F \colon \infcat C \to \infcat D$ a functor.
  Furthermore, let $\pos P$ be the full subposet $\set{s \in \pos S \mid \pos I^\sigma_s = \set \ini} \subseteq \pos S$.
  That $\sigma$ is inane is then equivalent to $\pos P$ having more elements than just $\ini$.
  
  Now consider the diagram
  \begin{equation} \label[diagram]{diag:uninteresting_shapes_transformation}
  \begin{tikzcd}[column sep = -10]
  & \gini{(\im u_\sigma \times \pos P)} \dlar[swap]{\pr} \drar{c} & \\
  \im u_\sigma \drar[swap]{w_\sigma} \ar[Rightarrow, shorten > = 20, shorten < = 20]{rr}{\eta} & & \ginipos{S} \dlar{\inc} \\
  & \pos S &
  \end{tikzcd}
  \end{equation}
  where $\inc$ denotes the inclusion, $\pr$ the (restriction of the) projection, $c$ is given by $(i, p) \mapsto w_\sigma(i) \cop p$, and the natural transformation $\eta$ comes from the fact that $w_\sigma(i) \le w_\sigma(i) \cop p$.
  
  \begin{claim}
    The functor $c$ is homotopy initial.
  \end{claim}
  
  \begin{claimproof}
    This is \Cref{lemma:gini_cop_htpy_initial} applied to $w_\sigma$ and the inclusion $\inc[\pos P] \colon \pos P \to \pos S$.
    For this we need to show that for all $s \in \ginipos S$ one of $w_\sigma \slice s$ and ${\inc[\pos P]} \slice s$ has a terminal object different from the respective initial object.
    Note that the category $w_\sigma \slice s = \pos I^\sigma_s$ always has a terminal object by \Cref{lemma:Is_has_terminal}.
    If this terminal object is equal to $\ini$ then we have $s \in \pos P$ by definition, in which case $s \neq \ini$ is a terminal object of ${\inc[\pos P]} \slice s$.
  \end{claimproof}
  
  \begin{claim}
    The functor $\pr$ is homotopy initial.
  \end{claim}
  
  \begin{claimproof}
    We need to show that for all $i \in \im u_\sigma$ the category ${\pr} \slice i$ is contractible.
    This category can be identified with $\gini{(((\im u_\sigma) \slice i) \times \pos P)}$.
    By \Cref{lemma:contractible_products}, it is now enough to show that $\pos P$ has a terminal object different from $\ini$.
    For this we show that if $p, p' \in \pos P$ then also $p \cop p' \in \pos P$ (where the coproduct is taken in $\pos S$), which is enough since $\pos P$ has more than one element (by assumption) and is finite as it is a subset of $\pos S$.
    Otherwise there would be $j \in \gini{(\im u_\sigma)}$ such that $w_\sigma(j) \le p \cop p'$.
    Hence, by definition of $u_\sigma$, there is $\c k \in \gpos S$ such that $\sigma(\c k) \ne \ini$ and $\sigma(\c k) \le p \cop p'$.
    Now we look at
    \[ \gpos S_{p,p',\c k} = \set{\c s \in \gpos S \mid \sigma(\c s) \le p \text{ and } \sigma(\c k) \le p' \cop \sigma(\c s)} \]
    which, since $\sigma$ is a shape and $\sigma(\c k) \le p \cop p'$, is contractible.
    In particular it is non-empty, so let $\c s \in \gpos S_{p,p',\c k}$.
    We have $\sigma(\c s) \le p$, thus $\sigma(\c s) \in \pos I^\sigma_p$ (as $\sigma(\c s) \in \im u_\sigma$), and hence $\sigma(\c s) = \ini$.
    Thus $\sigma(\c k) \le p' \cop \sigma(\c s) = p'$ and hence $\sigma(\c k) = \ini$, a contradiction.
  \end{claimproof}
  
  \begin{claim}
    The natural transformation $D \circ \eta \colon D \circ w_\sigma \circ \pr \to D \circ {\inc} \circ c$ is an equivalence.
  \end{claim}
  
  \begin{claimproof}
    Since $D$ is $\sigma$-cocartesian, we have $D \eq (\Ext \Res{\sigma}) (D)$.
    By \Cref{lemma:kan_local} the transformation $(\Ext \Res{\sigma})(D) \circ \eta$ is given, at $(i, p) \in \gini{(\im u_\sigma \times \pos P)}$, by the map
    \[ \colim{\sigma \slice w_\sigma(i)} \left( \Res{\sigma}(D) \circ \pr[\sigma \slice w_\sigma(i)] \right) \longto \colim{\sigma \slice (w_\sigma(i) \cop p)} \left( \Res{\sigma}(D) \circ \pr[\sigma \slice (w_\sigma(i) \cop p)] \right) \]
    induced by the canonical functor $\sigma \slice w_\sigma(i) \to \sigma \slice (w_\sigma(i) \cop p)$.
    
    We will now show that $\sigma \slice w_\sigma(i) \to \sigma \slice (w_\sigma(i) \cop p)$ is an isomorphism, for which surjectivity (on objects) suffices as the map is an inclusion of full subposets of $\gpos S$.
    For this we need that, if $\c k \in \gpos S$ such that $\sigma(\c k) \le w_\sigma(i) \cop p$, then already $\sigma(\c k) \le w_\sigma(i)$.
    For this consider
    \[ \gpos S_{p,w_\sigma(i),\c k} = \set{\c s \in \gpos S \mid \sigma(\c s) \le p \text{ and } \sigma(\c k) \le w_\sigma(i) \cop \sigma(\c s)} \]
    which is contractible (since $\sigma(\c k) \le w_\sigma(i) \cop p$), hence non-empty.
    So let $\c s$ be an element of $\gpos S_{p,w_\sigma(i),\c k}$.
    The condition $\sigma(\c s) \le p$ implies $\sigma(\c s) = \ini$, thus we have $\sigma(\c k) \le w_\sigma(i) \cop \ini = w_\sigma(i)$, as we wanted to show.
  \end{claimproof}
  
  Now we note that \cref{diag:uninteresting_shapes_transformation} extends to a diagram
  \[
  \begin{tikzcd}[column sep = -15, row sep = 25]
    & \cone{\left(\gini{(\im u_\sigma \times \pos P)}\right)} \dlar[swap]{\cone \pr} \drar{\cone c} & \\
    \cone{(\im u_\sigma)} \drar[swap]{\cone{(w_\sigma)}} \ar[Rightarrow, shorten > = 20, shorten < = 20]{rr}{\cone \eta} & & \cone{(\ginipos{S})} \dlar{\cone \inc} \\
    & \cone{\pos S} \dar{q} & \\
    & \pos S &
  \end{tikzcd}
  \]
  where $\cone \eta$ is given by the identity at the cone point and by $\eta$ otherwise, and $q$ is the identity on $\pos S$ and $\ini$ at the cone point.
  This allows us to obtain maps from $F(D(\ini))$ into certain limits indexed by the categories occurring in \cref{diag:uninteresting_shapes_transformation}.
  Namely we get, by (the dual of) \Cref{lemma:functors_and_map_from_colim}, the homotopy commutative diagram
  \[
  \begin{tikzcd}
    \lim{\gini{(\im u_\sigma \times \pos P)}} (F \circ D \circ w_\sigma \circ \pr) \ar{rr}{\eq}[swap]{\eta} & & \lim{\gini{(\im u_\sigma \times \pos P)}} (F \circ D \circ {\inc} \circ c) \\
    & F(D(\ini)) \dlar[swap]{\eq} \ular \urar \drar & \\
    \lim{\im u_\sigma} (F \circ D \circ w_\sigma) \ar{uu}{\pr^*}[swap]{\eq} & & \lim{\ginipos S} (F \circ D \circ \inc) \ar{uu}{\eq}[swap]{c^*}
  \end{tikzcd}
  \]
  where the maps around the boundary are equivalences by the previous three claims and the map into the bottom left is one by \Cref{lemma:terminal_object_colimit} as $\ini$ is initial in $\im u_\sigma$ and $q \circ \cone{(w_\sigma)}$ applied to the unique map from the cone point to $\ini$ is the identity.
  Hence the map into the bottom right is an equivalence as well, which is what we wanted to show.
\end{proof}

\subsection{Non-inane shapes}

This subsection is devoted to comparing a non-inane (pre)shape $\sigma \colon \gpos S \to \pos S$ to the free shape $\copcompmap[\gpos S] \colon \gpos S \to \copcompi(\gpos S)$ in two different ways.

\begin{proposition} \label{prop:indirect_to_univ}
  Let $\sigma \colon \gpos S \to \pos S$ be a full preshape.
  Then we have a map of posets $e \colon \pos S \to \copcompi(\gpos S)$ defined by $s \mapsto \set{\c s \in \gpos S \mid \sigma(\c s) \le s}$.
  If $\inv e(\set{\ini_{\gpos S}}) = \set{\ini_{\pos S}}$, this gives a map of preshapes
  \[
  \begin{tikzcd}
    \gpos S \rar{\id} \dar[swap]{\sigma} & \gpos S \dar{\copcompmap[\gpos S]} \\
    \pos S \rar{e} & \copcompi(\gpos S)
  \end{tikzcd}
  \]
  that is indirect.
  Furthermore, if we additionally assume $\sigma$ to be a shape, the condition $\inv e(\set{\ini_{\gpos S}}) = \set{\ini_{\pos S}}$ is equivalent to $\sigma$ being non-inane.
\end{proposition}

\begin{proof}
  For commutativity of the square we need, for all $\c k \in \gpos S$, the equality $\set{\c s \in \gpos S \mid \sigma(\c s) \le \sigma(\c k)} = \set{\c s \in \gpos S \mid \c s \le \c k}$, which follows from $\sigma$ being full.
  For indirectness, note that, for all $s \in \pos S$, the induced functor $\sigma \slice s \to \copcompmap[\gpos S] \slice e(s)$ is an isomorphism since both sides can be identified with the full subposet $e(s) = \set{\c s \mid \sigma(\c s) \le s} \subseteq \gpos S$.
  For the last statement of the proposition, note that $\sigma$ being non-inane is equivalent to, for all $s \in \ginipos S$, there existing an $\c s \in \gpos S$ such that $\ini_{\pos S} \neq \sigma(\c s) \le s$, which is a reformulation of the statement $\inv e(\set{\ini_{\gpos S}}) = \set{\ini_{\pos S}}$.
\end{proof}

\begin{corollary} \label{cor:exc_implies_free_exc}
  Let $\sigma \colon \gpos S \to \pos S$ be a non-inane full shape.
  Then a functor $F \colon \infcat C \to \infcat D$ between $\infty$-categories that is $\sigma$-excisive is also $\copcompmap[\gpos S]$-excisive (as long as $\infcat C$ is $\sigma$-nice and $\copcompmap[\gpos S]$-nice, and $\infcat D$ admits limits indexed both by $\ginipos S$ and by $\gini{\copcompi(\gpos S)}$).
\end{corollary}

Combining this with \Cref{cor:free_shape_eq_to_n-exc}, we obtain:

\begin{theorem} \label{thm:nb-excisive_implies_n-excisive}
  Let $\sigma \colon \gpos S \to \pos S$ be a finite non-inane full shape.
  Then a functor $F \colon \infcat C \to \infcat D$ between $\infty$-categories that is $\sigma$-excisive is also $(n-1)$-excisive, where $n$ is the number of minimal elements of $\gini{\gpos S}$ (as long as $\infcat C$ admits all finite colimits, and $\infcat D$ admits limits indexed by $\gini{\Cube n}$, $\ginipos S$, and $\gini{\copcompi(\gpos S)}$).
\end{theorem}

The next two results tell us that any non-inane full shape $\gpos S \to \pos S$ has the same excision properties as a shape $\gpos S \to \pos Q$ such that $\pos Q$ is a retract of $\copcompi(\gpos S)$.
In particular, for a fixed finite $\gpos S$, there are only finitely many full shapes of the form $\gpos S \to \pos S$ with different excision properties.
Combining this with \Cref{prop:finite_shape_eq_to_full}, we even get this result if we do not require the shapes to be full.
See \Cref{cor:fin_many_excision_notions} for a precise version of this statement.

\begin{proposition} \label{prop:equiv_to_image_of_free}
  Let $\sigma \colon \gpos S \to \pos S$ be a reduced shape.
  Then $v_\sigma \colon \gpos S \to \im u_\sigma$ (see \Cref{def:gen_sub}) is again a shape which is finite if $\sigma$ is, and the map of shapes
  \[
  \begin{tikzcd}
  \gpos S \rar{\id} \dar[swap]{v_\sigma} & \gpos S \dar{\sigma} \\
  \im u_\sigma \rar{w_\sigma} & \pos S
  \end{tikzcd}
  \]
  is indirect.
  Furthermore, if $\sigma$ is full and non-inane, then it is also direct.
\end{proposition}

\begin{proof}
  The first part follows from \Cref{lemma:full_sub_is_shape} using \Cref{lemma:univ_im}.
  For the second part we will use the following claim:
  
  \begin{claim}
    Let the functor $r \colon \pos S \to \im u_\sigma$ be given by $s \mapsto \max\, \pos I^\sigma_s$ and assume it restricts to a functor $\ginipos S \to \gini{(\im u_\sigma)}$.
    Then $r|_{\ginipos S}$ is right adjoint to $w_\sigma|_{(\gini{\im u_\sigma)}}$.
  \end{claim}
  
  \begin{claimproof}
    First note that $\max\, \pos I^\sigma_s$ exists by \Cref{lemma:Is_has_terminal} and that $r$ is functorial as $s \le s'$ implies $\pos I^\sigma_s \subseteq \pos I^\sigma_{s'}$ and hence $\max\, \pos I^\sigma_s \le \max\, \pos I^\sigma_{s'}$.
    For the adjunction we have to show that, for all $i \in \gini{(\im u_\sigma)}$ and $s \in \ginipos S$, we have $i \le r(s)$ if and only if $w_\sigma(i) \le s$, which follows from $\pos I^\sigma_s$ being defined as $w_\sigma \slice s$.
  \end{claimproof}
  
  Note that our assumption that $\sigma$ is not inane precisely says that $r$ restricts to a functor $\ginipos S \to \gini{(\im u_\sigma)}$.
  Hence $w_\sigma|_{(\gini{\im u_\sigma)}} \colon \gini{(\im u_\sigma)} \to \ginipos S$ is left adjoint and thus homotopy initial.
  Together with $w_\sigma$ and $\sigma$ being full this shows directness.
\end{proof}

\begin{proposition} \label{prop:retract_of_free}
  Let $\sigma \colon \gpos S \to \pos S$ be a full shape and assume that $u_\sigma \colon \copcompi(\gpos S) \to \pos S$ is surjective.
  Then $u_\sigma \circ e = \id[\pos S]$, where $e$ is as in \Cref{prop:indirect_to_univ}.
  In particular $\sigma$ is a retract of $\copcompmap[\gpos S]$ (in the category $\gpos S \slice \Posini$).
\end{proposition}

\begin{proof}
  We need to show that $u_\sigma(e(s)) = s$ for all $s \in \pos S$, where $e(s) = \set{\c s \in \gpos S \mid \sigma(\c s) \le s}$.
  So, let $s$ be an element of $\pos S$.
  By assumption, there is an $M \subseteq \gpos S$ such that $s = u_\sigma(M) = \coprod_{m \in M} \sigma(m)$.
  This implies $\sigma(m) \le s$ for all $m \in M$, hence $M \subseteq e(s)$.
  We obtain the inequalities $s = u_\sigma(M) \le u_\sigma(e(s)) = \coprod_{\c s \in e(s)} \sigma(\c s) \le s$, which imply $u_\sigma(e(s)) = s$.
\end{proof}

\begin{lemma} \label{lemma:u_v_surjective}
  Let $\sigma \colon \gpos S \to \pos S$ be a reduced shape.
  Then the following diagram commutes:
  \[
  \begin{tikzcd}
    \copcompi(\gpos S) \ar{rr}{u_\sigma} \drar[swap]{u_{v_\sigma}} & & \pos S \\
     & \im u_\sigma \urar[swap]{w_f} &
  \end{tikzcd}
  \]
  which in particular implies that $u_{v_\sigma}$ is surjective (as $u_\sigma$ maps surjectively onto $\im u_\sigma$).
\end{lemma}

\begin{proof}
  This follows directly from \Cref{lemma:univ_im}.
\end{proof}

\begin{theorem} \label{thm:equiv_to_retract_of_free}
  Let $\sigma \colon \gpos S \to \pos S$ be a full shape.
  Then $v_\sigma \colon \gpos S \to \im u_\sigma$ is a retract of $\copcompmap[\gpos S]$ (in the category $\gpos S \slice \Posini$).
  Moreover, if $\sigma$ is non-inane, a functor $F \colon \infcat C \to \infcat D$ between $\infty$-categories is $\sigma$-excisive if and only if it is $v_\sigma$-excisive (as long as $\infcat C$ is $\sigma$-nice (which implies $v_\sigma$-nice), and $\infcat D$ admits limits indexed by both $\ginipos S$ and $\gini{(\im u_\sigma)}$).
\end{theorem}

\begin{proof}
  This follows by applying \Cref{prop:equiv_to_image_of_free} to $\sigma$ and then \Cref{prop:retract_of_free} to $v_\sigma$.
  For the latter part we use that $v_\sigma$ is full by \Cref{lemma:poset_full_composition} since $\sigma$ is, and that $u_{v_\sigma}$ is actually surjective by \Cref{lemma:u_v_surjective}.
\end{proof}

\section{Cubical shapes} \label{sec:cubes}

\newcommand{\mc}[2]{\operatorname{mc}(#1, #2)}
\newcommand{\isol}[1]{\operatorname{isol}(#1)}
\newcommand{\nisol}[1]{\operatorname{nisol}(#1)}

In this section we study (pre)shapes of the form $\sigma \colon \gpos S \to \Cube S$, i.e.\ those with codomain a cube, and, more specifically, how the associated notions of excision relate to $n$-excision.

\begin{remark}
  By \Cref{prop:finite_shape_eq_to_full} we do not lose any generality when restricting ourselves to full shapes, at least if the shape is finite (which in the later parts of this section it will need to be anyway), so we will do so freely.
  
  Also note that, if $\sigma$ is full, then it is also injective (by \Cref{lemma:pos_full_implies_injective}).
  Hence we can, by a slight abuse of notation, consider $\gpos S$ to be a subset of $\Cube S$.
  We will do so throughout this section.
\end{remark}

\begin{lemma} \label{lemma:cube_shape}
  Let $S$ be a set and $\sigma \colon \gpos S \to \Cube S$ a full preshape.
  Then $\sigma$ is a shape if and only if $\gpos S$ is downward closed in $\Cube S$.
\end{lemma}

\begin{proof}
  For the ``if'' direction let $A, B \in \Cube{S}$ and $\c K \in \gpos S$ such that $\c K \subseteq A \cup B$, and note that ${\gpos S}_{A,B,\c K} = \set{\c A \in \gpos S \mid \c K \setminus B \subseteq \c A \subseteq A}$.
  Now, since $\gpos S$ is downward closed, the element $\c K \setminus B \in \Cube{S}$ is contained in $\gpos S$, and hence is an initial object of ${\gpos S}_{A,B,\c K}$.
  
  For the ``only if'' direction let $A \in \Cube S \setminus \gpos S$ and $\c K \in \gpos S$ such that $A \subset \c K$.
  Then ${\gpos S}_{A,\c K \setminus A, \c K} = \set{\c A \in \gpos S \mid A = \c K \setminus (\c K \setminus A) \subseteq \c A \subseteq A}$ is empty and thus not contractible.
\end{proof}

The next lemma is a direct application of the results of \Cref{section:maps} to the situation we are studying.

\begin{lemma} \label{lemma:cubes_map}
  Let $S$ and $T$ be sets and $\sigma \colon \gpos S \to \Cube S$ and $\tau \colon \gpos T \to \Cube T$ full preshapes.
  Furthermore let $f \colon S \to T$ be a map such that $\inv f(\gpos T) \subseteq \gpos S$ (we abuse notation slightly by denoting the induced map $\Cube S \to \Cube T$ by $f$ as well).
  Then each $\sigma$-excisive functor $F \colon \infcat C \to \infcat D$ between $\infty$-categories is also $\tau$-excisive (as long as $\infcat C$ is $\sigma$-nice and $\tau$-nice and $\infcat D$ admits limits indexed both by $\Cubegini{N}$ and by $\Cubegini{M}$).
\end{lemma}

\begin{proof}
  We write $\gpos Q \defeq \inv f(\gpos T)$ considered as a full subposet of $\Cube S$ and note that it is clear that the inclusion $\iota \colon \gpos Q \to \Cube S$ is a preshape.
  
  \begin{claim}
    The map of preshapes $(f, \c f) \colon \iota \to \tau$ is indirect, where $\c f \colon \gpos Q \to \gpos T$ is the restriction of $f$.
  \end{claim}
  
  \begin{claimproof}
    We need to show that for all $N \subseteq S$ the induced map ${\c f}_N \colon \iota \slice N \to \tau \slice f(N)$ is homotopy terminal, i.e.\ that for all $M \in \tau \slice f(N)$ the poset
    \[ M \slice {\c f}_N \iso \set{N' \subseteq S \mid N' \subseteq N \text{ and } M \subseteq f(N') \in \gpos T} \subseteq \gpos Q \]
    is contractible.
    For this purpose, define
    \[ \pos K \defeq \set{K \in \gpos T \mid M \subseteq K \subseteq f(N)} \subseteq \gpos T \]
    considered as a full subposet.
    Note that $\pos K$ has the initial object $M$, hence $\pos K$ is contractible, and it is enough to show that it is homotopy equivalent to $M \slice {\c f}_N$.
    For this, let $l \colon M \slice {\c f}_N \to \pos K$ be given by $f$ and $r \colon \pos K \to M \slice {\c f}_N$ by $K \mapsto \inv f(K) \cap N$.
    To see that $r$ is well-defined note that $K \subseteq f(\inv f(K) \cap N) \subseteq K$ where for the first inclusion we use that, since $K \subseteq f(N)$, there exists $N' \subseteq N$ such that $f(N') = K$ hence $N' \subseteq \inv f(K) \cap N$ and $K = f(N') \subseteq f(\inv f(K) \cap N)$.
    We claim that $l$ is left adjoint to $r$.
    For this we need to show that, for all $N' \in M \slice {\c f}_N$ and $K \in \pos K$, we have $f(N') \subseteq K$ if and only if $N' \subseteq \inv f(K) \cap N$, which is clear.
  \end{claimproof}
  
  By \Cref{prop:indirect}, we now know that a $\iota$-excisive functor $\infcat C \to \infcat D$ is $\tau$-excisive.
  But the map of preshapes $(\id, \c \iota) \colon \iota \to \sigma$ is direct, where $\c \iota \colon \gpos Q \to \gpos S$ is the restriction of $\iota$ (note that this map is well-defined by assumption).
  Hence, by \Cref{prop:direct}, a $\sigma$-excisive functor $\infcat C \to \infcat D$ is $\iota$-excisive, and we are done.
\end{proof}

We will now use this lemma in the situation where $\tau = \cubeinc{n}$ for some $n \in \NN$.
To state the result, we need the following:

\begin{notation}
  Let $S$ be a set and $M \subseteq \Cube{S}$ a subset of its power set.
  We write $\mc S M$ for the minimal number of elements of $M$ needed to cover $S$ and set $\mc S M = \infty$ if no such cover exists.
\end{notation}

\begin{remark} \label{rem:mc_infty}
  Note that, when $\sigma \colon \gpos S \to \Cube{S}$ is a full, finite shape such that $\mc S {\gpos S} = \infty$, then $\sigma$ is an inane shape.
\end{remark}

\begin{corollary}
  Let $S$ be a finite set and $\sigma \colon \gpos S \to \Cube S$ a full shape.
  Then a functor $F \colon \infcat C \to \infcat D$ between $\infty$-categories that is $\sigma$-excisive is also $(n-1)$-excisive where $n = \mc S {\gpos S}$ (as long as $\infcat C$ admits all finite colimits and $\infcat D$ admits limits indexed both by $\Cubegini{S}$ and by $\Cubegini{n}$).
  Here we take $\infty$-excisive to be a trivial condition, i.e.\ all functors are $\infty$-excisive.
\end{corollary}

\begin{proof}
  Note that the statement is vacuous if $\mc S {\gpos S} = \infty$, so that we can assume $\mc S {\gpos S}$ to be finite.
  In particular we can chose a cover $A_0, \dots, A_{n-1} \in \gpos S$ of S.
  Since $\gpos S$ is downward closed by \Cref{lemma:cube_shape}, we can assume the $A_i$ to be pairwise disjoint.
  Now let $f \colon S \to \finset n$ be the map given by $s \mapsto i$ for $s \in A_i$.
  Then
  \[ \inv f(\Cubeini{n}) = \set{B \subseteq S \mid \exists i \colon B \subseteq A_i} \subseteq \gpos S \]
  so that an application of \Cref{lemma:cubes_map} finishes the proof.
\end{proof}

One natural question to ask is whether the implication of the corollary is actually an equivalence.
The following proposition shows that this is indeed the case.

\begin{proposition}
  Let $S$ be a finite set, $\sigma \colon \gpos G \to \Cube{S}$ a full preshape, and $n \in \NN$.
  Assume that $n \le \mc S {\gpos G}$.
  Then a functor $F \colon \infcat C \to \infcat D$ between $\infty$-categories that is $(n-1)$-excisive is also $\sigma$-excisive (as long as $\infcat C$ admits all finite colimits and $\infcat D$ admits all finite limits).
\end{proposition}

\begin{proof}
  Throughout the proof let $D \colon \Cube S \to \infcat C$ denote a $\sigma$-cocartesian diagram.
  First note that the statement is true if $\gpos G = \Cubeini{S}$ (by \Cref{cor:n-excisive_implies_m-excisive}) or if $\mc S {\gpos G} = \infty$ (by \Cref{rem:mc_infty,prop:inane}).
  
  We will now proceed by a slightly complicated to state induction.
  For this we define
  \begin{align*}
    \isol \sigma &\defeq \set{s \in S \mid \set s \text{ is a maximal element of } \gpos G} \\
    \nisol \sigma &\defeq S \setminus \isol \sigma
  \end{align*}
  the sets of what we will call \emph{isolated} respectively \emph{non-isolated directions}.
  The induction is now on the tuple $(\card{\nisol \sigma}, \card S, \card{\gpos G}) \in (\NN)^3$, ordered lexicographically.
  Said differently, we will from now on assume that the statement has already been proven for all finite sets $T$ and full preshapes $\tau \colon \gpos G' \to \Cube T$ such that both $n \le \mc T {\gpos G'}$ and one of the following conditions is fulfilled:
  \begin{itemize}
    \item $\card{\nisol \tau} < \card{\nisol \sigma}$,
    \item $\card{\nisol \tau} = \card{\nisol \sigma}$ and $\card T < \card S$,
    \item $\card{\nisol \tau} = \card{\nisol \sigma}$ and $\card T = \card S$ and $\card{\gpos G'} < \card{\gpos G}$.
  \end{itemize}
  Note that this induction is possible since any strictly descending chain in $(\NN)^3$ ordered lexicographically is finite (i.e.\ $(\NN)^3$ is well-founded).
  
  Let $M$ denote the set of maximal elements of $\gpos G$ of cardinality larger than $1$.
  Note that if $M$ is empty, then we are in one of the two cases already handled in the beginning of the proof, hence we can assume that $M$ is not empty.
  Denote by $\pos R \subseteq \Cube{M} \times \gpos G$ the full subposet spanned by $\set \ini \times (\gpos G \setminus M)$ and $(\set m, m)$ for all $m \in M$, and set $r \defeq {\id} \times \sigma \colon \pos R \to \Cube{M} \times \Cube{S}$.
  Now let $E' \defeq \Res \sigma (D) \circ \pr[\gpos G] \colon \pos R \to \infcat C$ and set $E \defeq \Lan{r} (E') \colon \Cube{M} \times \Cube{S} \to \infcat C$.
  We note the following fact for later use:
  
  \begin{claim} \label{claim:equiv_to_terminal}
    Let $A \in \Cube{M}$ and $B \in \Cube{S}$.
    Assume that for all $m \in M$ such that $m \subseteq B$ we have $m \in A$. 
    Then $E(A, B) \to E(M, B)$ is an equivalence.
  \end{claim}
  
  \begin{claimproof}
    By assumption the induced functor $r \slice (A, B) \to r \slice (M, B)$ is an isomorphism.
    This implies the claim by \Cref{lemma:kan_local}.
  \end{claimproof}
  
  Now, for any $A \in \Cube{M}$, let $\pos R_A$ denote the full subposet of $\gpos G$ spanned by $\gpos G \setminus M$ and $A \subseteq \gpos G$, and set $r_A \defeq \restrict{\sigma}{\pos R_A} \colon \pos R_A \to \Cube{S}$.
  
  \begin{claim}
    We have, for all $A \in \Cube{M}$, that the diagram $\restrict{E}{\set A \times \Cube{S}} \colon \Cube{S} \to \infcat C$ is $r_A$-cocartesian.
    Moreover there is an equivalence $\restrict{E}{\set{M} \times \Cube{S}} \eq D$.
  \end{claim}
  
  \begin{claimproof}
    For the first part it is enough to show that the functors
    \[
    \begin{tikzcd}
    \set A \times \pos R_A \dar[swap]{r_A} & \pos R \dar{r} \\
    \set A \times \Cube{S} \rar[hook]{\inc} & \Cube{M} \times \Cube{S}
    \end{tikzcd}
    \]
    fulfill the assumptions of \Cref{lemma:gc_equiv}.
    For this let $B \in \Cube S$ and $(A', G') \in \pos R$ such that $r(A', G') \le (A, B)$.
    Then we have, for all $G \in \pos R_A$ such that $r_A(G) \subseteq B$ and $r(A', G') \le (A, r_A(G))$, that $G' \subseteq G$.
    But directly from the definitions we obtain $G' \in \pos R_{A'} \subseteq \pos R_A$, so that $G'$ is an initial object of $\set{G \in \pos R_A \mid r_A(G) \subseteq B \text{ and } r(A', G') \le (A, r_A(G))}$.
    
    For the second part note the existence of the diagram
    \[
    \begin{tikzcd}[row sep = 25]
      \pos R_M \rar[equal] \drar[swap]{r_M} & \gpos G \dar[swap]{\sigma} \rar{\iso}[swap]{\inv{(\pr[\gpos G])}} & \pos R \dar{r} \dlar[Rightarrow, shorten < = 20, shorten > = 20]{\alpha} \\
       & \set M \times \Cube S \rar{\inc} & \Cube M \times \Cube S
    \end{tikzcd}
    \]
    and that, by \Cref{claim:equiv_to_terminal}, the map $\alpha_* \colon \big( {\Res{\inv{(\pr[\gpos G])}} \Res{r}} \big) (E) \to {\Res{\sigma}} \big( \restrict{E}{\set M \times \Cube S} \big)$ is an equivalence.
    We obtain
    \begin{align*}
      \restrict{E}{\set M \times \Cube S} &\eq (\Lan{r_M} \Res{r_M}) \big( \restrict{E}{\set M \times \Cube S} \big) \\
       &\eq (\Lan{\sigma} \Res{\sigma}) \big( \restrict{E}{\set M \times \Cube S} \big) \\
       &\eq \big( {\Lan{\sigma} \Res{\inv{(\pr[\gpos G])}} \Res{r}} \big) (E) \\
       &\eq \big( {\Lan{\sigma} \Res{\inv{(\pr[\gpos G])}}} \big) (E') \\
       &= (\Lan{\sigma} \Res{\sigma}) (D) \\
       &\eq D
    \end{align*}
    where, for the fourth equivalence, we use \Cref{lemma:fully_faithful_kan_unit} and that $r$ is full.
  \end{claimproof}
  
  Now consider the following diagram obtained from \Cref{lemma:functors_and_map_from_colim}:
  \begin{equation} \label[diagram]{diag:cubes_reduction}
  \begin{tikzcd}
    (F \circ E)(\ini, \ini) \dar[swap]{\eq} & (F \circ E)(\ini, \ini) \rar{\eq} \lar[equal] \dar & (F \circ E)(M, \ini) \dar \\
    \lim{\set{\ini} \times \Cubegini{S}} (F \circ E) & \lim{\Cube{M} \times \Cubegini{S}} (F \circ E) \rar \lar[swap]{\eq} & \lim{\set{M} \times \Cubegini{S}} (F \circ E)
  \end{tikzcd}
  \end{equation}
  and note that the lower left horizontal map is an equivalence since the inclusion $\set{\ini} \times \Cubegini{S} \to \Cube{M} \times \Cubegini{S}$ is homotopy initial, and that the upper right horizontal map is an equivalence by \Cref{claim:equiv_to_terminal}.
  Moreover the left vertical map is an equivalence since $F$ is $r_\ini$-excisive by the induction hypothesis.
  To see this note that the passage from $\sigma = r_M$ to $r_\ini$ cannot increase the number of non-isolated directions and does not change the set $S$.
  Moreover we have $\pos R_\ini \subsetneq R_M = \gpos G$, since $M$ is not empty, and hence $\card{\pos R_\ini} < \card{\gpos G}$ and $\mc S {\pos R_\ini} \ge \mc S {\gpos G} \ge n$.
  
  Our goal is to show that $F \circ D \eq F \circ \restrict{E}{\set{M} \times \Cube{S}}$ is cartesian, i.e.\ that the right vertical map in \cref{diag:cubes_reduction} is an equivalence.
  For this it is, by the above, enough to show that the lower right horizontal map in the diagram is an equivalence.
  We will do this in multiple steps.
  For the first one we consider the poset
  \[ \pos Q \defeq \set{(A, B) \mid \text{for all } m \in M \text{ such that } m \subseteq B \text{ we have } m \in A} \subseteq \Cube{M} \times \Cubegini{S} \]
  and claim that the canonical map $\eta \colon \restrict{(F \circ E)}{\pos Q} \to \Ran{\inc} \restrict{(F \circ E)}{\set{M} \times \Cubegini{S}}$ is an equivalence where $\inc \colon \set{M} \times \Cubegini{S} \to \pos Q$ is the inclusion.
  For this, let $(A, B) \in \pos Q$ and consider the square
  \[
  \begin{tikzcd}
    (F \circ E)(A, B) \rar{\eta_{A,B}} \dar & \lim[s]{(A,B) \slice \inc} (F \circ E \circ \pr) \dar \\
    (F \circ E)(M, B) \rar{\eta_{M,B}} & \lim[s]{(M,B) \slice \inc} (F \circ E \circ \pr)
  \end{tikzcd}
  \]
  for which we want to show that the upper horizontal map is an equivalence.
  That the left vertical map is an equivalence was precisely the statement of \Cref{claim:equiv_to_terminal}.
  Moreover the right vertical map is an equivalence by \Cref{lemma:kan_local} since the induced functor $(M,B) \slice \inc \to (A,B) \slice \inc$ is an isomorphism, and the bottom horizontal map is an equivalence by \Cref{lemma:kan_counit,lemma:terminal_object_colimit} since $(M, B)$ is the initial object of $(M,B) \slice \inc$.
  
  We obtain, by \Cref{lemma:colim_of_kan}, that in the following factorization of the map we desire to be an equivalence the second map is an equivalence:
  \[ \lim{\Cube{M} \times \Cubegini{S}} (F \circ E) \longto \lim{\pos Q} (F \circ E) \xlongto{\eq} \lim{\set{M} \times \Cubegini{S}} (F \circ E) \]
  and hence that it is enough to show that the first map is an equivalence as well.
  For this we factor it further via the poset
  \[ \pos Q \subseteq \pos L \defeq \set{(A, B) \mid {\inc[\pos Q]} \slice (A, B) \text{ is contractible}} \subseteq \Cube{M} \times \Cubegini{S} \]
  which is defined precisely such that the inclusion $\pos Q \to \pos L$ becomes homotopy initial.
  Hence it is enough to show that $\restrict{(F \circ E)}{\Cube{M} \times \Cubegini{S}}$ is a right Kan extension of $\restrict{(F \circ E)}{\pos L}$ along the inclusion $\inc[\pos L] \colon \pos L \to \Cube{M} \times \Cubegini{S}$.
  For this we choose a filtration
  \[ \pos L = \pos L_0 \subsetneq \pos L_1 \subsetneq \dots \subsetneq \pos L_k = \Cube{M} \times \Cubegini{S} \]
  such that for all $1 \le i \le k$ there exists an $l_i \in \Cube{M} \times \Cubegini{S}$ such that $\pos L_i \setminus \pos L_{i-1} = \set{l_i}$ and $l \in \pos L_{i-1}$ for all $l_i < l \in \Cube{M} \times \Cubegini{S}$.
  By \Cref{lemma:colim_of_kan} it is enough to show that $\restrict{(F \circ E)}{\pos L_i}$ is a right Kan extension of $\restrict{(F \circ E)}{\pos L_{i-1}}$ for all $1 \le i \le k$.
  Similarly to before, this is equivalent to the canonical map $\eta_l \colon (F \circ E)(l) \to \lim{l \slice \inc[\pos L_{i-1}]} (F \circ E \circ \pr)$ being an equivalence for all $l \in \pos L_i$ which is again automatic if $l \in \pos L_{i-1}$, so it is enough to check it for $l = l_i$.
  In that case we write $l_i = (A, B) \in \Cube{M} \times \Cubegini{S}$ and note the existence of the commutative diagram
  \[
  \begin{tikzcd}
     & \pos L_{i-1} \drar[hook, bend left = 15] & \\
    l_i \slice \inc[\pos L_{i-1}] \rar{\iso} \urar[bend left = 15]{\pr} \dar[hook] & \gini{\left(\Cube{M \setminus A} \times \Cube{S \setminus B}\right)} \rar[hook]{\iota} \uar[hook, swap]{\iota} \dar[hook] & \Cube{M} \times \Cubegini{S} \\
    \cone{(l_i \slice \inc[\pos L_{i-1}])} \rar{\iso} & \Cube{M \setminus A} \times \Cube{S \setminus B} \urar[hook, bend right = 15, swap]{\iota} &
  \end{tikzcd}
  \]
  where the three maps labeled $\iota$ are all given by $(A', B') \mapsto (A \cup A', B \cup B')$.
  This tells us, by \Cref{lemma:kan_counit}, that $\eta_{l_i}$ being an equivalence is equivalent to $\restrict{(F \circ E)}{\pos H}$ being cartesian, where $\pos H = \iota(\Cube{M \setminus A} \times \Cube{S \setminus B})$, which we will now prove.
  
  For this we set $N \defeq \set{m \in M \setminus A \mid m \subseteq B} \subseteq M \setminus A$ and $\bar{N} \defeq (M \setminus A) \setminus N$, yielding $\pos H = \iota(\Cube{N} \times \Cube{\bar N} \times \Cube{S \setminus B})$.
  We now claim that, for all $C \subseteq \bar N$, the restriction of $F \circ E$ to the cube $\pos H_C \defeq \iota(\Cube{N} \times \set{C} \times \Cube{S \setminus B})$ is cartesian (this implies that $\restrict{(F \circ E)}{\pos H}$ is cartesian by \Cref{lemma:cartesian_products}).
  The main step in proving this consists of the following claim:
  
  \begin{claim}
    The diagram $\restrict{E}{\pos H_C}$ is $h_C$-cocartesian, where $h_C$ is the inclusion of the full subposet
    \[ \gpos H_C \defeq \iota\left((\Cubeini{N} \times \set C \times \set \ini) \cup (\set \ini \times \set C \times \pi_{S \setminus B}(\pos R_{A \cup C}))\right) \]
    into $\pos H_C$.
    Here $\pi_{S \setminus B} \colon \Cube{S} \to \Cube{S \setminus B}$ denotes the map given by $B' \mapsto B' \cap (S \setminus B)$.
  \end{claim}
  
  \begin{claimproof}
    It is enough to show that the functors
    \[
    \begin{tikzcd}
      \gpos H_C \dar[swap]{h_C} & \pos R \dar{r} \\
      \pos H_C \rar{\inc} & \Cube M \times \Cube S
    \end{tikzcd}
    \]
    fulfill the assumptions of \Cref{lemma:gc_equiv}.
    For this we need to show that, for all $\iota (N', C, B') \in \pos H_C$ and $(A', G') \in \pos R$ such that $(A', G') \le (A \cup C \cup N', B \cup B')$, the full subposet $\pos T \subseteq \gpos H_C$ defined as
    \[ \set{\iota(N'', C, B'') \in \gpos H_C \mid (A', G') \le (A \cup C \cup N'', B \cup B'') \text{ and } \iota(N'', C, B'') \le \iota (N', C, B')} \]
    is contractible.
    We distinguish two cases regarding $(A', G')$:
    
    Case 1: We have $A' = \ini$ and $G' \in \gpos G \setminus M$.
    Note that, if $G' \subseteq B$, then the initial object of $\gpos H_C$ is contained in $\pos T$ and we are done.
    So we can assume that $G' \cap B' \neq \ini$.
    In particular the only elements of $\gpos H_C$ that can potentially be contained in $\pos T$ are those of the form $\iota(\ini, C, \pi_{S \setminus B}(G))$ for some $G \in \pos R_{A \cup C}$.
    But this lies in $\pos T$ if and only if $\pi_{S \setminus B}(G') \subseteq \pi_{S \setminus B}(G) \subseteq B'$.
    Hence $\iota(\ini, C, \pi_{S \setminus B}(G'))$ is an initial object of $\pos T$.
    
    Case 2: We have $(A', G') = (\set m, m)$ for some $m \in \pos M$.
    First note that, if $(\set m, m) \le (A \cup C, B)$, then the initial object of $\gpos H_C$ is contained in $\pos T$ and we are done.
    Now assume $m \in N'$.
    By definition of $N$ this implies $m \subseteq B$ and hence that $\pos T$ has the single element $(\set m, C, \ini)$ and is thus contractible.
    The last case we have to consider is $m \in A \cup C$ and $m \not\subseteq B$.
    Here the only elements of $\gpos H_C$ that can potentially be contained in $\pos T$ are those of the form $\iota(\ini, C, \pi_{S \setminus B}(G))$ for some $G \in \pos R_{A \cup C}$.
    But this lies in $\pos T$ if and only if $\pi_{S \setminus B}(m) \subseteq \pi_{S \setminus B}(G) \subseteq B'$.
    Noting that $m \in A \cup C$ implies $m \in \pos R_{A \cup C}$, we obtain that $\iota(\ini, C, \pi_{S \setminus B}(m))$ is an initial object of $\pos T$.
  \end{claimproof}
  
  To show that $\restrict{(F \circ E)}{\pos H_C}$ is cartesian it is, by the claim, enough to show that $F$ is $h_C$-excisive.
  For this we want to use the induction hypothesis.
  First note that $\nisol{h_C} = \nisol{r_{A \cup C}^{S \setminus B}}$, where $r_{A \cup C}^{S \setminus B} \colon \pi_{S \setminus B}(\pos R_{A \cup C}) \to \Cube{S \setminus B}$ is the inclusion.
  Moreover
  \[ \isol{r_{A \cup C}^{S \setminus B}} \supseteq \isol{r_{A \cup C}} \cap (S \setminus B) \supseteq \isol{r_M} \cap (S \setminus B) = \isol{\sigma} \cap (S \setminus B) \]
  hence $\nisol{r_{A \cup C}^{S \setminus B}} \subseteq \nisol{\sigma} \cap (S \setminus B)$.
  So we have $\card{\nisol{h_C}} \le \card{\nisol \sigma}$ with equality if and only if $\nisol \sigma \subseteq S \setminus B$.
  But this inclusion is equivalent to $B \subseteq \isol \sigma$ which implies that $N$ is empty, in which case $\pos H_C \iso \Cube{S \setminus B}$.
  Noting that $B \neq \ini$ and hence $\card{S \setminus B} < \card S$, this implies the second of the two conditions we need to be able to apply the induction hypothesis. 
  
  It is now enough to show that $h_C$ fulfills the first condition of the induction hypothesis, i.e.\ that $\mc{N \sqcup (S \setminus B)}{\gpos H_C} \ge n$ (here we abuse notation slightly by identifying $\gpos H_C$ with its preimage under $\iota$ in $\Cube{N} \times \set{C} \times \Cube{S \setminus B} \iso \Cube{N \sqcup (S \setminus B)}$).
  For this we note the (in)equalities
  \begin{align*}
    \mc{N \sqcup (S \setminus B)}{\gpos H_C} &= \card N + \mc{S \setminus B}{\pi_{S \setminus B}(\pos R_{A \cup C})} \\
    \mc{S \setminus B}{\pi_{S \setminus B}(\pos R_{A \cup C})} &\ge \mc{S \setminus B}{\pi_{S \setminus B}(\pos R_{M})} \\
    \mc{S}{\pos R_{M}} &\le \mc{B}{\pi_{B}(\pos R_{M})} + \mc{S \setminus B}{\pi_{S \setminus B}(\pos R_{M})} \\
    \mc{S}{\pos R_{M}} &= \mc{S}{\gpos G} \ge n
  \end{align*}
  which together imply
  \[ \mc{N \sqcup (S \setminus B)}{\gpos H_C} \ge n + \card N - \mc{B}{\pi_{B}(\pos R_{M})} \]
  so that it is enough to show $\card N \ge \mc{B}{\pi_{B}(\pos R_{M})}$.
  Noting that $M_B \defeq \set{m \in M \mid m \subseteq B} \subseteq \pi_{B}(\pos R_M)$, we can further reduce this to showing $\card N \ge \mc{B}{M_B}$.
  For this we will use the following claim:
  
  \begin{claim}
    Let $(A, B) \in \Cube{M} \times \Cubegini{S}$.
    Assume that there exists $b \in B$ such that for all $m \in M$ with $b \in m \subseteq B$ we have $m \in A$.
    Then ${\inc[\pos Q]} \slice (A, B)$ is contractible, and in particular $(A, B) \in \pos L$.
  \end{claim}
  
  \begin{claimproof}
    We set $\pos Q_A = \pos Q \cap (\set A \times \Cubegini{S})$.
    Using that $(A', B') \in \pos Q$ implies $(A'', B') \in \pos Q$ for all $A' \subseteq A''$ we see that the induced functor ${\inc[\pos Q_A]} \slice (A, B) \to {\inc[\pos Q]} \slice (A, B)$ is right adjoint, so in particular a homotopy equivalence.
    Furthermore we have that $\pos Q_{A,B} \defeq \pos Q_A \cap (\set A \times \Cubegini{B}) \iso {\inc[\pos Q_A]} \slice (A, B)$.
    Now the assumption implies that if $(A, B') \in \pos Q_{A,B}$, then $(A, B' \cup \set b) \in \pos Q_{A,B}$ as well.
    In particular the functor from $\pos Q_{A,B}$ to $\pos Q_{A,B,b} \defeq \pos Q_{A,B} \cap (\set A \times \set{B' \subseteq B \mid b \in B'})$ given by $(A, B') \mapsto (A, B' \cup \set b)$ is well-defined.
    Moreover it is also left adjoint to the inclusion and hence a homotopy equivalence.
    Since $(A, \set b) \in \pos Q_{A,B,b}$ is an initial object, this finishes the proof.
  \end{claimproof}
  
  Assume $\card N < \mc{B}{M_B}$.
  Since $N \subseteq M_B$ this implies that there exists some $b \in B \setminus \bigcup_{m \in N} m$.
  Now $M_B \setminus N \subseteq A$ implies that if $b \in m \in M_B$, then $m \in A$.
  Hence we have $(A, B) \in \pos L$ by the claim.
  But now we recall from earlier that $(A, B) = l_i \not\in \pos L$, a contradiction.
\end{proof}

Together with \Cref{prop:finite_shape_eq_to_full} the last two statements yield the second main theorem of this paper.

\begin{theorem} \label{thm:cubes}
  Let $S$ be a finite set and $\sigma \colon \gpos S \to \Cube S$ a finite shape.
  Then, for a functor $F \colon \infcat C \to \infcat D$ between $\infty$-categories, being $\sigma$-excisive is equivalent to being $(n-1)$-excisive where $n = \mc S {\im \sigma}$ (as long as $\infcat C$ admits all finite colimits and $\infcat D$ admits all finite limits).
\end{theorem}

\section{The Taylor graph} \label{section:taylor}

The goal of this section is to assemble all of the various excisive approximations for different (pre)shapes into a natural coherent diagram lying under the functor we are approximating, analogous to the Taylor tower in classical Goodwillie calculus.
For this we first have to define the category this will be indexed by.
Since the existences of the approximations depend on the $\infty$-categories $\infcat C$ and $\infcat D$ between which we consider functors, the indexing category also needs to depend on these $\infty$-categories.
However, later in this section, we will also obtain a version which puts more conditions on $\infcat C$ and $\infcat D$ but in return uses a fixed indexing category.

\begin{definition}
  Let $\sigma \colon \gpos S \to \pos S$ be a preshape and $\infcat C$ and $\infcat D$ two $\infty$-categories.
  We say that \emph{$\Fun{\infcat C}{\infcat D}$ admits a universal $\sigma$-excisive approximation} if the inclusion
  \[ {\inc} \colon \Exc{\sigma}{\infcat C}{\infcat D} \longto \Fun{\infcat C}{\infcat D} \]
  has a left adjoint.
  In this case we fix an adjunction $\tP \dashv \inc$ and denote its unit by $\tp \colon \id \to \inc \circ \tP$.
\end{definition}

\begin{remark}
  If $\sigma$ is a full shape, we can, by \Cref{thm:full_approximation}, choose $\tP = \P$ and $\tp = \p$ (at least when $\infcat C$ and $\infcat D$ are nice enough).
  Here $\P$ and $\p$ are as in \cref{def:T}.
\end{remark}

\begin{notation}
  Let $\infcat C$ and $\infcat D$ be $\infty$-categories.
  We write $\PSharel{\infcat C}{\infcat D}$ for the thin category with objects the preshapes $\sigma$ such that $\Fun{\infcat C}{\infcat D}$ admits a universal $\sigma$-excisive approximation and such that there is a morphism from $\sigma$ to $\tau$ if and only if any functor $F \colon \infcat C \to \infcat D$ that is $\sigma$-excisive is also $\tau$-excisive.
\end{notation}

The main input in the construction of these diagrams is the following lemma.
Both statement and proof are due to Markus Land (though any mistake is mine).

\begin{lemma} \label{lemma:functors_from_proset}
  Let $\cat I$ be a category and $\infcat C$ an $\infty$-category.
  Denote by $\disccat{\cat I} \subseteq \cat I$ the discrete subcategory containing all objects and let $f \colon \disccat{\cat I} \to \infcat C$ be a functor such that, for all $x$ and $y$ in $\cat I$ such that there is a morphism $x \to y$, the mapping space $\Map[\infcat C]{f(x)}{f(y)}$ is contractible.
  Then there is an essentially unique extension of $f$ to a functor $\cat I \to \infcat C$.
\end{lemma}

\begin{proof}
  We want to show that the pullback $\infcat D$ (in the 1-category of simplicial sets) of the lower right corner of the middle part of the diagram
  \[
  \begin{tikzcd}
    \bdry{\Simplex n} \rar{s} \dar[hook] & \infcat D \rar{p} \dar &[10] \Fun{\cat I}{\infcat C} \dar{\Res{\inc}}\\
    \Simplex n \rar \urar[dashed] & \termCat \rar{\const[f]} & \Fun{\disccat{\cat I}}{\infcat C}
  \end{tikzcd}
  \]
  is a contractible Kan complex, i.e.\ that any map $s \colon \bdry{\Simplex n} \to \infcat D$ extends to a map $\Simplex n \to \infcat D$ as indicated on the left side of the above diagram.
  By the universal property of the pullback, this is equivalent to finding a map $t \colon \Simplex n \to \Fun{\cat I}{\infcat C}$ such that $t \circ \inc[\bdry{\Simplex n}] = p \circ s$ and ${\Res{\inc}} \circ t = \const[f]$.
  By currying $t$ corresponds to a map $\Simplex n \times \cat I \to \infcat C$ which restricts to $f \circ \pr[2]$ on $\Simplex n \times \disccat{\cat I}$ and to the curried morphism associated to $p \circ s$ on $\bdry{\Simplex n} \times \cat I$.
  Said differently, we want to find a dashed extension as in the diagram
  \[
  \begin{tikzcd}
    \Simplex n \times \disccat{\cat I} \cup \bdry{\Simplex n} \times \cat I \rar{l} \dar[hook] & \infcat C \\
    \Simplex n \times \cat I \urar[dashed] &
  \end{tikzcd}
  \]
  where $l$ is given by $f \circ \pr[2]$ on $\Simplex n \times \disccat{\cat I}$ and by $p' \circ (s \times \id[\cat I])$ on $\bdry{\Simplex n} \times \cat I$, where $p' \colon \infcat D \times \cat I \to \infcat C$ is the curried morphism associated to $p$.
  We will now factor $l$ through another $\infty$-category that only sees the information relevant to us and we can thus control.
  
  For this, note that our assumptions imply the existence of a unique extension $g$ of the composition $\hcatmap[\infcat C] \circ f$ to $\cat I$ (here $\hcatmap[\infcat C]$ is the canonical map $\infcat C \to \hcat{\infcat C}$).
  Now we write $\infcat E$ for the pullback (in the 1-category of simplicial sets) of the lower right corner in the diagram
  \[
  \begin{tikzcd}
    \disccat{\cat I} \ar[bend right, hook]{ddr}[swap]{\inc} \ar[bend left]{drr}{f} \drar[dashed]{\hat f} &[-10] & \\[-10]
    & \infcat E \dar[swap]{q} \rar{k} & \infcat C \dar{\hcatmap[\infcat C]} \\
    & \cat I \rar{g} & \hcat{\infcat C}
  \end{tikzcd}
  \]
  and note that, since the large outer diagram commutes, we obtain a map $\hat f$ making the two triangles commute (in particular $f$ factors through $\infcat E$).
  Now, by \Cref{lemma:generalized_subcat}, we obtain that the pullback $\infcat E$ is again an $\infty$-category and that $q$ can be identified with $\hcatmap[\infcat E]$, so in particular that $q$ is essentially surjective.
  Using our condition on $f$, the lemma also implies that $q$ is fully faithful, hence a categorical equivalence.
  (Note that we cannot just use the subcategory spanned by the essential image $\im g$ instead of $\infcat E$ since, when $g$ is not full, there could be morphisms in $\im g$ that our conditions do not control.)
  
  To see that $l$ actually factors through $\infcat E$, we will first show that $p'$ factors through $\infcat E$.
  For this, consider the commutative diagram
  \[
  \begin{tikzcd}
    \infcat D \rar{p} \dar &[10] \Fun{\cat I}{\infcat C} \dar[swap]{\Res{\inc}} \rar{\hcatmap[\infcat C] \circ} & \Fun{\cat I}{\hcat{\infcat C}} \dar{\Res{\inc}} \\
    \termCat \rar{\const[f]} & \Fun{\disccat{\cat I}}{\infcat C} \rar{\hcatmap[\infcat C] \circ} & \Fun{\disccat{\cat I}}{\hcat{\infcat C}}
  \end{tikzcd}
  \]
  and note that $(\hcatmap[\infcat C] \;\circ) \circ p$ is just $\const[g]$ since the subcategory of $\Fun{\cat I}{\hcat{\infcat C}}$ lying over (the discrete subcategory spanned by) the object $\hcatmap[\infcat C] \circ f \in \Fun{\disccat{\cat I}}{\hcat{\infcat C}}$ has the single object $g$ and is discrete.
  Hence $p'$ fits into a commutative diagram of the form
  \[
  \begin{tikzcd}
    \infcat D \times \disccat{\cat I} \ar[bend right]{ddr}[swap]{{\inc} \circ \pr[2]} \ar[bend left]{drr}{f \circ \pr[2]} \drar[hook, start anchor = south east]{\id[\infcat D] \times \inc} & & \\
     & \infcat D \times \cat I \rar{p'} \dar[swap]{\pr[\cat I]} & \infcat C \dar{\hcatmap[\infcat C]} \\
     & \cat I \rar{g} & \hcat{\infcat C}
  \end{tikzcd}
  \]
  and we obtain a functor $\hat p' \colon \infcat D \times \cat I \to \infcat E$ such that $k \circ \hat p' = p'$ and $q \circ \hat p' = \pr[\cat I]$.
  Thus also $k \circ \hat p' \circ (\id[\infcat D] \times \inc) = f \circ \pr[2]$ and $q \circ \hat p' \circ (\id[\infcat D] \times \inc) = {\inc} \circ \pr[2]$, which implies $\hat p' \circ (\id[\infcat D] \times \inc) = \hat f \circ \pr[2]$ by the universal property of $\infcat E$.
  
  Now we can construct the diagram
  \[
  \begin{tikzcd}
    \Simplex n \times \disccat{\cat I} \cup \bdry{\Simplex n} \times \cat I \rar{\hat l} \dar[hook] & \infcat E \rar{k} \dar{q} & \infcat C \\
    \Simplex n \times \cat I \rar{\pr[\cat I]} \urar[dashed] & \cat I &
  \end{tikzcd}
  \]
  where $\hat l$ is given by $\hat f \circ \pr[2]$ on $\Simplex n \times \disccat{\cat I}$ and by $\hat p' \circ (s \times \id[\cat I])$ on $\bdry{\Simplex n} \times \cat I$.
  It follows from what we said before that $\hat l$ is well-defined, that the diagram commutes (without the dashed arrow), and that we actually have $k \circ \hat l = l$.
  Remembering that $q$ is a categorical equivalence and, by \Cref{lemma:hcatmap_categorical_fibration}, also a categorical fibration, we obtain the desired dashed lift since trivial categorical fibrations have the right lifting property against inclusions of simplicial sets.
\end{proof}

To use the above lemma we need information about mapping spaces in slice categories (since we want a diagram which lies under a given functor in a coherent way) which the following lemma and its consequences will provide.

\begin{lemma} \label{lemma:maps_from_unit}
  Let $\infcat C$ be an $\infty$-category and $i \colon \infcat C_0 \to \infcat C$ the inclusion of a full subcategory such that there is an adjunction $l \dashv i$ with unit $\eta \colon \id \to i \circ l$.
  Further let $C \in \infcat C$, $C_0 \in \infcat C_0$, and $f \colon C \to i(C_0)$ be a morphism in $\infcat C$.
  Then $\Map[\undercat{\infcat C}{C}]{\eta(C)}{f}$ is contractible.
\end{lemma}

\begin{proof}
  By \cite[Lemma 5.5.5.12]{LurHTT}, we have that $\Map[\undercat{\infcat C}{C}]{\eta(C)}{f}$ is a homotopy fiber of the map $(\circ\; \eta(C)) \colon \Map[\infcat C]{i(l(C))}{i(C_0)} \to \Map[\infcat C]{C}{i(C_0)}$ over the point $f$.
  But, by \Cref{lemma:unit_classical}, the composition
  \[ \Map[\infcat C_0]{l(C)}{C_0} \xlongto{i} \Map[\infcat C]{i(l(C))}{i(C_0)} \xlongto{\circ \eta(C)} \Map[\infcat C]{C}{i(C_0)} \]
  is an equivalence.
  Now we note that the first map in this composition is an equivalence since $i$ is fully faithful.
  Hence the second map is also an equivalence and thus the homotopy fiber we are interested in is trivial.
\end{proof}

\begin{corollary} \label{lemma:maps_from_local_univ}
  Let $\sigma \colon \gpos S \to \pos S$ be a preshape and $\infcat C$ and $\infcat D$ $\infty$-categories such that $\Fun{\infcat C}{\infcat D}$ admits a universal $\sigma$-excisive approximation.
  Furthermore, let $\alpha \colon F \to G$ be a natural transformation of functors $\infcat C \to \infcat D$ such that $G$ is $\sigma$-excisive.
  Then $\Map[\undercat{\Fun{\infcat C}{\infcat D}}{F}]{\tp(F)}{\alpha}$ is contractible.
\end{corollary}

\begin{lemma}
  Let $\infcat C$ be an $\infty$-category and $i \colon \infcat C_0 \to \infcat C$ the inclusion of a full subcategory such that there is an adjunction $l \dashv i$ with unit $\eta \colon \id \to i \circ l$.
  Furthermore, let $\infcat E$ be an $\infty$-category, $F \colon \infcat E \to \infcat C$ and $F_0 \colon \infcat E \to \infcat C_0$ functors, and $\alpha \colon F \to i \circ F_0$ a natural transformation.
  Then $\Map[\undercat{\Fun{\infcat E}{\infcat C}}{F}]{\eta \circ F}{\alpha}$ is contractible.
\end{lemma}

\begin{proof}
  By \Cref{lemma:composing_with_adjunction}, there is an adjunction with left adjoint $(l \;\circ) \colon \Fun{\infcat E}{\infcat C} \to \Fun{\infcat E}{\infcat C_0}$, right adjoint $(i \;\circ) \colon \Fun{\infcat E}{\infcat C_0} \to \Fun{\infcat E}{\infcat C}$, and unit $(\eta \;\circ)$.
  Noting that $(i \;\circ)$ is the inclusion of a full subcategory, we can apply \Cref{lemma:maps_from_unit} to obtain the desired statement.
\end{proof}

\begin{corollary} \label{cor:maps_from_univ}
  Let $\sigma \colon \gpos S \to \pos S$ be a preshape and $\infcat C$ and $\infcat D$ $\infty$-categories such that $\Fun{\infcat C}{\infcat D}$ admits a universal $\sigma$-excisive approximation.
  Furthermore, let $A \colon \Fun{\infcat C}{\infcat D} \to \Exc{\sigma}{\infcat C}{\infcat D}$ be a functor and $\alpha \colon \id[\Fun{\infcat C}{\infcat D}] \to {\inc} \circ A$ a natural transformation.
  Then the space of maps from $\tp$ to $\alpha$ in $\undercat{\Fun{\Fun{\infcat C}{\infcat D}}{\Fun{\infcat C}{\infcat D}}}{\id}$ is contractible.
\end{corollary}

We now know enough to be able to construct a diagram as promised in the beginning of this section.
There are two versions: one which also takes into account maps between functors (i.e.\ makes the naturality of the diagram precise), and a second one only considering a single functor.

\begin{theorem} \label{thm:taylor_graph}
  Let $\infcat C$ and $\infcat D$ be two $\infty$-categories.
  \begin{enumerate}[label=\alph*)]
    \item There is an essentially unique functor
    \[ \Taylor \colon \op{(\PSharel{\infcat C}{\infcat D})} \longto \undercat{\Fun{\Fun{\infcat C}{\infcat D}}{\Fun{\infcat C}{\infcat D}}}{\id} \]
    such that $\Taylor(\sigma) = \tp$.
    
    \item Let $F \colon \infcat C \to \infcat D$ be a functor.
    Then there is an essentially unique functor
    \[ \Taylor(F) \colon \op{(\PSharel{\infcat C}{\infcat D})} \longto \undercat{\Fun{\infcat C}{\infcat D}}{F} \]
    such that $\Taylor(F)(\sigma) = \tp(F)$.
  \end{enumerate}
  In particular, we obtain such functors for any subcategory of $\PSharel{\infcat C}{\infcat D}$.
  These restricted functors are also essentially unique.
\end{theorem}

\begin{proof}
  We want to apply \Cref{lemma:functors_from_proset} to obtain the desired functors.
  For the first point, note that, if there is a map $\sigma \to \tau$ in $\PSharel{\infcat C}{\infcat D}$, then $\tP$ lands in $\Exc{\tau}{\infcat C}{\infcat D}$, hence \Cref{cor:maps_from_univ} implies that $\Map{\tp[\tau]}{\tp}$ is contractible.
  The second follows in the same way, using \Cref{lemma:maps_from_local_univ} instead of \Cref{cor:maps_from_univ}.
  The essential uniqueness of the restrictions also follows from \Cref{lemma:functors_from_proset}.
\end{proof}

\begin{remark}
  Using the essential uniqueness of $\Taylor(F)$, we see that it is equivalent to the functor obtained from $\Taylor$ by postcomposing with evaluation at $F$, which explains our notation.
\end{remark}

We now describe a version of these diagrams with a single indexing category independent of the $\infty$-categories $\infcat C$ and $\infcat D$.

\begin{notation}
  We write $\Shafin$ for the thin (i.e.\ each hom-set has at most one element) category with
  \begin{itemize}
    \item objects the collection of (small) finite shapes;
    \item a morphism from $\sigma$ to $\tau$ if and only if, for all $\infty$-categories $\infcat C$ and $\infcat D$ such that $\infcat C$ has a terminal object and admits all finite colimits and such that $\infcat D$ is differentiable, each functor $F \colon \infcat C \to \infcat D$ that is $\sigma$-excisive is also $\tau$-excisive.
  \end{itemize}
\end{notation}

\begin{definition}
  We will say that two finite shapes are \emph{equivalent} if they are isomorphic in $\Shafin$.
\end{definition}

\begin{corollary}
  Let $\infcat C$ be an $\infty$-category that has a terminal object and admits all finite colimits, and $\infcat D$ a differentiable $\infty$-category.
  \begin{enumerate}[label=\alph*)]
    \item There is an essentially unique functor
    \[ \Taylor \colon \op{(\Shafin)} \longto \undercat{\Fun{\Fun{\infcat C}{\infcat D}}{\Fun{\infcat C}{\infcat D}}}{\id} \]
    such that $\Taylor(\sigma) = \tp$.
    
    \item Let $F \colon \infcat C \to \infcat D$ be a functor.
    Then there is an essentially unique functor
    \[ \Taylor(F) \colon \op{(\Shafin)} \longto \undercat{\Fun{\infcat C}{\infcat D}}{F} \]
    such that $\Taylor(F)(\sigma) = \tp(F)$.
  \end{enumerate}
  In particular, we obtain such functors for any subcategory of $\Shafin$.
  These restricted functors are also essentially unique.
\end{corollary}

\begin{proof}
  By \Cref{thm:finite_approximation}, there is a functorial inclusion $\Shafin \to \PSharel{\infcat C}{\infcat D}$.
  Then \Cref{thm:taylor_graph} implies the statement.
\end{proof}

\begin{remark}
  There are analogous statements for shapes with higher cardinality bounds (as long as we restrict ourselves to full shapes).
  However, they are less useful as the necessary differentiability condition on $\infcat D$ becomes very strong.
\end{remark}

\subsection{Relation to the Taylor tower}

\begin{notation}
  We write $\Shafinnb$ for the full (thin) subcategory of $\Shafin$ spanned by the shapes equivalent to a finite non-inane full shape.
\end{notation}

The following proposition tells us that restricting ourselves to $\Shafinnb$ drops precisely the shapes with uninteresting excision properties.

\begin{proposition}
  A finite shape $\sigma$ is equivalent to a finite inane full shape if and only if, for all $\infty$-categories $\infcat C$ and $\infcat D$ such that $\infcat C$ has a terminal object and admits all finite colimits and such that $\infcat D$ is differentiable, each functor $F \colon \infcat C \to \infcat D$ is $\sigma$-excisive.
  Furthermore, a finite non-inane full shape is not equivalent to a finite inane full shape.
  
  In particular, the finite shapes equivalent to a finite inane full shape are precisely the terminal objects of $\Shafin$, and $\Shafinnb$ consists precisely of the non-terminal objects of $\Shafin$.
\end{proposition}

\begin{proof}
  By \Cref{prop:inane}, any shape equivalent to a finite inane full shape has trivial excision properties.
  This shows one direction of the first statement.
  
  Moreover, by \Cref{thm:nb-excisive_implies_n-excisive}, for any $\sigma \in \Shafinnb$, there exists an $n \in \NN$ such that $\sigma$-excisive implies $n$-excisive.
  Since being $n$-excisive is a non-trivial condition (i.e.\ there exists a functor $\infcat C \to \infcat D$ (with $\infcat C$ and $\infcat D$ as above) that is not $n$-excisive; one example is the functor $F$ from pointed spaces to spectra given by $X \mapsto \Suspspec(X^{\smashprod (n + 1)})$ as it is not weakly constant but still $(n+1)$-reduced, i.e.\ $\P[n](F)$ is terminal; see \cite[Remark 1.16]{Goo03}), we obtain that being $\sigma$-excisive is also a non-trivial condition.
  In particular, no shape in $\Shafinnb$ is equivalent to a finite inane full shape.
  This shows the second statement.
  
  For the other implication of the first statement, note that any finite shape is equivalent to a finite full shape by \Cref{prop:finite_shape_eq_to_full}.
  But, if it has trivial excision properties, it cannot be equivalent to a finite non-inane full shape by the same argument as for the second statement.
  So it must be equivalent to a finite inane full shape.
  
  The last statement now follows by noting that, since there exists a finite shape $\sigma$ such that any functor is $\sigma$-excisive (see \Cref{ex:inane_shape}), these shapes are precisely the terminal objects of $\Shafin$.
\end{proof}

We also have the following direct corollary of \Cref{thm:equiv_to_retract_of_free} and \Cref{prop:finite_shape_eq_to_full}, which tells us a bit more about the structure of $\Shafinnb$:

\begin{corollary} \label{cor:fin_many_excision_notions}
  Let $\gpos S$ be a finite poset.
  Then the full subcategory of $\Shafinnb$ spanned by the shapes with domain $\gpos S$ (i.e.\ those of the form $\sigma \colon \gpos S \to \pos S$ for some $\pos S$) has only finitely many isomorphism classes.
\end{corollary}

We will now state and prove the third main theorem of this paper, relating our class of finite non-inane shapes to the classical cubes.

\begin{theorem} \label{thm:tower_is_initial}
  The functor $\Tower \colon \op{(\NN)} \to \op{(\Shafinnb)}$ given by sending $n$ to $\cubeinc{n}$ is homotopy initial.
\end{theorem}

\begin{proof}
  First note that $\Tower$ is well-defined since $\cubeinc{n+1}$ is a shape (see \Cref{ex:shapes}) that is full but not inane (see \Cref{ex:cube_not_inane}) and since $n$-excisive implies $(n+1)$-excisive (see \Cref{cor:n-excisive_implies_m-excisive}).
  To see that it is homotopy initial we need to show that, for all $\sigma \in \op{(\Shafinnb)}$, the category ${\Tower} \slice \sigma$ is contractible.
  However, this is just the full subposet $\pos P \subseteq \op{(\NN)}$ spanned by those $n$ such that $\sigma$-excisive implies $(n-1)$-excisive.
  By definition of $\Shafinnb$ we can assume that $\sigma$ is full and non-inane.
  Hence, by \Cref{thm:nb-excisive_implies_n-excisive}, the poset $\pos P$ is not empty, and, by \Cref{cor:n-excisive_implies_m-excisive}, it is closed below (i.e.\ if it contains $n$ then it contains all $m$ such that $m \le n$ in $\op{(\NN)}$).
  But any non-empty subposet of $\op{(\NN)}$ that is downward closed is isomorphic to $\op{(\NN)}$ and hence contractible (since it has a terminal object).
\end{proof}

\begin{remark}
  This theorem tells us in particular that the Taylor graph $\Taylor(F)$ of a functor $F$ converges at an object $X$ if and only if its Taylor tower $\Taylor(F) \circ \Tower$ does (here convergence at $X$ means that the canonical map from $F(X)$ to the limit of the respective diagram is an equivalence).
  In particular any convergence criteria for the tower, such as the analytic functors of Goodwillie (cf.\ \cite[Definition 4.2]{Goo92} and \cite[Theorem 1.13]{Goo03}), can also be applied to the graph.
\end{remark}

It is also possible to rephrase \Cref{thm:cubes} in terms of the functor $\Tower$:

\begin{theorem} \label{thm:cubes_equivalence}
  The functor $\Tower \colon \op{(\NN)} \to \op{(\Shafinnb)}$ of \Cref{thm:tower_is_initial} becomes an equivalence when the codomain is restricted to the full subcategory of $\op{(\Shafinnb)}$ spanned by the finite, non-inane shapes that have a cube as codomain (i.e.\ those of the form $\gpos S \to \Cube{n}$ for some (finite) poset $\gpos S$ and $n \in \NN$).
\end{theorem}

\Cref{thm:tower_is_initial,thm:cubes_equivalence} suggest the following conjecture.
All evidence known to the author, including the two theorems, points towards it being true.

\begin{conjecture}
  The functor $\Tower$ is an equivalence of categories.
  Or, equivalently, any shape in $\Shafinnb$ is equivalent to $\cubeinc{n}$ for some $n \in \NN$.
\end{conjecture}

An answer in the affirmative would provide even more compelling evidence that the cubes are the ``correct'' shapes to use for functor calculus.
If the conjecture were false that would also be very interesting: in that case the Taylor graph would be a finer resolution of the tower, potentially containing additional information.

\begin{appendices}

\crefalias{section}{appsec}
\crefalias{subsection}{appsec}

\section{The calculus of mates} \label{section:mates}

In this appendix we recall the mate construction as well as a number of lemmas concerning it, which are quite useful when working with adjunctions and natural transformations.
Since this is not supposed to be a comprehensive exposition of the topic, we will be brief and only state and give references for the statements we will use.
A concise summary of these, and a few more, important statements, though without proofs, can be found in \cite[Appendix A]{GPS}.
A longer exposition with proofs is given (in French) in \cite[Section 1.1.2]{Ayo}.

\begin{notation}
  Suppose we have, in a (strict) 2-category, a diagram of the form
  \[
  \begin{tikzcd}
    A \rar{a} \dar[swap]{h} & B \dar{k} \dlar[Rightarrow, shorten < = 10, shorten > = 10, swap]{\alpha} \\
    C \rar{c} & D
  \end{tikzcd}
  \]
  and fixed adjunctions $a_! \dashv a$ and $c_! \dashv c$.
  In this situation we write $\mate \alpha$ for the \emph{mate} of $\alpha$, which is a 2-morphism of the form
  \[
  \begin{tikzcd}
    A \dar[swap]{h} & B \dar{k} \lar[swap]{a_!} \\
    C & D \lar[swap]{c_!} \ular[Rightarrow, shorten < = 10, shorten > = 10, swap]{\mate \alpha}
  \end{tikzcd}
  \]
  defined as the composition
  \[ c_! k \xlongto{c_! k \eta_a} c_! k a a_! \xlongto{c_! \alpha a_!} c_! c h a_! \xlongto{\epsilon_c h a_!} h a_! \]
  where $\eta_a$ and $\epsilon_c$ are the unit respectively counit of the adjunctions $a_! \dashv a$ respectively $c_! \dashv c$ fixed above.
\end{notation}

The following is a property of the mate that follows easily from the definitions (and actually characterizes it uniquely).

\begin{lemma} \label{lemma:mate_and_units}
  Let the following be a diagram in a 2-category:
  \[
  \begin{tikzcd}
    A \rar{a} \dar[swap]{h} & B \dar{k} \dlar[Rightarrow, shorten < = 10, shorten > = 10, swap]{\alpha} \\
    C \rar{c} & D
  \end{tikzcd}
  \]
  and $a_! \dashv a$ and $c_! \dashv c$ two fixed adjunctions.
  Then the following two diagrams commute:
  \[
  \begin{tikzcd}
    k \rar{k \eta_a} \dar[swap]{\eta_c k} & k a a_! \dar{\alpha a_!} & & c_! k a \rar{\mate \alpha a} \dar[swap]{c_! \alpha} & h a_! a \dar{h \epsilon_a} \\
    c c_! k \rar{c \mate \alpha} & c h a_! & & c_! c h \rar{\epsilon_c h} & h
  \end{tikzcd}
  \]
  where $\eta$ and $\epsilon$ denote the respective (co)units and $\mate \alpha$ is the mate of $\alpha$.
\end{lemma}

\begin{proof}
  This is (the dual of) \cite[Proposition 1.1.9]{Ayo}.
\end{proof}

The following two lemmas express a certain functoriality of the mate construction with respect to pasting of squares.
A more abstract (and maybe conceptual) way to formulate them is to present the mate construction as an isomorphism of certain double categories.
This can be found in \cite[Proposition 2.2]{KS}.

\begin{lemma}[Pasting law I]
  Let the following be a diagram in a 2-category and its paste:
  \[
  \begin{tikzcd}[sep = 35]
    A \rar{a} \dar[swap]{h} & B \dar{k} \dlar[Rightarrow, shorten < = 15, shorten > = 15, swap]{\alpha} \rar{b} & E \dlar[Rightarrow, shorten < = 15, shorten > = 15, swap]{\beta} \dar{l} & & A \rar{ba} \dar[swap]{h} & E \dar{l} \dlar[Rightarrow, shorten < = 15, shorten > = 15, swap]{\alpha \paste \beta} \\
    C \rar{c} & D \rar{d} & F & & C \rar{dc} & F
  \end{tikzcd}
  \]
  and $a_! \dashv a$, $b_! \dashv b$, $c_! \dashv c$, and $d_! \dashv d$ four fixed adjunctions.
  We obtain mates $\alpha_!$ and $\beta_!$ that fit into diagrams of the form
  \[
  \begin{tikzcd}[sep = 35]
    A \dar[swap]{h} & B \lar[swap]{a_!} \dar{k} & E \lar[swap]{b_!} \dar{l} & & A \dar[swap]{h} & E \lar[swap]{a_! b_!} \dar{l} \\
    C & D \lar[swap]{c_!} \ular[Rightarrow, shorten < = 15, shorten > = 15, swap]{\mate \alpha} & F \lar[swap]{d_!} \ular[Rightarrow, shorten < = 15, shorten > = 15, swap]{\mate \beta} & & C & F \lar[swap]{c_! d_!} \ular[Rightarrow, shorten < = 15, shorten > = 15, swap]{\mate \alpha \paste \mate \beta}
  \end{tikzcd}
  \]
  and it holds that $\mate \alpha \paste \mate \beta = \mate{(\alpha \paste \beta)}$, where, for the latter mate, we use the adjunctions $a_! b_! \dashv b a$ and $c_! d_! \dashv d c$ given by composing the original ones.
\end{lemma}

\begin{proof}
  This is (the dual of) \cite[Proposition 1.1.11]{Ayo}\footnote{Note that the composition $\mate \alpha \paste \mate \beta$ is erroneously written the wrong way around there, and that what is actually proven is the dual version we stated.}.
\end{proof}

\begin{lemma}[Pasting law II]
  Let the following be a diagram in a 2-category and its paste:
  \[
  \begin{tikzcd}[sep = 35]
    A \rar{a} \dar[swap]{h} & B \dar{k} \rar{b} & E \dar{l} & & A \rar{ba} \dar[swap]{h} & E \dar{l} \\
    C \rar{c} \urar[Rightarrow, shorten < = 15, shorten > = 15]{\alpha} & D \rar{d} \urar[Rightarrow, shorten < = 15, shorten > = 15]{\beta} & F & & C \rar{dc} \urar[Rightarrow, shorten < = 15, shorten > = 15]{\beta \paste \alpha} & F
  \end{tikzcd}
  \]
  and $h_! \dashv h$, $k_! \dashv k$, and $l_! \dashv l$ three fixed adjunctions.
  We obtain mates $\alpha_!$ and $\beta_!$ that fit into diagrams of the form
  \[
  \begin{tikzcd}[sep = 35]
    A \rar{a} & B \rar{b} & E & & A \rar{ba} & E \\
    C \uar{h_!} \rar{c} & D \uar{k_!} \rar{d} \ular[Rightarrow, shorten < = 15, shorten > = 15, swap]{\mate \alpha} & F \uar[swap]{l_!} \ular[Rightarrow, shorten < = 15, shorten > = 15, swap]{\mate \beta} & & C \uar{h_!} \rar{dc} & F \uar[swap]{l_!} \ular[Rightarrow, shorten < = 15, shorten > = 15, swap]{\mate \alpha \paste \mate \beta}
  \end{tikzcd}
  \]
  and it holds that $\mate \alpha \paste \mate \beta = \mate{(\beta \paste \alpha)}$.
\end{lemma}

\begin{proof}
  This is (the dual of) \cite[Proposition 1.1.12]{Ayo}.
\end{proof}

\begin{lemma} \label{lemma:mate_equiv}
  Let the following be a diagram in a 2-category:
  \[
  \begin{tikzcd}
    A \rar{a} \dar[swap]{h} & B \dar{k} \dlar[Rightarrow, shorten < = 10, shorten > = 10, swap]{\alpha} \\
    C \rar{c} & D
  \end{tikzcd}
  \]
  and $a_! \dashv a$ and $c_! \dashv c$ two fixed adjunctions.
  Furthermore assume that $h$ and $k$ are isomorphisms, and that $\alpha$ is a 2-isomorphism.
  Then the mate $\mate \alpha$ is a 2-isomorphism.
\end{lemma}

\begin{proof}
  First note that if $a = c$ (with the same fixed adjunction) and $h$, $k$, and $\alpha$ are all identities, then the mate $\mate \alpha$ is the identity (by one of the triangle identities).
  For the general case consider the diagram
  \[
  \begin{tikzcd}
    A \rar{a} \dar[swap]{h} & B \dar{k} \dlar[Rightarrow, shorten < = 10, shorten > = 10, swap]{\alpha} \\
    C \rar{c} \dar[swap]{\inv h} & D \dar{\inv k} \dlar[Rightarrow, shorten < = 10, shorten > = 10, swap]{\beta} \\
    A \rar{a} & B
  \end{tikzcd}
  \]
  where $\beta$ is given by
  \[ \inv k c = \inv k c h \inv h \xlongto{\inv k \inv \alpha \inv h} \inv k k a \inv h = a \inv h \]
  and note that the paste $\beta \paste \alpha$ is the identity.
  This follows from the diagram
  \[
  \begin{tikzcd}[row sep = 10, column sep = 60]
    \inv k k a \rar{\inv k \alpha} \dar[equal] & \inv k c h \rar{\beta h} \dar[equal] & a \inv h h \dar[equal] \\
    \inv k k a \inv h h \rar{\inv k \alpha \inv h h} \dar[equal] & \inv k c h \inv h h \rar{\inv k \inv \alpha \inv h h} & \inv k k a \inv h h \dar[equal] \\
    a \ar{rr}{\id} & & a
  \end{tikzcd}
  \]
  being commutative.
  Thus, by the pasting law for mates, we obtain that $\inv h \mate \alpha$ has a right inverse (namely $\mate \beta k$).
  In the same way we can show that $\mate \alpha \inv k$ has a left inverse.
  Since $\inv h$ and $\inv k$ are both isomorphisms, this implies that $\mate \alpha$ has both a left and a right inverse and thus is a 2-isomorphism.
\end{proof}

\begin{remark}
  The statement of \Cref{lemma:mate_equiv} is still true when we only require $h$ and $k$ to be equivalences.
  Moreover the converse is also true, i.e.\ if $\mate \alpha$ is a 2-isomorphism, then $\alpha$ is as well.
  (See \cite[Appendix A]{GPS}.)
\end{remark}

\begin{remark}
  Naturally, there is also a dual version of everything we have done here, using right adjoints instead of left adjoints.
\end{remark}

\section{Basic \texorpdfstring{$\infty$}{infty}-categorical facts} \label{section:facts}

This appendix consists of a collection of basic $\infty$-categorical facts that are used throughout this paper.
They are simply stated here, to make it easier to quickly remind oneself of them.
The references (or proofs) can be found in \Cref{section:basics}.
Note that, even though we often only state things for indexing categories (instead of $\infty$-categories or simplicial sets), this is purely for convenience and there are more general versions of all of these statements.

\begin{restatable}{lemma}{lemmaCompositionOfKan}
  Let $f \colon I \to J$ and $g \colon J \to K$ be maps of simplicial sets and $\infcat C$ an $\infty$-category that is both weakly left $f$-extensible and weakly left $g$-extensible.
  Then it is also weakly left $(g \circ f)$-extensible, and we have $\Lan{g \circ f} \eq \Lan{g} \circ \Lan{f}$.
\end{restatable}

\begin{restatable}{lemma}{lemmaKanUnit} \label{lemma:fully_faithful_kan_unit}
  Let $f \colon \cat I \to \cat J$ be a fully faithful functor between categories and $\infcat C$ a left $f$-extensible $\infty$-category.
  Then the unit $\id \to \Res{f} \Ext[f]$ of the adjunction $\Ext[f] \dashv \Res{f}$ is an equivalence of functors $\Fun{\cat I}{\infcat C} \to \Fun{\cat I}{\infcat C}$.
\end{restatable}

\begin{restatable}{lemma}{lemmaCoconeComparison} \label{lemma:cocone_comparison}
  Let $\cat I$ be a category and $\infcat C$ an $\infty$-category.
  Then $\infcat C$ admits all colimits indexed by $\cat I$ if and only if, for all diagrams $D \colon \cat I \to \infcat C$, there is a colimit diagram extending $D$.
  In this case a diagram $D \colon \cocone{\cat I} \to \infcat C$ lies in the essential image of $\Lan{\inc} \colon \Fun{\cat I}{\infcat C} \to \Fun{\cocone{\cat I}}{\infcat C}$ if and only if it is a colimit diagram.
\end{restatable}

\begin{restatable}{lemma}{lemmaHomotopyTerminalStuff} \label{lemma:homotopy_terminal_stuff}
  Let $f \colon \cat I \to \cat J$ be a functor between categories.
  \begin{enumerate}[label=\alph*)]
    \item If $\cat J$ has a terminal object, then it is contractible.
    \item The functor $f$ is homotopy terminal if and only if, for each $j \in \cat J$, the category $j \slice f$ is contractible.
    \item If $f$ is right adjoint, then it is homotopy terminal.
    \item If $f$ is homotopy terminal, then it is a homotopy equivalence.
  \end{enumerate}
\end{restatable}

\begin{restatable}{lemma}{lemmaHomotopyTerminal}
  Let $f \colon \cat I \to \cat J$ be a homotopy terminal functor between categories and $\infcat C$ an $\infty$-category that admits colimits indexed both by $\cat I$ and by $\cat J$.
  Then the natural transformation $f_* \colon \colim{\cat I} \Res{f} \to \colim{\cat J}$ of functors $\Fun{\cat J}{\infcat C} \to \infcat C$ is an equivalence.
\end{restatable}

\begin{restatable}{lemma}{lemmaPreservation} \label{lemma:preservation}
  Let $\cat I$ be a category and $F \colon \infcat C \to \infcat D$ a functor between $\infty$-categories that admit colimits indexed by $\cat I$.
  Then the following conditions are equivalent:
  \begin{enumerate}[label=\alph*)]
    \item $F$ preserves left Kan extension along the inclusion ${\inc} \colon \cat I \to \cocone{\cat I}$.
    \item $F$ preserves colimits indexed by $\cat I$.
    \item $F$ sends $\cocone{\cat I}$-indexed colimit diagrams to colimit diagrams.
  \end{enumerate}
\end{restatable}

\begin{restatable}{lemma}{lemmaResPreservesKan}
  Let $f \colon \cat I \to \cat J$ be a functor between categories, $g \colon \infcat K \to \infcat L$ a functor between $\infty$-categories, and $\infcat C$ a left $f$-extensible $\infty$-category.
  Then $\Fun{\infcat L}{\infcat C}$ is left $f$-extensible, and ${\Res{g}} \colon \Fun{\infcat L}{\infcat C} \to \Fun{\infcat K}{\infcat C}$ preserves left Kan extension along $f$.
\end{restatable}

\section{Tools for Kan extensions and (co)limits} \label{section:basics}

In this appendix we collect some basic tools for working with Kan extensions and (co)limits in $\infty$-categories that we need in the rest of this paper.
Note that, even though we often only state things for indexing categories (instead of $\infty$-categories or simplicial sets), this is purely for convenience and there are more general versions of most of those statements.
Generally, if there is a pair of dual statements, we will only give one of them and leave the other implicit.

There is no claim of originality for any of the statements found in this appendix (the correctness of most, if not all, of them should be more or less clear to anyone familiar with the theory); the ones for which a proof is given are merely those for which the author could not find a reference.

\subsection{Kan extensions}

\lemmaCompositionOfKan*

\begin{proof}
  Since adjunctions compose, we have that $\Lan{g} \circ \Lan{f}$ is a left adjoint of the composition $\Res{f} \circ \Res{g} = \Res{g \circ f}$.
  As adjoints are unique up to isomorphism (in the homotopy 2-category of $\infty$-categories), we obtain that $\Lan{g} \circ \Lan{f} \eq \Lan{g \circ f}$.
\end{proof}

\lemmaKanUnit*

\begin{proof}
  This follows from the Beck-Chevalley condition \cite[Lemma 12.3.11]{RV}, using that by \cite[Lemma 9.4.4]{RV} when $f$ is fully faithful a certain square fulfills a condition called exact (here we use that the nerve functor is cosmological by \cite[Example 1.3.5]{RV}, hence preserves fully faithfulness (cf.\ \cite[Corollary 3.5.6]{RV}) since it preserves absolute right and left lifting diagrams by \cite[Proposition 10.1.4]{RV}).
\end{proof}

\begin{lemma} \label{lemma:extensions_are_cocartesian}
  Let $f \colon \cat I \to \cat J$ be a fully faithful functor between categories and $\infcat C$ a left $f$-extensible $\infty$-category.
  Then $\epsilon \circ \Lan{f} \colon \Lan{f} \Res{f} \Lan{f} \to \Lan{f}$ is an equivalence, where $\epsilon$ is the counit of the adjunction $\Lan{f} \dashv \Res{f}$.
\end{lemma}

\begin{proof}
  Consider the diagram
  \[
  \begin{tikzcd}
    \Lan{f} \rar{\id} \dar[swap]{\eta} & \Lan{f} \\
    \Lan{f} (\Res{f} \Lan{f}) \rar[equal] & (\Lan{f} \Res{f}) \Lan{f} \uar[swap]{\epsilon}
  \end{tikzcd}
  \]
  where the vertical maps are given by the unit respectively counit of the adjunction $\Lan{f} \dashv \Res{f}$.
  It commutes up to homotopy by one of the triangle identities.
  Since $f$ is fully faithful, the left vertical morphism is an equivalence.
  Hence the right vertical map is an equivalence as well.
\end{proof}


\begin{lemma} \label{lemma:kan_and_currying}
  Let $J$ be a simplicial set, $f \colon I \to I'$ a map of simplicial sets, and $\infcat C$ a weakly left $f$-extensible $\infty$-category.
  Then there is a homotopy commutative diagram of the form
  \[
  \begin{tikzcd}
    \Fun{I}{\Fun{J}{\infcat C}} \rar{\Lan{f}} & \Fun{I'}{\Fun{J}{\infcat C}} \\
    \Fun{I \times J}{\infcat C} \rar{\Lan{f \times \id}} \uar{\iso} \dar[swap]{\iso} & \Fun{I' \times J}{\infcat C} \uar[swap]{\iso} \dar{\iso} \\
    \Fun{J}{\Fun{I}{\infcat C}} \rar{\Lan{f} \circ} & \Fun{J}{\Fun{I'}{\infcat C}}
  \end{tikzcd}
  \]
  where the vertical maps are the respective currying isomorphisms (in particular, all of these left Kan extensions actually exist).
\end{lemma}

\begin{proof}
  Note that $(\Lan{f} \;\circ)$ is left adjoint to $(\Res{f} \;\circ)$ by \Cref{lemma:composing_with_adjunction}.
  Since the above diagram with the restrictions, instead of their left adjoints, commutes, we obtain that $\Res{f}$ and $\Res{f \times \id}$ actually have left adjoints.
  Then \Cref{lemma:mate_equiv} implies the statement.
\end{proof}

\begin{lemma} \label{lemma:colimits_and_currying}
  Let $I$ and $J$ be simplicial sets, $\infcat C$ an $\infty$-category that admits colimits indexed by $I$, and $D \colon I \times J \to \infcat C$ a functor.
  Denote by $D_{I} \colon I \to \Fun{J}{\infcat C}$ and $D_{J} \colon J \to \Fun{I}{\infcat C}$ the curried functors.
  Then $\Lan{\pr[J]} (D)$, $\colim{I} D_{I}$, and $\colim{I} \circ D_{J}$ exist and are all equivalent in $\Fun{J}{\infcat C}$.
\end{lemma}

\begin{proof}
  This is a special case of \Cref{lemma:kan_and_currying}.
\end{proof}

\begin{lemma} \label{lemma:comma_square_mate}
  Let $f \colon \cat I \to \cat K$ and $g \colon \cat J \to \cat K$ be functors between categories.
  Consider the natural transformation
  \[
  \begin{tikzcd}[sep = 30]
    f \slice g \rar{\pr[\cat I]} \dar[swap]{\pr[\cat J]} & \cat I \dar{f} \dlar[Rightarrow, shorten < = 13, shorten > = 13, swap]{\alpha} \\
    \cat J \rar{g} & \cat K
  \end{tikzcd}
  \]
  given, at $(i, j, k \colon\! f(i) \to g(j)) \in f \slice g$, by $k$.
  Now let $\infcat C$ be a left $f$-extensible and left $\pr[\cat J]$-extensible $\infty$-category.
  After applying $\Fun{\blank}{\infcat C}$ to the diagram above we get
  \[
  \begin{tikzcd}[sep = 35]
    \Fun{\cat K}{\infcat C} \rar{\Res{f}} \dar[swap]{\Res{g}} & \Fun{\cat I}{\infcat C} \dar{\Res{\pr[\cat I]}} \dlar[Rightarrow, shorten < = 22, shorten > = 22, swap]{\alpha} \\
    \Fun{\cat J}{\infcat C} \rar{\Res{\pr[\cat J]}} & \Fun{f \slice g}{\infcat C}
  \end{tikzcd}
  \]
  and taking the mate we obtain a transformation $\mate \alpha \colon \Lan{\pr[\cat J]} \Res{\pr[\cat I]} \to \Res{g} \Lan{f}$.
  This transformation $\mate \alpha$ is an equivalence.
\end{lemma}

\begin{proof}
  This follows from the fact that the second diagram satisfies the Beck-Chevalley condition by \cite[Lemma 12.3.11]{RV} as the first one is a so called exact square by \cite[Lemma 9.2.6]{RV} (again using that the nerve is a cosmological functor by \cite[Example 1.3.5]{RV} and thus preserves comma categories by \cite[Proposition 10.1.2]{RV}).
\end{proof}

\begin{lemma} \label{lemma:kan_local}
  Let $\cat I$ and $\cat J$ be categories, $f \colon \cat I \to \cat J$ a functor, $\infcat C$ a left $f$-extensible $\infty$-category, and $D \colon \cat I \to \infcat C$ a diagram.
  Then, for any $j \in \cat J$, the mate $\vartheta \colon \colim{f \slice j} \Res{\pr} \to \Res{j} \Lan{f}$ of the natural transformation in the right diagram (which is the image of the left diagram under $\Fun{\blank}{\infcat C}$)
  \[
  \begin{tikzcd}[sep = 35]
  f \slice j \rar{\const} \dar[swap]{\pr} & \termCat \dar{j} & & \Fun{\cat J}{\infcat C} \rar{\Res{f}} \dar[swap]{\Res{j}} & \Fun{\cat I}{\infcat C} \dar{\Res{\pr}} \dlar[Rightarrow, shorten < = 27, shorten > = 27, swap]{\rho_j} \\
  \cat I \rar{f} \urar[Rightarrow, shorten < = 17, shorten > = 17]{\tilde \rho_j} & \cat J & & \infcat C \rar{\Diag} & \Fun{f \slice j}{\infcat C}
  \end{tikzcd}
  \]
  is an equivalence, where $\tilde \rho_j$ is, at $(i, k \colon f(i) \to j)$, just given by $k$.
  Furthermore it is natural in $j$, in the sense that, for a morphism $\kappa \colon j \to j'$ in $\cat J$, the diagram
  \[
  \begin{tikzcd}[column sep = 35]
  \colim[s]{f \slice j} \Res{\pr[f \slice j]} \rar{(f \slice \kappa)_*} \dar{\eq}[swap]{\vartheta} & \colim[s]{f \slice j'} \Res{\pr[f \slice j']} \dar{\vartheta}[swap]{\eq} \\
  \Res{j} \Lan{f} \rar & \Res{j'} \Lan{f}
  \end{tikzcd}
  \]
  commutes up to homotopy.
\end{lemma}

\begin{proof}
  That $\vartheta$ is an equivalence is a special case of \Cref{lemma:comma_square_mate}.
  For the naturality in $j$ we consider, for a map $\kappa \colon j \to j'$, the two diagrams
  \[
  \begin{tikzcd}[sep = 25]
  f \slice j \rar{f \slice \kappa} \dar & f \slice j' \dar \rar{\pr} \dlar[Rightarrow, shorten < = 13, shorten > = 13, swap]{\id} & \cat I \dar{f} \dlar[Rightarrow, shorten < = 13, shorten > = 13, swap]{\tilde \rho_{j'}} & & f \slice j \rar \dar[swap]{\pr} & \termCat \dar{j} \rar & \termCat \dar{j'} \\
  \termCat \rar & \termCat \rar{j'} & \cat J & & \cat I \rar{f} \urar[Rightarrow, shorten < = 13, shorten > = 13]{\tilde \rho_j} & \cat J \rar{\id} \urar[Rightarrow, shorten < = 10, shorten > = 10]{\kappa} & \cat J
  \end{tikzcd}
  \]
  for which we note that ${\id} \paste \tilde \rho_{j'} = \kappa \paste \tilde \rho_j$ by definition of the involved maps.
  Hence, after applying $\Fun{\blank}{\infcat C}$, we obtain, by the pasting laws for mates, that $\mate{(\rho_{j'})} \paste \mate \id = \mate {({\id} \paste \rho_{j'})} = \mate {(\kappa \paste \rho_j)} = \mate \kappa \paste \mate {(\rho_j)}$ (in the homotopy 2-category of $\infty$-categories).
  This is the statement we wanted to show since $\mate \kappa$ is the map $\Res{j} \to \Res{j'}$ given by evaluation at $\kappa$ and $\mate \id$ is the map on colimits induced by $f \slice \kappa$.
\end{proof}

\begin{lemma} \label{lemma:kan_counit}
  Let $f \colon \cat I \to \cat J$ be a functor between categories, $\infcat C$ a left $f$-extensible $\infty$-category, and $D \colon \cat J \to \infcat C$ a diagram.
  Let $j \in \cat J$ and note that the projection $\pr \colon f \slice j \to \cat I$ can be extended over the inclusion $\inc \colon f \slice j \to \cocone{(f \slice j)}$ to a map $\pr' \colon \cocone{(f \slice j)} \to \cat I$ by sending $\coconept$ to $j$ and the unique map $(i, k \colon f(i) \to j) \to \coconept$ to the map $k$.
  Then the following diagram commutes up to homotopy:
  \[
  \begin{tikzcd}
    \colim[s]{f \slice j} \Res{\pr} \Res{f} \rar{\eq}[swap]{\vartheta} \dar[equal] & \Res{j} \Lan{f} \Res{f} \ar{dd}{\epsilon_f} \\[-12]
    \colim[s]{f \slice j} \Res{\inc} \Res{\pr'} \Res{f} \dar & \\
    \Res{\coconept} \Res{\pr'} \Res{f} \rar[equal] & \Res{j}
  \end{tikzcd}
  \]
  where $\epsilon_f$ is the counit of the adjunction $\Ext[f] \dashv \Res{f}$, $\vartheta$ is as in \Cref{lemma:kan_local}, and the left vertical morphism is the canonical map out of the colimit.
\end{lemma}

\begin{proof}
  Consider the diagram
  \[
  \begin{tikzcd}[sep = 30]
    f \slice j \rar{\inc} \dar & \cocone{(f \slice j)} \rar{\pr'} \dar[swap]{\id} \dlar[Rightarrow, shorten < = 15, shorten > = 15, swap]{\xi} & \cat I \dar{f} \dlar[Rightarrow, shorten < = 12, shorten > = 12, swap]{\id} \\
    \termCat \rar{\coconept} & \cocone{(f \slice j)} \rar{f \circ \pr'} & \cat J
  \end{tikzcd}
  \]
  and note that applying $\Fun{\blank}{\infcat C}$ and taking mates yields, by the pasting law for mates, that the upper part of the diagram
  \[
  \begin{tikzcd}
    \colim[s]{f \slice j} \Res{\inc} \Res{\pr'} \Res{f} \rar[equal] \dar & \colim[s]{f \slice j} \Res{\pr} \Res{f} \drar[bend left = 12]{\vartheta} & \\
    \Res{\coconept} \Lan{\id} \Res{\pr'} \Res{f} \rar \dar{\id} & \Res{\coconept} \Res{f \circ \pr'} \Lan{f} \Res{f} \rar[equal] \dar[swap]{\epsilon_f} & \Res{j} \Lan{f} \Res{f} \dar{\epsilon_f} \\
    \Res{\coconept} \Lan{\id} \Res{\id} \Res{f \circ \pr'} \rar[equal] & \Res{\coconept} \Res{f \circ \pr'} \rar[equal] & \Res{j}
  \end{tikzcd}
  \]
  commutes up to homotopy.
  Since the lower left square in the above diagram commutes up to homotopy by \Cref{lemma:mate_and_units} this finishes the proof.
\end{proof}

\begin{lemma} \label{lemma:kan_mate}
  Let the diagram in the left be a diagram of categories, functors between them, and a natural transformation and the one in the right its image under $\Fun{\blank}{\infcat C}$
  \[
  \begin{tikzcd}[sep = 30]
    \cat I \rar{a} \dar[swap]{b} & \cat J \dar{d} \dlar[Rightarrow, shorten < = 13, shorten > = 13, swap]{\gamma} & & \Fun{\cat L}{\infcat C} \rar{\Res{d}} \dar[swap]{\Res{c}} & \Fun{\cat J}{\infcat C} \dar{\Res{a}} \dlar[Rightarrow, shorten < = 19, shorten > = 19, swap]{\gamma} \\
    \cat K \rar{c} & \cat L & & \Fun{\cat K}{\infcat C} \rar{\Res{b}} & \Fun{\cat I}{\infcat C}
  \end{tikzcd}
  \]
  where $\infcat C$ is a left $b$-extensible and left $d$-extensible $\infty$-category.
  Then, for any $k \in \cat K$, the following diagram commutes up to homotopy:
  \[
  \begin{tikzcd}
    \colim[s]{b \slice k} \Res{\pr[b \slice k]} \Res{a} \rar{f_*} \dar{\eq}[swap]{\vartheta} & \colim[s]{d \slice c(k)} \Res{\pr[d \slice c(k)]} \dar{\vartheta}[swap]{\eq} \\
    \Res{k} \Ext[b] \Res{a} \rar{\mate \gamma} & \Res{k} \Res{c} \Ext[d]
  \end{tikzcd}
  \]
  where $\mate \gamma$ is the mate of $\gamma$, the maps denoted $\vartheta$ are as in \Cref{lemma:kan_local}, and $f$ is the functor
  \[ b \slice k \longto d \slice c(k), \quad \left( i,\; b(i) \xto{g} k \right) \longmapsto \left( a(i),\; d(a(i)) \xto{\gamma} c(b(i)) \xto{c(g)} c(k) \right) \]
  acting on morphisms via $a$.
\end{lemma}

\begin{proof}
  Consider the two diagrams
  \[
  \begin{tikzcd}[sep = 30]
    b \slice k \rar{\pr[b \slice k]} \dar & \cat I \dar{b} \dlar[Rightarrow, shorten < = 15, shorten > = 15, swap]{\tilde \rho_k} \rar{a} & \cat J \dlar[Rightarrow, shorten < = 12, shorten > = 12, swap]{\gamma} \dar{d} & & b \slice k \rar{f} \dar & d \slice c(k) \dar \dlar[Rightarrow, shorten < = 17, shorten > = 17, swap]{\id} \rar{\pr[d \slice c(k)]} &[10] \cat J \dlar[Rightarrow, shorten < = 20, shorten > = 20, swap]{\tilde \rho_{c(k)}} \dar{d} \\
    \termCat \rar{k} & \cat K \rar{c} & \cat L & & \termCat \rar & \termCat \rar{c(k)} & \cat L
  \end{tikzcd}
  \]
  where $\tilde \rho$ is as in \Cref{lemma:kan_local}.
  Note that it follows directly from the definitions that their pastes $\tilde \rho_k \paste \gamma$ and ${\id} \paste \tilde \rho_{c(k)}$ are the same.
  Applying $\Fun{\blank}{\infcat C}$ and using the pasting law for mates yields the desired result.
\end{proof}

\subsection{(Co)Limits}

\begin{lemma} \label{lemma:colim_is_kan_to_cocone}
  Let $\infcat C$ be an $\infty$-category that admits colimits indexed by a category $\cat I$.
  Then the mate $\colim{\cat I} \to \Res{\coconept} \Lan{\inc}$ of the diagram on the right (which is the image of the diagram on the left under $\Fun{\blank}{\infcat C}$)
  \[
  \begin{tikzcd}[sep = 35]
  \cat I \rar{\id} \dar & \cat I \dar{\inc} \dlar[Rightarrow, shorten < = 17, shorten > = 17, swap]{\alpha} & & \Fun{\cocone{\cat I}}{\infcat C} \rar{\Res{\inc}} \dar[swap]{\Res{\coconept}} & \Fun{\cat I}{\infcat C} \dar{\id} \dlar[Rightarrow, shorten < = 27, shorten > = 27, swap]{\alpha} \\
  \termCat \rar{\coconept} & \cocone{\cat I} & & \infcat C \rar{\Diag} & \Fun{\cat I}{\infcat C}
  \end{tikzcd}
  \]
  is an equivalence.
\end{lemma}

\begin{proof}
  This is a special case of \Cref{lemma:kan_local} since there is an isomorphism ${\inc} \slice \coconept \iso \cat I$ over $\cocone{\cat I}$.
\end{proof}

\begin{lemma} \label{lemma:colim_of_kan}
  Let $f \colon \cat I \to \cat J$ be a fully faithful functor between categories and $\infcat C$ a left $f$-extensible $\infty$-category that admits colimits indexed both by $\cat I$ and by $\cat J$.
  Then the map
  \[ f_* \Lan f \colon \colim{\cat I} \Res{f} \Lan f \to \colim{\cat J} \Lan f \]
  of functors $\Fun{\cat I}{\infcat C} \to \infcat C$ is an equivalence.
\end{lemma}

\begin{proof}
  Taking mates of the two natural transformations in the diagram
  \[
  \begin{tikzcd}[sep = 35]
    \infcat C \rar{\Diag[\cat J]} \dar[swap]{\id} & \Fun{\cat J}{\infcat C} \rar{\Res{f}} \dar{\Res{f}} \dlar[Rightarrow, shorten < = 22, shorten > = 22, swap]{\id} & \Fun{\cat I}{\infcat C} \dar{\id} \dlar[Rightarrow, shorten < = 22, shorten > = 22, swap]{\id} \\
    \infcat C \rar{\Diag[\cat I]} & \Fun{\cat I}{\infcat C} \rar{\id} & \Fun{\cat I}{\infcat C}
  \end{tikzcd}
  \]
  and using the pasting law for mates yields that the composition
  \[ \colim{\cat I} \xlongto{\eta} \colim{\cat I} \Res{f} \Lan{f} \xlongto{f_*} \colim{\cat J} \Lan{f} \]
  is homotopic to $(\id[\cat I])_*$.
  Noting that $\eta$ is an equivalence since $f$ is fully faithful, this implies the desired statement.
\end{proof}

\lemmaCoconeComparison*

\begin{proof}
  The first statement follows from \cite[Proposition F.2.1 and Corollary 12.2.10]{RV}.
  The second from \cite[Proposition F.2.1, Lemma 2.3.6, and Lemma 2.3.7]{RV}.
\end{proof}

\lemmaHomotopyTerminalStuff*

\begin{proof}
  The first statement is clear (one can explicitly construct the contraction).
  The latter three statements follow, in order, from \cite[Theorem 4.1.3.1]{LurHTT}, \cite[Proposition 4.1.5]{RV}, and \cite[Proposition 4.1.1.3]{LurHTT}.
\end{proof}

\begin{lemma} \label{lemma:htpy_terminal_admitting}
  Let $f \colon \cat I \to \cat J$ be a homotopy terminal functor between categories and $\infcat C$ an $\infty$-category.
  Then $\infcat C$ admits colimits indexed by $\cat I$ if and only if it admits colimits indexed by $\cat J$.
\end{lemma}

\begin{proof}
  This follows from \Cref{lemma:cocone_comparison} and \cite[Proposition 4.1.1.8 (2)]{LurHTT}.
\end{proof}

\lemmaHomotopyTerminal*

\begin{proof}
  By \Cref{lemma:kan_mate}, the mate $\mate \id \colon \Lan{\inc[\cat I]} \Res{f} \to \Res{\cocone f} \Lan{\inc[\cat J]}$ of the natural transformation in the diagram on the right (which is the image of the diagram on the left under $\Fun{\blank}{\infcat C}$)
  \[
  \begin{tikzcd}[sep = 30]
  \cat I \rar{f} \dar[swap]{\inc} & \cat J \dar{\inc} \dlar[Rightarrow, shorten < = 13, shorten > = 13, swap]{\id} & & \Fun{\cocone{\cat J}}{\infcat C} \rar{\Res{\inc}} \dar[swap]{\Res{\cocone f}} & \Fun{\cat J}{\infcat C} \dar{\Res{f}} \dlar[Rightarrow, shorten < = 19, shorten > = 19, swap]{\id} \\
  \cocone{\cat I} \rar{\cocone f} & \cocone{\cat J} & & \Fun{\cocone{\cat I}}{\infcat C} \rar{\Res{\inc}} & \Fun{\cat I}{\infcat C}
  \end{tikzcd}
  \]
  is given, at the cocone point, by $f_*$.
  Hence it is enough to prove that $\mate \id$ is an equivalence.
  This mate is given by the composition
  \[
  \begin{tikzcd}[row sep = 7]
  \Lan{\inc[\cat I]} \Res{f} \rar{\eta} & \Lan{\inc[\cat I]} \Res{f} \Res{\inc[\cat J]} \Lan{\inc[\cat J]} \dar[equal] & \\
  & \Lan{\inc[\cat I]} \Res{\inc[\cat I]} \Res{\cocone f} \Lan{\inc[\cat J]} \rar{\epsilon} & \Res{\cocone f} \Lan{\inc[\cat J]}
  \end{tikzcd}
  \]
  of the unit $\eta$ of the adjunction $\Lan{\inc[\cat J]} \dashv \Res{\inc[\cat J]}$ and the counit $\epsilon$ of the adjunction $\Lan{\inc[\cat I]} \dashv \Res{\inc[\cat I]}$.
  Since $\inc[\cat J]$ is fully faithful, the map $\eta$ is an equivalence.
  So we only need to show that $\epsilon \circ \Res{\cocone f} \circ \Lan{\inc[\cat J]}$ is an equivalence.
  Let $D \colon \cat I \to \infcat C$ be a diagram.
  By \Cref{lemma:cocone_comparison} and assumption the diagram $(\Res{\cocone f} \circ \Lan{\inc[\cat J]}) (D)$ is a colimit diagram.
  But $\epsilon$ applied to a colimit diagram is an equivalence by \Cref{lemma:extensions_are_cocartesian} and again \Cref{lemma:cocone_comparison}.
\end{proof}

\begin{lemma} \label{lemma:functors_and_map_from_colim}
  Let $f \colon \cat I \to \cat J$ be a functor between categories and $\infcat C$ an $\infty$-category that admits colimits indexed both by $\cat I$ and by $\cat J$.
  Then the following diagram in $\Fun{\Fun{\cocone{\cat J}}{\infcat C}}{\infcat C}$ commutes up to homotopy:
  \[
  \begin{tikzcd}
    \colim[s]{\cat I} \Res{f} \Res{\inc[\cat J]} \rar{f_*} \dar[equal] & \colim[s]{\cat J} \Res{\inc[\cat J]} \ar{dd} \\[-12]
    \colim[s]{\cat I} \Res{\inc[\cat I]} \Res{\cocone{f}} \dar & \\
    \Res{\coconept} \Res{\cocone{f}} \rar[equal] & \Res{\coconept}
  \end{tikzcd}
  \]
  where the vertical maps are the canonical maps out of the colimit.
\end{lemma}

\begin{proof}
  Consider the two diagrams
  \[
  \begin{tikzcd}[sep = 30]
  \cat I \rar{\inc} \dar & \cocone{\cat I} \dar[swap]{\id} \dlar[Rightarrow, shorten < = 15, shorten > = 15, swap]{\xi_{\cat I}} \rar{\cocone{f}} & \cocone{\cat J} \dlar[Rightarrow, shorten < = 13, shorten > = 13, swap]{\id[1]} \dar{\id} & & \cat I \rar{f} \dar & \cat J \dar \dlar[Rightarrow, shorten < = 15, shorten > = 15, swap]{\id[2]} \rar{\inc} & \cocone{\cat J} \dlar[Rightarrow, shorten < = 15, shorten > = 15, swap]{\xi_{\cat J}} \dar{\id} \\
  \termCat \rar{\coconept} & \cocone{\cat I} \rar{\cocone{f}} & \cocone{\cat J} & & \termCat \rar & \termCat \rar{\coconept} & \cocone{\cat J}
  \end{tikzcd}
  \]
  where $\xi_{\cat I}$ and $\xi_{\cat J}$ are as in \Cref{def:can_map}, and note that their pastes agree, i.e.\ $\xi_{\cat I} \paste \id[1] = {\id[2]} \paste \xi_{\cat J}$.
  Applying $\Fun{\blank}{\infcat C}$ and using the pasting law for mates yields $\mate{(\id[1])} \paste \mate{(\xi_{\cat I})} \eq \mate{(\xi_{\cat J})} \paste \mate{(\id[2])}$.
  This is what we wanted to show since $\mate{(\id[2])} = f_*$ and $\mate{(\id[1])}$ is the identity.
\end{proof}

\begin{lemma} \label{lemma:terminal_object_colimit}
  Let $\cat I$ be a category with a terminal object $\term$, $\infcat C$ an $\infty$-category, and $D \colon \cocone{\cat I} \to \infcat C$ a diagram such that $D$ applied to the unique morphism $\term \to \coconept$ is an equivalence.
  Then the canonical morphism $\colim{\cat I} \restrict{D}{\cat I} \to D(\coconept)$ is an equivalence as well.
\end{lemma}

\begin{proof}
  First note that $\infcat C$ admits colimits indexed by $\cat I$ by \Cref{lemma:htpy_terminal_admitting}.
  Applying \Cref{lemma:functors_and_map_from_colim} to the functor $\const[\term] \colon \termCat \to \cat I$, we obtain a diagram
  \[
  \begin{tikzcd}
    \colim[s]{\termCat} D(\term) \rar{\eq} \dar[swap]{\eq} & \colim[s]{\cat I} \restrict D {\cat I} \dar \\
    D(\coconept) \rar[equal] & D(\coconept)
  \end{tikzcd}
  \]
  where the top horizontal morphism is an equivalence since $\const[\term]$ is homotopy terminal, and the left vertical morphism is an equivalence by \Cref{rem:one_point_colimit}.
\end{proof}

\begin{lemma} \label{lemma:can_is_counit_of_kan}
  Let $\cat I$ be a category, $\infcat C$ an $\infty$-category that admits colimits indexed by $\cat I$, and $D \colon \cocone{\cat I} \to \infcat C$ a diagram.
  Then the canonical map $(\colim{\cat I} \Res{\inc})(D) \to \Res{\coconept} (D)$ is an equivalence if and only if the counit $\Lan{\inc} \Res{\inc} \to \id$ of the adjunction $\Lan{\inc} \dashv \Res{\inc}$ is an equivalence at $D$.
\end{lemma}

\begin{proof}
  The mate of the natural transformation on the right (which is the image under $\Fun{\blank}{\infcat C}$ of the natural transformation on the left)
  \[
  \begin{tikzcd}[sep = 30]
  \cat I \rar{\inc} \dar[swap]{\inc} & \cocone{\cat I} \dar{\id} \dlar[Rightarrow, shorten < = 13, shorten > = 13, swap]{\id} & & \Fun{\cocone{\cat I}}{\infcat C} \rar{\id} \dar[swap]{\id} & \Fun{\cocone{\cat I}}{\infcat C} \dar{\Res{\inc}} \dlar[Rightarrow, shorten < = 19, shorten > = 19, swap]{\id} \\
  \cocone{\cat I} \rar{\id} & \cocone{\cat I} & & \Fun{\cocone{\cat I}}{\infcat C} \rar{\Res{\inc}} & \Fun{\cat I}{\infcat C}
  \end{tikzcd}
  \]
  is precisely the counit $\Lan{\inc} \Res{\inc} \to \id$.
  Hence, by \Cref{lemma:kan_mate}, it is an equivalence at $D$ if and only if, for all $k \in \cocone{\cat I}$, the map
  \[ \colim[s]{{\inc} \slice k} \Res{\pr[{\inc} \slice k]} \Res{\inc} = \colim[s]{{\inc} \slice k} \Res{f_k} \Res{\pr[{\id} \slice k]} \xlongto{(f_k)_*} \colim[s]{{\id} \slice k} \Res{\pr[{\id} \slice k]} \]
  is an equivalence at $D$, where $f_k \colon {\inc} \slice k \to {\id} \slice k$ is the canonical inclusion.
  When $k$ is not $\coconept$, then $f_k$ is an isomorphism and $(f_k)_*$ is an equivalence.
  When $k$ is $\coconept$, then $f_k$ is just (isomorphic to) $\inc$, and $\pr[{\id} \slice k]$ is an isomorphism.
  So it is enough to show that ${\inc}_* \colon \colim{\cat I} \Res{\inc} \to \colim{\cocone{\cat I}}$ is an equivalence at $D$ if and only if the canonical map $(\colim{\cat I} \Res{\inc})(D) \to \Res{\coconept} (D)$ is an equivalence.
  This follows from \Cref{lemma:functors_and_map_from_colim} by considering the diagram $D' \colon \cocone{(\cocone{\cat I})} \to \infcat C$ obtained from $D$ by pulling back along the functor $\cocone{(\cocone{\cat I})} \to \cocone{\cat I}$ that is the identity on $\cocone{\cat I}$ and sends the new cocone point to the old one (here we also use \Cref{lemma:terminal_object_colimit} to see that canonical map $(\colim{\cocone{\cat I}} \Res{\inc[\cocone{\cat I}]}) (D') \to \Res{\coconept} (D')$ is an equivalence).
\end{proof}

\begin{lemma} \label{lemma:limits_and_currying}
  Let $I$ and $J$ be simplicial sets, $\infcat C$ an $\infty$-category that admits colimits indexed by $J$, and $f \colon I \to \Fun{\cocone{J}}{\infcat C}$ a functor.
  Denote by $g \colon \cocone{J} \to \Fun{I}{\infcat C}$ the functor obtained from $f$ via currying.
  Then there is a homotopy commutative diagram of the form
  \begin{equation} \label[diagram]{diag:limits_and_currying}
  \begin{tikzcd}
    {\Res{\coconept}} \circ f \dar[equal] & {\colim[s]{J}} \circ {\Res{\inc}} \circ f \dar{\eq} \lar \\
    \Res{\coconept} g & \colim[s]{J} (\Res{\inc} g) \lar
  \end{tikzcd}
  \end{equation}
  where the horizontal morphisms are the canonical maps from the colimit.
\end{lemma}

\begin{proof}
  Consider the two diagrams
  \[
  \begin{tikzcd}[row sep = 30]
    \Fun{\cocone{J}}{\Fun{I}{\infcat C}} \rar{\id} \dar[swap]{\Res{\coconept}} & \Fun{\cocone{J}}{\Fun{I}{\infcat C}} \dar{\Res{\inc}} \dlar[Rightarrow, shorten < = 30, shorten > = 25, swap]{\id[1]} \\
    \Fun{I}{\infcat C} \rar{\Diag} \dar[swap]{\id} & \Fun{J}{\Fun{I}{\infcat C}} \dar{\iso} \dlar[Rightarrow, shorten < = 30, shorten > = 25, swap]{\id[2]} \\
    \Fun{I}{\infcat C} \rar{\Diag \circ} & \Fun{I}{\Fun{J}{\infcat C}}
  \end{tikzcd}
  \]
  and
  \[
  \begin{tikzcd}[row sep = 30]
    \Fun{\cocone{J}}{\Fun{I}{\infcat C}} \rar{\id} \dar[swap]{\iso} & \Fun{\cocone{J}}{\Fun{I}{\infcat C}} \dar{\iso} \dlar[Rightarrow, shorten < = 30, shorten > = 25, swap]{\id[3]} \\
    \Fun{I}{\Fun{\cocone{J}}{\infcat C}} \rar{\id} \dar[swap]{\Res{\coconept} \circ} & \Fun{I}{\Fun{\cocone{J}}{\infcat C}} \dar{\Res{\inc} \circ} \dlar[Rightarrow, shorten < = 30, shorten > = 25, swap]{\id[4]} \\
    \Fun{I}{\infcat C} \rar{\Diag \circ} & \Fun{I}{\Fun{J}{\infcat C}}
  \end{tikzcd}
  \]
  and note that their pastes agree.
  Now the mate of $\id[1]$ gives the lower horizontal map in \cref{diag:limits_and_currying}, the mate of $\id[4]$ the upper horizontal map (where we use the adjunction $(\colim{J} \;\circ) \dashv (\Diag \;\circ)$ obtained from \Cref{lemma:composing_with_adjunction}), the mate of $\id[2]$ the right equivalence (using \Cref{lemma:mate_equiv}), and the mate of $\id[3]$ the left identity.
  An application of the pasting law for mates yields the desired statement.
\end{proof}

\begin{lemma} \label{lemma:contractible_colimit_over_equivalences}
  Let $\cat I$ be a contractible category, $\infcat C$ an $\infty$-category that admits colimits indexed by $\cat I$, and $D \colon \cat I \to \infcat C$ a diagram such that, for all morphisms $k$ of $\cat I$, the induced map $D(k)$ is an equivalence.
  Then, for all $i \in \cat I$, the structure map $D(i) \to \colim{\cat I} D$ is an equivalence.
\end{lemma}

\begin{proof}
  We will show that the composition
  \[ \Fun{\cat I}{\infcat C} \xlongto{\Lan{\inc}} \Fun{\cocone{\cat I}}{\infcat C} \xlongto{\Res{t_i}} \Fun{\Simplex{1}}{\infcat C} \]
  sends $D$ to an equivalence, where $t_i \colon \Simplex{1} \to \cocone{\cat I}$ is as in \Cref{def:structure_map}, i.e.\ the functor representing the unique morphism $i \to \coconept$.
  Note that $\Lan{\inc} (D)$ is a colimit diagram indexed by $\cocone{\cat I}$ that sends any morphism in $\cat I$ to an equivalence.
  Hence, by \cite[Proposition 4.3.1.12]{LurHTT} (together with \cite[Proposition 2.4.1.5]{LurHTT}), the diagram $\Lan{\inc} (D)$ sends every morphism of $\cocone{\cat I}$ to an equivalence, which implies the claim.
\end{proof}

\subsection{Preservation of Kan extensions and (co)limits}

\lemmaResPreservesKan*

\begin{proof}
  That $\Fun{\infcat L}{\infcat C}$ is weakly left $f$-extensible when $\infcat C$ is was part of \Cref{lemma:kan_and_currying}.
  This also implies the corresponding statement for left $f$-extensible since this was defined as certain colimits existing which in turn was defined via weakly $\const$-extensible.
  
  For the second part we want that the mate of the transformation $\id[2]$ in the diagram
  \[
  \begin{tikzcd}[row sep = 30]
    \Fun{\cat J \times \infcat L}{\infcat C} \rar{\Res{f \times \id}} \dar[swap]{\iso} \ar[start anchor = south west, end anchor = north west, bend right = 40]{ddd}[swap]{\Res{{\id} \times g}} & \Fun{\cat I \times \infcat L}{\infcat C} \dar{\iso} \ar[start anchor = south east, end anchor = north east, bend left = 40]{ddd}{\Res{{\id} \times g}} \dlar[Rightarrow, shorten < = 30, shorten > = 25, swap]{\id[1]} \\
    \Fun{\cat J}{\Fun{\infcat L}{\infcat C}} \rar{\Res{f}} \dar[swap]{\Res{g} \circ} & \Fun{\cat I}{\Fun{\infcat L}{\infcat C}} \dar{\Res{g} \circ} \dlar[Rightarrow, shorten < = 30, shorten > = 25, swap]{\id[2]} \\
    \Fun{\cat J}{\Fun{\infcat K}{\infcat D}} \rar{\Res{f}} \dar[swap]{\iso} & \Fun{\cat I}{\Fun{\infcat K}{\infcat D}} \dar{\iso} \dlar[Rightarrow, shorten < = 30, shorten > = 25, swap]{\id[3]} \\
    \Fun{\cat J \times \infcat K}{\infcat D} \rar{\Res{f \times \id}} & \Fun{\cat I \times \infcat K}{\infcat D}
  \end{tikzcd}
  \]
  is an equivalence.
  For this note that the mates of $\id[1]$ and $\id[3]$ are equivalences by \Cref{lemma:mate_equiv} and that the paste of all three transformations is just $\Fun{\blank}{\infcat C}$ applied to the transformation
  \[
  \begin{tikzcd}[row sep = 30]
  \cat I \times \infcat K \rar{{\id} \times g} \dar[swap]{f \times \id} & \cat I \times \infcat L \dar{f \times \id} \dlar[Rightarrow, shorten < = 15, shorten > = 15, swap]{\id} \\
  \cat J \times \infcat K \rar{{\id} \times g} & \cat J \times \infcat L
  \end{tikzcd}
  \]
  which is a so called exact square by \cite[Lemma 9.2.8]{RV}.
  Hence, again by the Beck-Chevalley condition \cite[Lemma 12.3.11]{RV}, the mate of this paste is an equivalence.
  Now the pasting law for mates implies that the mate of $\id[2]$ is an equivalence as we wanted to show.
\end{proof}

\begin{lemma} \label{lemma:pointwise_preservation}
  Let $\infcat K$ be an $\infty$-category, $f \colon \cat I \to \cat J$ a functor between categories, $\infcat C$ and $\infcat D$ two left $f$-extensible $\infty$-categories, and $F \colon \infcat C \to \Fun{\infcat K}{\infcat D}$ a functor.
  Then $F$ preserves left Kan extension along $f$ if and only if, for all $k \in \infcat K$, the functor ${\Res{k}} \circ F \colon \infcat C \to \infcat D$ preserves left Kan extension along $f$.
\end{lemma}

\begin{proof}
  Consider the diagram
  \[
  \begin{tikzcd}[row sep = 30]
  \Fun{\cat J}{\infcat C} \rar{\Res{f}} \dar[swap]{F \circ} & \Fun{\cat I}{\infcat C} \dar{F \circ} \dlar[Rightarrow, shorten < = 30, shorten > = 25, swap]{\id[1]} \\
  \Fun{\cat J}{\Fun{\infcat K}{\infcat D}} \rar{\Res{f}} \dar[swap]{\Res{k} \circ} & \Fun{\cat I}{\Fun{\infcat K}{\infcat D}} \dar{\Res{k} \circ} \dlar[Rightarrow, shorten < = 30, shorten > = 25, swap]{\id[2]} \\
  \Fun{\cat J}{\infcat D} \rar{\Res f} & \Fun{\cat I}{\infcat D}
  \end{tikzcd}
  \]
  and note that $\mate{(\id[1])} \paste \mate{(\id[2])} \eq \mate{(\id[2] \paste \id[1])}$ by the pasting law for mates.
  Since $\Res{k}$ preserves left Kan extension along $f$, the mate $\mate{(\id[2])}$ is an equivalence.
  Now, noting that $({\Res{k}}\; \circ) \circ \mate{(\id[1])}$ is the other part of the composition in the paste $\mate{(\id[1])} \paste \mate{(\id[2])}$, we obtain that $({\Res{k}}\; \circ) \circ \mate{(\id[1])}$ is an equivalence if and only if $\mate{(\id[2] \paste \id[1])}$ is an equivalence.
  As the former being true for all $k$ is equivalent to $F$ preserving left Kan extension along $f$, and the latter is the definition of ${\Res{k}} \circ F$ preserving left Kan extension along $f$, this implies the claim.
\end{proof}

\begin{lemma} \label{lemma:fun_preserving_limits}
  Let $\infcat I$ and $\infcat J$ be $\infty$-categories, $\cat K$ a category, and $\infcat C$ and $\infcat D$ two $\infty$-categories that admit colimits indexed by $\cat K$.
  \begin{enumerate}[label=\alph*)]
    \item Let $g \colon \infcat C \to \infcat D$ be a functor that preserves colimits indexed by $\cat K$.
    Then the induced functor $(g \;\circ) \colon \Fun{\infcat I}{\infcat C} \to \Fun{\infcat I}{\infcat D}$ preserves colimits indexed by $\cat K$.
    \item The functor $h \colon \Fun{\infcat J}{\infcat C} \to \Fun{\Fun{\infcat I}{\infcat J}}{\Fun{\infcat I}{\infcat C}}$ given by $f \mapsto (f \;\circ)$ preserves colimits indexed by $\cat K$.
  \end{enumerate}
\end{lemma}

\begin{proof}
  The first statement follows from \Cref{lemma:pointwise_preservation} since, for all $i \in \infcat I$, it holds that ${\Res{i}} \circ (g \;\circ) = g \circ \Res{i}$ and both functors in the latter composition preserve colimits indexed by $\cat K$.
  The second statement follows from the same lemma by noting that, for any $f \in \Fun{\infcat I}{\infcat J}$, the functor ${\Res{\set f}} \circ h = \Res{f}$ preserves colimits indexed by $\cat K$.
\end{proof}

\begin{lemma} \label{lemma:preserving_Kan_extensions}
  Let $f \colon \cat I \to \cat J$ be a functor between categories, $\infcat C$ and $\infcat D$ two left $f$-extensible $\infty$-categories, and $F \colon \infcat C \to \infcat D$ a functor that preserves colimits indexed by $f \slice j$ for all $j \in \cat J$.
  Then $F$ preserves left Kan extensions along $f$.
\end{lemma}

\begin{proof}
  Consider the two diagrams
  \[
  \begin{tikzcd}[row sep = 30]
  \Fun{\cat J}{\infcat C} \rar{\Res{f}} \dar[swap]{F \circ} & \Fun{\cat I}{\infcat C} \dar{F \circ} \dlar[Rightarrow, shorten < = 20, shorten > = 20, swap]{\id[1]} & & \Fun{\cat J}{\infcat C} \rar{\Res{f}} \dar[swap]{\Res{j}} & \Fun{\cat I}{\infcat C} \dar{\Res{\pr}} \dlar[Rightarrow, shorten < = 25, shorten > = 25, swap]{\rho_j} \\
  \Fun{\cat J}{\infcat D} \rar{\Res{f}} \dar[swap]{\Res{j}} & \Fun{\cat I}{\infcat D} \dar{\Res{\pr}} \dlar[Rightarrow, shorten < = 25, shorten > = 25, swap]{\rho_j} & & \infcat C \rar{\Diag} \dar[swap]{F} & \Fun{f \slice j}{\infcat C} \dar{F \circ} \dlar[Rightarrow, shorten < = 25, shorten > = 25, swap]{\id[2]} \\
  \infcat D \rar{\Diag} & \Fun{f \slice j}{\infcat D} & & \infcat D \rar{\Diag} & \Fun{f \slice j}{\infcat D}
  \end{tikzcd}
  \]
  (where $\rho_j$ is as in \Cref{lemma:kan_local}) and note that their pastes agree.
  We want to show that $\mate{(\id[1])} \colon {\Lan{f}} \circ (F \;\circ) \to (F \;\circ) \circ \Lan{f}$ is an equivalence.
  For this it is enough to show that ${\Res{j}} \circ \mate{(\id[1])}$ is an equivalence for all $j \in \cat J$.
  This is one of the transformations that is composed in the paste $\mate{(\id[1])} \paste \mate{(\rho_j)}$, which is homotopic to $\mate{(\id[2])} \paste \mate{(\rho_j)}$ by the pasting law for mates.
  Since $\mate{(\rho_j)}$ is an equivalence by \Cref{lemma:kan_local} and $\mate{(\id[2])}$ is one by assumption, this finishes the proof.
\end{proof}

\lemmaPreservation*

\begin{proof}
  By \Cref{lemma:preserving_Kan_extensions}, if the functor $F$ preserves colimits indexed by $\cat I$, then it also preserves left Kan extension along $\inc$.
  The proof of the same lemma also shows that if $F$ preserves left Kan extension along $\inc$, then it preserves the colimits of all diagrams $\cat I \to \infcat C$ that lie in the essential image of $\Res{\inc} \colon \Fun{\cocone{\cat I}}{\infcat C} \to \Fun{\cat I}{\infcat C}$.
  But, as $\inc$ is fully faithful, we have $\Res{\inc} \Lan{\inc} \eq \id$ and thus all diagrams lie in the essential image of $\Res{\inc}$.
  This shows the equivalence of the first two conditions.
  
  Now note that, by definition, the functor $F$ preserves left Kan extension along $\inc$ if and only if the natural transformation
  \[
  \begin{tikzcd}[row sep = 7]
    {\Lan{\inc}} \circ (F \;\circ) \rar{\eta} & {\Lan{\inc}} \circ (F \;\circ) \circ \Res{\inc} \circ \Lan{\inc} \dar[equal] & \\
     & {\Lan{\inc}} \circ {\Res{\inc}} \circ (F \;\circ) \circ \Lan{\inc} \rar{\epsilon} & (F \;\circ) \circ \Lan{\inc}
  \end{tikzcd}
  \]
  is an equivalence, where $\eta$ and $\epsilon$ are the unit respectively counit of the adjunction $\Lan{\inc} \dashv \Res{\inc}$.
  Since $\eta$ is an equivalence (as $\inc$ is fully faithful), this is equivalent to $\epsilon$ being an equivalence on any diagram in the essential image of $(F \;\circ) \circ \Lan{\inc}$.
  By \Cref{lemma:cocone_comparison,lemma:extensions_are_cocartesian}, this is equivalent to $F$ sending $\cocone{\cat I}$-indexed colimit diagrams to colimit diagrams.
\end{proof}

\begin{lemma} \label{lemma:limits_preserve_limits}
  Let $I$, $J$, and $K$ be simplicial sets, $f \colon I \to J$ a map, and $\infcat C$ a weakly left $f$-extensible $\infty$-category that admits colimits indexed by $K$.
  Then the functor ${\Lan{f}} \colon \Fun{I}{\infcat C} \to \Fun{J}{\infcat C}$ preserves colimits indexed by $K$.
\end{lemma}

\begin{proof}
  By \cite[Proposition 5.2.3.5]{LurHTT} left adjoints preserve colimits.
  Noting that $\Lan{f}$ is left adjoint, this implies the statement.
\end{proof}

\begin{lemma} \label{lemma:colim_commutes_lim}
  Let $\cat I$ and $\cat J$ be two categories and $\infcat C$ an $\infty$-category that admits colimits indexed by $\cat I$ and limits indexed by $\cat J$.
  Then ${\colim{\cat I}} \colon \Fun{\cat I}{\infcat C} \to \infcat C$ preserves limits indexed by $\cat J$ if and only if ${\lim{\cat J}} \colon \Fun{\cat  J}{\infcat C} \to \infcat C$ preserves colimits indexed by $\cat I$.
\end{lemma}

\begin{proof}
  We show that, if ${\colim{\cat I}}$ preserves limits indexed by $\cat J$, then ${\lim{\cat J}}$ preserves colimits indexed by $\cat I$.
  The other direction follows dually.
  
  By \Cref{lemma:preservation,lemma:pointwise_preservation,lemma:colim_is_kan_to_cocone}, our assumption implies that $\Lan{\inc[\cat I]}$ preserves right Kan extension along $\inc[\cat J]$.
  Hence there is an equivalence
  \[ ({\Lan{\inc[\cat I]}} \;\circ) \circ \Ran{\inc[\cat J]} \eq {\Ran{\inc[\cat J]}} \circ ({\Lan{\inc[\cat I]}} \;\circ) \]
  of functors $\Fun{\cat J}{\Fun{\cat I}{\infcat C}} \longto \Fun{\cone{\cat J}}{\Fun{\cocone{\cat I}}{\infcat C}}$.
  This transforms, through a few applications of \Cref{lemma:kan_and_currying}, to an equivalence
  \[ {\Lan{\inc[\cat I]}} \circ ({\Ran{\inc[\cat J]}} \;\circ) \eq ({\Ran{\inc[\cat J]}} \;\circ) \circ \Lan{\inc[\cat I]} \]
  of functors $\Fun{\cat I}{\Fun{\cat J}{\infcat C}} \longto \Fun{\cocone{\cat I}}{\Fun{\cone{\cat J}}{\infcat C}}$.
  This becomes, after postcomposing with $({\Res{\conept_{\cat J}}} \;\circ)$, an equivalence
  \[ {\Lan{\inc[\cat I]}} \circ ({\lim[s]{\cat J}} \;\circ) \eq ({\lim[s]{\cat J}} \;\circ) \circ \Lan{\inc[\cat I]} \]
  of functors $\Fun{\cat I}{\Fun{\cat J}{\infcat C}} \longto \Fun{\cocone{\cat I}}{\infcat C}$ (using \Cref{lemma:colim_is_kan_to_cocone} and that restrictions preserve left Kan extensions).
  Thus $\lim{\cat J}$ sends colimit diagrams indexed by $\cocone{\cat I}$ to colimit diagrams, as we wanted to show.
\end{proof}

\begin{lemma} \label{lemma:preserving_structure_map}
  Let $\cat I$ be a category, $i$ an object of $\cat I$, $\infcat C$ and $\infcat D$ two $\infty$-categories that admit colimits indexed by $\cat I$, $F \colon \infcat C \to \infcat D$ a functor that preserves colimits indexed by $\cat I$, and $D \colon \cat I \to \infcat C$ a diagram.
  Then $F$ applied to the structure map $D(i) \to \colim{\cat I} D$ is an equivalence if and only if the structure map $(F \circ D)(i) \to \colim{\cat I} (F \circ D)$ is.
\end{lemma}

\begin{proof}
  This follows from \Cref{lemma:preservation} and \Cref{rem:structure_maps}.
\end{proof}

\begin{lemma} \label{lemma:preservation_comp_unit}
  Let $f \colon I \to J$ be a map of simplicial sets and $F \colon \infcat C \to \infcat D$ a functor between weakly left $f$-extensible $\infty$-categories.
  Then the diagram
  \[
  \begin{tikzcd}
    (F \;\circ) \rar{\eta^{\infcat D}} \dar[swap]{\eta^{\infcat C}} & {\Res{f}} \circ {\Lan{f}} \circ (F \;\circ) \dar{\chi} \\
    (F \;\circ) \circ \Res{f} \circ \Lan{f} \rar[equal] & {\Res{f}} \circ (F \;\circ) \circ \Lan{f}
  \end{tikzcd}
  \]
  commutes up to homotopy, where $\chi$ is as in \Cref{def:preservation}, and $\eta^{\infcat C}$ and $\eta^{\infcat D}$ are the units of the adjunctions $\Lan{f} \dashv \Res{f}$ with the respective target categories.
\end{lemma}

\begin{proof}
  This is a special case of \Cref{lemma:mate_and_units}.
\end{proof}

\begin{lemma} \label{lemma:preservation_and_can}
  Let $\cat I$ be a category, $F \colon \infcat C \to \infcat D$ a functor between $\infty$-categories that admit limits indexed by $\cat I$, and denote by ${\inc} \colon \cat I \to \cone{\cat I}$ the inclusion.
  Then the following diagram in $\Fun{\Fun{\cone{\cat I}}{\infcat C}}{\infcat D}$ commutes up to homotopy:
  \[
  \begin{tikzcd}
    F \circ \Res{\conept} \rar \dar[equal] & F \circ \lim[s]{\cat I} \circ \Res{\inc} \rar{\chi} & \lim[s]{\cat I} \circ (F \;\circ) \circ \Res{\inc} \dar[equal] \\
    {\Res{\conept}} \circ (F \;\circ) \ar{rr} & & \lim[s]{\cat I} \circ {\Res{\inc}} \circ (F \;\circ)
  \end{tikzcd}
  \]
  where the upper left and the bottom horizontal morphism are given by the respective canonical map to the limit, and $\chi$ is as in \Cref{def:preservation}.
\end{lemma}

\begin{proof}
  Consider the two diagrams
  \[
  \begin{tikzcd}[row sep = 30]
    \Fun{\cone{\cat I}}{\infcat C} \rar{\id} \dar[swap]{\Res{\conept}} & \Fun{\cone{\cat I}}{\infcat C} \dar{\Res{\inc}} & & \Fun{\cocone{\cat I}}{\infcat C} \rar{\id} \dar[swap]{F \circ} & \Fun{\cone{\cat I}}{\infcat C} \dar{F \circ} \\
    \infcat C \rar{\Diag} \dar[swap]{F} \urar[Rightarrow, shorten < = 25, shorten > = 25]{\xi} & \Fun{\cat I}{\infcat C} \dar{F \circ} & & \Fun{\cone{\cat I}}{\infcat D} \rar{\id} \dar[swap]{\Res{\conept}} \urar[Rightarrow, shorten < = 20, shorten > = 20]{\id[2]} & \Fun{\cone{\cat I}}{\infcat D} \dar{\Res{\inc}} \\
    \infcat D \rar{\Diag} \urar[Rightarrow, shorten < = 25, shorten > = 25]{\id[1]} & \Fun{\cat I}{\infcat D} & & \infcat D \rar{\Diag} \urar[Rightarrow, shorten < = 25, shorten > = 25]{\xi} & \Fun{\cat I}{\infcat D}
  \end{tikzcd}
  \]
  where $\xi$ is as in \Cref{def:can_map}, and note that their pastes agree.
  Now the pasting law for mates implies the desired statement since the mates of the transformations labeled $\xi$ are the canonical maps to the limit, the mate of $\id[1]$ is $\chi$, and the mate of $\id[2]$ is the identity.
\end{proof}

\section{Generalities} \label{section:generalitites}

%

In this appendix we collect a number of general lemmas that we need throughout the rest of this paper.

\subsection{about posets}

\begin{lemma} \label{lemma:pos_full_implies_injective}
  Let $\pos I$ be a poset, $\cat C$ a category, and $f \colon \pos I \to \cat C$ a full functor.
  Then $f$ is injective.
\end{lemma}

\begin{proof}
  Assume that there exist $i \neq i' \in \pos I$ such that $f(i) = f(i')$.
  Since $f$ is full there must be both a map $i \to i'$ and a map $i' \to i$ being mapped to $\id[f(i)]$ by $f$.
  This contradicts the definition of a poset.
\end{proof}

\begin{lemma} \label{lemma:poset_full_composition}
  Let $f \colon \cat I \to \pos J$ and $g \colon \pos J \to \cat K$ be functors between categories such that $\pos J$ is a poset and $g \circ f$ is full.
  Then $f$ is full.
\end{lemma}

\begin{proof}
  This follows from the fact that, when a surjective map of sets factors over a set with at most one element, the first map in this factorization is also surjective.
\end{proof}

\begin{lemma} \label{lemma:contractible_products}
  Let $\pos I$ and $\pos J$ be posets.
  Assume that both $\pos I$ and $\pos J$ have initial objects $\ini_{\pos I}$ respectively $\ini_{\pos J}$ and that $\pos I$ has a terminal object $\term_{\pos I} \neq \ini_{\pos I}$.
  Then $\gini{(\pos I \times \pos J)}$ is contractible.
\end{lemma}

\begin{proof}
  Since $\pos J$ has an initial object and is thus contractible, it is enough to show that there is an adjoint pair of functors between $\gini{(\pos I \times \pos J)}$ and $\pos J$ as this implies that they are homotopy equivalent.
  
  To this end, let $l \colon \gini{(\pos I \times \pos J)} \to \pos J$ be given by the projection, i.e.\ $l(i, j) = j$, and $r \colon \pos J \to \gini{(\pos I \times \pos J)}$ by $r(j) = (\term_{\pos I}, j)$.
  Note that $r$ is well-defined as, by assumption, we have $\term_{\pos I} \neq \ini_{\pos I}$.
  To check that $l$ is indeed left adjoint to $r$, we need to prove that, for all $(i, j) \in \gini{(\pos I \times \pos J)}$ and $j' \in \pos J$, we have $j = l(i,j) \le j'$ if and only if $(i, j) \le r(j') = (\term_{\pos I}, j')$.
  This is true by the assumption of $\term_{\pos I}$ being terminal in $\pos I$.
\end{proof}

\begin{lemma} \label{lemma:gini_cop_htpy_initial}
  Let $f \colon \pos S \to \pos T$ and $f' \colon \pos S' \to \pos T$ be maps of posets where $\pos S, \pos S' \in \Posini$ and $\pos T \in \Poscop$.
  Assume that $\inv f(\ini_{\pos T}) = \set{\ini_{\pos S}}$ and $\inv{(f')}(\ini_{\pos T}) = \set{\ini_{\pos S'}}$ and that for all $t \in \ginipos T$ one of the posets $f \slice t$ and $f' \slice t$ has a terminal object which is different from the initial object (in particular this is fulfilled if $f = \id[\pos T]$).
  Then $p \colon \gini{(\pos S \times \pos S')} \to \gini{\pos T}$ given by $(s, s') \mapsto f(s) \cop f'(s')$ is homotopy initial.
\end{lemma}

\begin{proof}
  We need that, for all $t \in \gini{\pos T}$, the category $p \slice t$ is contractible.
  This comma category can be identified with the full subposet
  \[ \set{(s, s') \in \gini{(\pos S \times \pos S')} \mid f(s) \le t \text{ and } f'(s') \le t } \subseteq \gini{(\pos S \times \pos S')} \]
  using the universal property of the coproduct.
  This, in turn, is isomorphic to the category $\gini{((f \slice t) \times (f' \slice t))}$ which is contractible by \Cref{lemma:contractible_products}.
  Here, we use that, by our assumptions both $f \slice t$ and $f' \slice t$ have an initial object ($\ini_{\pos S}$ respectively $\ini_{\pos S'}$) and one of them has a terminal object different from the initial object.
\end{proof}

\subsection{\texorpdfstring{about $\infty$-categories}{about infty-categories}}

%

\begin{lemma} \label{lemma:composing_with_adjunction}
  Let $l \colon \infcat C \to \infcat D$ and $r \colon \infcat D \to \infcat C$ be two functors between $\infty$-categories such that $l$ is left adjoint to $r$ with unit $\eta$ and counit $\epsilon$.
  Then, for any simplicial set $K$, the functor $(l \;\circ) \colon \Fun{K}{\infcat C} \to \Fun{K}{\infcat D}$ is left adjoint to $(r \;\circ) \colon \Fun{K}{\infcat D} \to \Fun{K}{\infcat C}$ with unit $(\eta \;\circ)$ and counit $(\epsilon \;\circ)$.
\end{lemma}

\begin{proof}
  This is \cite[Proposition 2.1.7 (iii)]{RV}.
\end{proof}

\begin{lemma} \label{lemma:pullbacks_of_inftycats}
  Let the following be a pullback square in the 1-category of simplicial sets:
  \[
  \begin{tikzcd}
    K \rar \dar[swap]{f} & \infcat D \dar{g} \\
    \infcat E \rar & \cat C
  \end{tikzcd}
  \]
  where $\infcat D$ and $\infcat E$ are $\infty$-categories and $\cat C$ is a category.
  Then $K$ is an $\infty$-category.
\end{lemma}

\begin{proof}
  By \cite[Proposition 2.3.1.5]{LurHTT}, the functor $g$ is an inner fibration.
  But then $f$ is also an inner fibration since they are stable under pullbacks.
  Since $\infcat E$ is an $\infty$-category, the constant map $c \colon \infcat E \to \termCat$ is also an inner fibration.
  Since inner fibrations are closed under composition, the constant map $K \to \termCat$ is thus also an inner fibration and hence $K$ an $\infty$-category.
\end{proof}

\begin{lemma} \label{lemma:generalized_subcat}
  Let $\infcat C$ be an $\infty$-category, $\cat D$ a category, and $f \colon \cat D \to \hcat{\infcat C}$ a functor.
  Furthermore, let $\infcat E$ be a pullback (in the 1-category of simplicial sets) as in the diagram
  \[
  \begin{tikzcd}
    \infcat E \rar{g} \dar[swap]{p} & \infcat C \dar{\hcatmap[\infcat C]} \\
    \cat D \rar{f} & \hcat{\infcat C}
  \end{tikzcd}
  \]
  where $\hcatmap[\infcat C]$ denotes the canonical functor to the homotopy category.
  Then $\infcat E$ is an $\infty$-category, there is a unique isomorphism $\cat D \iso \hcat{\infcat E}$ under $\infcat E$, and, for two objects $E, E' \in \infcat E$ and morphism $d \colon p(E) \to p(E')$ in $\cat D$, the functor $g$ induces an equivalence from the path component of $\Map[\infcat E]{E}{E'}$ over $d$ (there is only one such component by the identification $\cat D \iso \hcat{\infcat E}$) to the path component of $\Map[\infcat C]{g(E)}{g(E')}$ over $f(d)$.
\end{lemma}

\begin{proof}
  That $\infcat E$ is an $\infty$-category follows directly from \Cref{lemma:pullbacks_of_inftycats}.
  Furthermore note that, since $\hcatmap[\infcat C]$ is a bijection on objects, the map $p$ is as well, i.e.\ we can identify objects of $\infcat E$ with objects of $\cat D$.
  Now, by the universal property of the pullback the functor $g$ induces, for any morphism $d \colon D \to D'$ in $\infcat D$, an isomorphism from the simplicial subset of $\HomR[\infcat E]{D}{D'}$ lying over $d$ to the simplicial subset of $\HomR[\infcat C]{f(D)}{f(D')}$ lying over $f(d)$ (cf.\ \cite[Section 1.2.2]{LurHTT} for the definition of $\HomRwo$).
  This shows the last statement.
  To obtain the identification $\cat D \iso \hcat{\infcat E}$, note that what we have already shown implies that the part of $\Map[\infcat E]{D}{D'}$ lying over $d$ is path-connected and that these parts are, for different morphisms in $\cat D$, disjoint path-components that cover the whole space.
\end{proof}

\begin{lemma} \label{lemma:hcatmap_categorical_fibration}
  Let $\infcat C$ be an $\infty$-category.
  Then the canonical map $\hcatmap[\infcat C] \colon \infcat C \to \hcat{\infcat C}$ to its homotopy category is a categorical fibration.
\end{lemma}

\begin{proof}
  By \cite[Corollary 2.4.6.5]{LurHTT} the statement is equivalent to $\hcatmap[\infcat C]$ being an inner fibration such that for every equivalence $f \colon D \to D'$ in $\hcat{\infcat C}$ and $C \in \infcat C$ with $\hcatmap[\infcat C](C) = D$ there exists an equivalence $g \colon C \to C'$ in $\infcat C$ such that $\hcatmap[\infcat C](g) = f$.
  That it is an inner fibration follows directly from \cite[Proposition 2.3.1.5]{LurHTT} and the other property is clear from the definition of the homotopy category.
\end{proof}

\begin{lemma} \label{lemma:sequential_diagrams}
  Denote by $S$ the simplicial set obtained from the directed graph with vertices $\NN$ and an edge $n \to n+1$ for every $n \in \NN$, and by $i \colon S \to \NN$ the canonical inclusion of simplicial sets.
  Then, for every $\infty$-category $\infcat C$, the restriction $\Res{i} \colon \Fun{\NN}{\infcat C} \to \Fun{S}{\infcat C}$ is a trivial Kan fibration.
  
  In particular, for every functor $f \colon S \to \infcat C$, there is an essentially unique functor $g \colon \NN \to \infcat C$ such that $g \circ i = f$.
  Less formally: to specify a sequential diagram in $\infcat C$ it is enough to specify a sequence of composable morphisms $(D_n \to D_{n+1})_{n \in \NN}$ in $\infcat C$.
\end{lemma}

\begin{proof}
  This is \cite[\href{https://kerodon.net/tag/00J4}{Theorem 00J4}]{LurK} applied to the directed graph used to define $S$.
\end{proof}

\begin{lemma} \label{lemma:unit_classical}
  Let $L \colon \infcat C \to \infcat D$ and $R \colon \infcat D \to \infcat C$ be two functors between $\infty$-categories such that there is an adjunction $L \dashv R$ with unit $\eta \colon \id[\infcat C] \to R \circ L$.
  Then, for all $c \in \infcat C$ and $d \in \infcat D$, the composition
  \[ \Map[\infcat D]{L(c)}{d} \xlongto{R} \Map[\infcat C]{R L(c)}{R (d)} \xlongto{\circ \eta(c)} \Map[\infcat C]{c}{R (d)} \]
  is an equivalence.
\end{lemma}

\begin{proof}
  We claim that $(\epsilon(d) \;\circ) \circ L$ is a quasi-inverse, where $\epsilon \colon L \circ R \to \id[\infcat D]$ is the counit of the above adjunction $L \dashv R$.
  To see that it is a left inverse consider the homotopy commutative diagram
  \[
  \begin{tikzcd}
    \Map[\infcat D]{L(c)}{d} \dar[swap]{R} \drar[start anchor = south east]{LR} \ar[start anchor = east, bend left = 10]{drr}{\circ \epsilon(L(c))} & & \\
    \Map[\infcat C]{R L(c)}{R (d)} \rar{L} \dar[swap]{\circ \eta(c)} & \Map[\infcat D]{L R L(c)}{L R (d)} \dar[swap]{\circ L(\eta(c))} \rar{\epsilon(d) \circ} & \Map[\infcat D]{L R L(c)}{d} \dar{\circ L(\eta(c))} \\
    \Map[\infcat C]{c}{R (d)} \rar{L} & \Map[\infcat D]{L(c)}{L R (d)} \rar{\epsilon(d) \circ} & \Map[\infcat D]{L(c)}{d}
  \end{tikzcd}
  \]
  and note that the composition along the right side of the diagram is homotopic to the identity by one of the triangle identities.
  Analogously one can show that it is also a right inverse.
\end{proof}

\subsection{about cartesian diagrams}

\begin{lemma} \label{lemma:preserves_cartesian}
  Let $\cat I$ and $\cat J$ be categories that have initial objects, $f \colon \cat I \to \cat J$ an initial object preserving functor, and $\infcat C$ an $\infty$-category that admits limits indexed by both $\gini{\cat I}$ and by $\gini{\cat J}$.
  Furthermore, assume that $f$ restricts to a functor $\gini{\cat I} \to \gini{\cat J}$ and that this functor is homotopy initial.
  
  Then a diagram $D \colon \cat J \to \infcat C$ is cartesian if and only if $D \circ f \colon \cat I \to \infcat C$ is cartesian.
\end{lemma}

\begin{proof}
  We have, by (the dual of) \Cref{lemma:functors_and_map_from_colim}, a homotopy commutative diagram
  \[
  \begin{tikzcd}
  D(\ini_{\cat J}) \rar[equal] \dar & (D \circ f)(\ini_{\cat I}) \dar \\
  \lim{\gini{\cat J}} \restrict{D}{\gini{\cat J}} \rar{\eq}[swap]{f^*} & \lim{\gini{\cat I}} \restrict{(D \circ f)}{\gini{\cat I}}
  \end{tikzcd}
  \]
  in which the bottom horizontal map is an equivalence by assumption.
  Hence, the left vertical map is an equivalence if and only if the right vertical map is an equivalence, as we wanted to show.
\end{proof}

\begin{lemma} \label{lemma:pointwise_cartesian}
  Let $\cat I$ and $\cat J$ be categories such that $\cat J$ has an initial object, $\infcat C$ an $\infty$-category that admits limits indexed by $\gini{\cat J}$, and $D \colon \cat I \times \cat J \to \infcat C$ a functor.
  Denote by $D_{\cat I} \colon \cat I \to \Fun{\cat J}{\infcat C}$ and $D_{\cat J} \colon \cat J \to \Fun{\cat I}{\infcat C}$ the curried functors.
  Furthermore assume that $D_{\cat I}(i) = {\Res{i}} \circ D_{\cat J} \colon \cat J \to \infcat C$ is cartesian for all $i \in \cat I$.
  \begin{enumerate}[label=\alph*)]
    \item If $\infcat C$ admits limits indexed by $\cat I$, then $D_{\cat J}$, $\lim{\cat I} \circ D_{\cat J}$, and $\lim{\cat I} D_{\cat I}$ are all cartesian.
    \item If $\infcat C$ admits colimits indexed by $\cat I$ and the functor ${\colim{\cat I}} \colon \Fun{\cat I}{\infcat C} \to \infcat C$ preserves limits indexed by $\gini{\cat J}$, then $\colim{\cat I} \circ D_{\cat J}$ and $\colim{\cat I} D_{\cat I}$ are cartesian.
  \end{enumerate}
\end{lemma}

\begin{proof}
  We first show that $D_{\cat J}$ is cartesian, i.e.\ that the canonical map $D_{\cat J}(\ini) \to (\lim{\gini{\cat J}} \Res{\gini{\cat J}})(D_{\cat J})$ is an equivalence.
  For this it is enough that its restriction to $i$ is an equivalence for all $i \in \cat I$, which follows by assumption and \Cref{lemma:preservation_and_can} since $\Res{i}$ preserves limits.
  
  Now, if $\infcat C$ admits limits indexed by $\cat I$, then $\lim{\cat I} \circ D_{\cat J}$ is also cartesian by again \Cref{lemma:preservation_and_can} since $\lim{\cat I}$ preserves limits by \Cref{lemma:limits_preserve_limits}.
  This also implies that $\lim{\cat I} D_{\cat I}$ is cartesian by \Cref{lemma:colimits_and_currying}.
  The statement about colimits can be shown in the same way since we only used that $\lim{\cat I}$ preserves limits indexed by $\gini{\cat J}$.
\end{proof}

\begin{lemma} \label{lemma:cartesian_products}
  Let $\pos I$ and $\pos J$ be categories with initial objects and $\infcat C$ an $\infty$-category that admits limits indexed by $\gini{\cat J}$.
  Furthermore, let $D \colon \pos I \times \pos J \to \infcat C$ be a diagram such that, for each $i \in \pos I$, the restriction $\restrict D {\set i \times \cat J} \colon \cat J \to \infcat C$ is cartesian.
  Then $D$ is a limit diagram.
\end{lemma}

\begin{proof}
  We consider the inclusions
  \[ \pos I \times (\ginipos J) \xlongto{\iota} \gini{(\pos I \times \pos J)} \xlongto{\kappa} \pos I \times \pos J. \]
  By assumption and \cite[Proposition~4.3.2.9]{LurHTT}, the functor $D$ is a right Kan extension of $\Res{\kappa \circ \iota}$ along $\kappa \circ \iota$ in the sense of \cite[Definition~4.3.2.2]{LurHTT}.
  In the same way we also obtain that $\Res{\kappa} (D)$ is a right Kan extension of $\Res{\kappa \circ \iota} (D)$ along $\iota$.
  Then, by \cite[Proposition~4.3.2.8]{LurHTT}, the diagram $D$ is a right Kan extension of $\Res{\kappa} (D)$ along $\kappa$, i.e.\ $D$ is a limit diagram.
\end{proof}

\end{appendices}

\printbibliography

\end{document}